\documentclass[reqno]{amsart}            
\usepackage{amsthm,amssymb,amsfonts,graphicx,cite,color}
\usepackage[bookmarks,bookmarksnumbered,bookmarksopen,colorlinks,backref,linkcolor=blue,citecolor=red]{hyperref}%
\usepackage{xcolor}
\usepackage{graphicx}
\usepackage{mathrsfs}
\usepackage{bm}

\newtheorem{theorem}{Theorem}[section]
\newtheorem{definition}{Definition}[section]

\newtheorem{corollary}{Corollary}[section]
\newtheorem{remark}{Remark}[section]
\newtheorem{proposition}{Proposition}[section]
\numberwithin{equation}{section}

\newcommand{\rd}[1]{{\color{red}{#1}}}

\begin{document}
	
	\title[Reactive Transport in a Bulk Domain with a Branched Thin Fracture]
	{Dimension Reduction and Asymptotic Approximation of Reactive Transport in a Bulk Domain with a Branched Thin Fracture}

	\author[Taras Mel'nyk \ \& \ Christian Rohde]{Taras Mel'nyk$^\flat$ \  \& \ Christian Rohde$^\natural$  }
	\address{\hskip-12pt  $^\flat$ 
		1) \, Institute of Applied Analysis and Numerical Simulation,
		Faculty of Mathematics and Physics, Stuttgart University\\
		Pfaffenwaldring 57,\ 70569 Stuttgart,  \ Germany,\\
		\ \ 2)
		Department of Mathematical Physics, Faculty of Mathematics and Mechanics\\
		Taras Shevchenko National University of Kyiv\\
		Volodymyrska str. 64,\ 01601 Kyiv,  \ Ukraine
	}
	\email{taras.melnyk@mathematik.uni-stuttgart.de}

	\address{\hskip-12pt  $^\natural$ Institute of Applied Analysis and Numerical Simulation,
		Faculty of Mathematics and Physics, Stuttgart University\\
		Pfaffenwaldring 57,\ 70569 Stuttgart,  \ Germany
	}
	\email{christian.rohde@mathematik.uni-stuttgart.de }

	\begin{abstract}
	We study nonlinear reactive transport in a two‑dimensional bulk domain containing a thin branched fracture composed of three narrow branches of thickness $\varepsilon$ connected through a junction node of diameter $\mathcal{O}(\varepsilon).$ The microscopic model couples nonlinear parabolic reaction–diffusion equations in the bulk with an advection–diffusion equation in the fracture, where the longitudinal Péclet number is of order $\varepsilon^{-1},$ leading to advection‑dominated transport along the branches. Nonlinear side‑dependent flux conditions describe the coupling between the bulk and the fracture. As $\varepsilon \to 0,$ 
	the fracture collapses to a one‑dimensional graph, and we derive a recurrent structure of effective limit problems: a first‑order hyperbolic problem on the graph satisfying the classical Kirchhoff transmission condition at the node, and reaction–diffusion equations in the bulk with nonlinear Robin conditions on the graph edges involving the graph solution. 
	The node boundary conditions of the microscopic model do not affect these leading‑order limits. To capture the influence of the node geometry, we construct node‑layer and corner‑layer correctors and determine subsequent terms of the asymptotics expansion. Their coefficients solve auxiliary boundary‑value problems in unbounded domains with outlets at infinity and in corner‑type geometries. 
	We assemble a complete multiscale approximation combining bulk, branch, node‑layer, and corner‑layer contributions, and establish quantitative error estimates in appropriate energy norms. These estimates demonstrate the accuracy of the approximation relative to 
	$\varepsilon,$ and depend explicitly on the corner angles of the limiting graph, reflecting the geometric complexity of the fracture network.
	\end{abstract}

	\maketitle
	\tableofcontents
	

	\section{Introduction}
	
	Thin fractures, whose width is small compared to their length, act as preferential pathways for solute transport in many heterogeneous media: porous geological structures with thin crack networks, biological tissues with embedded blood vessels or root systems, microfluidic devices, and others. The geometric complexity of such fractures makes direct numerical simulation prohibitively expensive, motivating the derivation of effective macroscopic models in which thin fractures are replaced by lower-dimensional interfaces; see, for example, \cite{Ber-Grap-2022,Koch-2022,Wohl-2020} and references therein. However, many existing reduction approaches rely on simplifications or assumptions that do not accurately reflect the underlying physical processes, and their mathematical justification is often incomplete. This motivates the development of rigorous asymptotic methods capable of capturing the influence of fracture microstructure and explaining the resulting transport behavior.
	
	Rigorous justifications using various asymptotic approaches have been carried out in many works for models where a fracture is represented by a \emph{single} thin channel or layer of thickness~$\varepsilon$. Recent contributions \cite{Gahn-Radu-2018,Gahn-Jager-Radu-2022,Gahn-Jager-Radu-2021,Gahn-Radu-2025,Freud-Eden-2026,Hoerl-Rohde-2024,Lis-Kum-Pop-Rad-2020,Mel-Pop-Rohde-preprint-2026,Orl-Pan-2021} derive macroscopic dimensionally reduced models with effective interface laws in the limit $\varepsilon\to0$ and contain detailed reviews of the state of the art. Error estimates for thin heterogeneous layers coupled to bulk domains have hardly been considered in the literature so far; among the cited works, only \cite{Gahn-Jager-Radu-2021,Orl-Pan-2021,Mel-Pop-Rohde-preprint-2026} provide asymptotic error estimates.  
	
	However, to the best of our knowledge, no rigorous asymptotic justification is currently available for models involving \emph{branched} thin fractures embedded in a bulk medium. The present work initiates such an analysis.

	We study reactive, advection-dominated transport in a thin two-dimensional fracture network embedded in a bulk domain. Our focus is on understanding how the \emph{local nodal geometry}, where several thin branches meet, and the \emph{nonlinear physical interactions} at the fracture--bulk interface influence the macroscopic transport behavior. 
	
	This understanding is crucial in many applications: in geological formations,
	where thin crack networks control contaminant migration and reaction rates; in
	biological tissues, where vascular and root junctions regulate perfusion and
	chemical signaling; and in engineered porous materials and microfluidic
	devices, where nodal regions determine flow partitioning and reaction
	efficiency. These examples highlight the importance of local nodal geometry
	and nonlinear fracture–bulk interactions in shaping macroscopic transport
	behavior.
	
	\begin{figure}[htbp]  
		\vspace*{-0.3cm}
		\begin{center}
			\includegraphics[width=6cm]{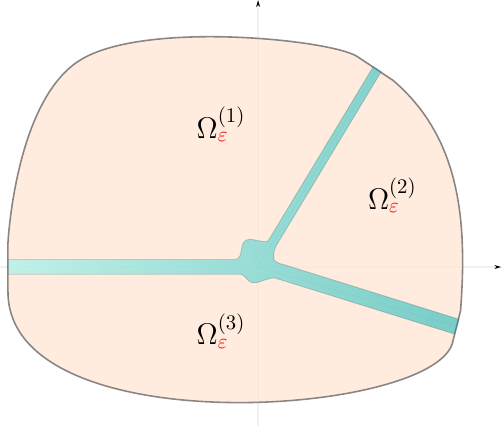} \qquad \qquad \includegraphics[width=6cm]{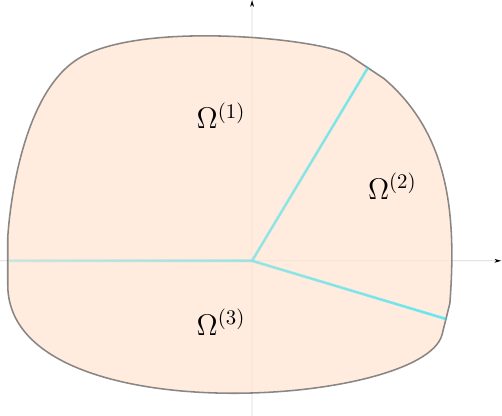}
		\end{center}
		\vspace*{-0.4cm}
		\caption{{\small a) On the left, the domain $\Omega$ is partitioned into three subdomains $\Omega_{\rd{\varepsilon}}^{(1)}$, $\Omega_{\rd{\varepsilon}}^{(2)}$ and $\Omega_{\rd{\varepsilon}}^{(3)}$ by a thin branched fracture  $\mathcal{R}_{\rd{\varepsilon}}$ of thickness$\mathcal{O}({\rd{\varepsilon}})$; \ b) on the right, the domain $\Omega$ after the degeneration of $\mathcal{R}_{\rd{\varepsilon}}$ into a graph $\mathcal{I}$ as $\rd{\varepsilon} \to 0$}}\label{fig1}
	\end{figure}
	
	Since this is the first rigorous study of processes in such a bulk‑fractured medium, we begin with a simplified fracture geometry and problem formulation; see Section~\ref{Sect-5} for a discussion of possible generalizations.
	
	For a small parameter $\varepsilon$, the thin fracture $\mathcal{R}_\varepsilon$ consists of three thin branches $\{\mathcal{R}^{(i)}_\varepsilon\}_{i=1}^3$ of thickness of order $\varepsilon$, connected by a small node $\mathcal{R}^{(0)}_\varepsilon$ of diameter of order $\varepsilon$. The fracture lies inside a bounded domain $\Omega \subset \mathbb{R}^2$ and divides it into three subdomains
	$\Omega_\varepsilon^{(1)},$ $\Omega_\varepsilon^{(2)},$ $\Omega_\varepsilon^{(3)},$
	see Fig.~\ref{fig1} on the left. A detailed description of the geometry and the problem formulation is given in Section~\ref{Sec-2}.
	
	As $\varepsilon\to0$, the fracture $\mathcal{R}_\varepsilon$ collapses to a one-dimensional graph
	$\mathcal{I}:=\mathcal{I}_1\cup\mathcal{I}_2\cup\mathcal{I}_3,$
	and each domain $\Omega_\varepsilon^{(i)}$ is transformed into a domain $\Omega^{(i)},$ so that $\Omega\setminus\mathcal{I}=\bigcup_{i=1}^{3}\Omega^{(i)},$ see Fig.~\ref{fig1} on the right. The boundary of each domain $\Omega^{(i)}$ contains a corner where two edges of the graph $\mathcal{I}$ meet with opening angle $\theta_i < \pi$.
	
	Transport in each bulk region $\Omega_\varepsilon^{(i)}$ is described by the nonlinear reaction--diffusion equation
	\[
	\partial_t u_\varepsilon^{(i)} - D_i \Delta_x u_\varepsilon^{(i)}
	+ F_i\big(u_\varepsilon^{(i)}, x, t\big) = f_i(x,t),
	\]
	and inside the thin fracture $\mathcal{R}_\varepsilon$ it is governed by the advection--diffusion equation
	\[
	\partial_t w_\varepsilon - \varepsilon\, \Delta_x w_\varepsilon +
	\mathrm{div}_x\big(w_\varepsilon\, \overrightarrow{V_\varepsilon}\big)
	= 0,
	\]
	where $\overrightarrow{V_\varepsilon}$ is a given advection field such that the associated Péclet number (a ratio of advective to diffusive transport rates) is proportional to $\varepsilon^{-1}$ in the
	longitudinal direction of the branches.

	The coupling between the bulk domains and the fracture is described by nonlinear side-dependent flux conditions
	\begin{equation}\label{int-1}
		D_i \nabla_x u^{(i)}_{\varepsilon} \cdot \boldsymbol{\nu}_\varepsilon
		= \Psi_0\big(u^{(i)}_{\varepsilon}, w_{\varepsilon}, \tfrac{x}{\varepsilon}, t \big),
		\qquad
		- \varepsilon \nabla_x w_{\varepsilon} \cdot \boldsymbol{\nu}_\varepsilon
		= \varepsilon^{\beta} \Phi_0\big(u^{(i)}_{\varepsilon}, w_{\varepsilon},\tfrac{x}{\varepsilon}, t \big),
	\end{equation}
	on the part $\Gamma^{(i,0)}_\varepsilon$ of the node boundary, and
	\begin{equation}\label{int-2}
		D_i \nabla_x u^{(i)}_{\varepsilon} \cdot \boldsymbol{\nu}_\varepsilon
		= \Psi^{(i,j)}\big(u^{(i)}_{\varepsilon}, w_{\varepsilon}, \ldots\big),
		\qquad
		\big(-\varepsilon \nabla_x w_\varepsilon + w_\varepsilon \overrightarrow{V_\varepsilon} \big) \cdot \boldsymbol{\nu}_\varepsilon
		= \varepsilon^\alpha \Phi^{(i,j)}\big(u^{(i)}_{\varepsilon}, w_{\varepsilon}, \ldots\big),
	\end{equation}
	on the lateral sides $\Gamma^{(i,j)}_\varepsilon$ of the thin branches. Here, $\alpha$ and $\beta$ are fixed intensity parameters. In the present paper we consider the case $\alpha=2$ and $\beta=1$, while other regimes are discussed in Section~\ref{Sect-5}.
	
	To analyze the behavior of reactive transport in this fractured medium, we develop a multiscale asymptotic strategy that resolves six  distinct geometric and physical scales.
	
	(i) In the bulk regions, the solution exhibits smooth behavior governed by nonlinear reaction–diffusion equations for the leading-order terms $\{u^{(i)}_0\}_{i=1}^3$, subject to nonlinear Robin conditions 
	\begin{equation}\label{int-3}
		D_i\nabla_x u^{(i)}_0\cdot\boldsymbol{\nu}_0
		=  \Psi^{(i,j)}\big(u^{(i)}_0,w^{(j)}_0, \ldots\big),
	\end{equation}
	on the graph edges (see problems \eqref{relations+}).

	(ii) Inside each thin branch, transport is advection‑dominated with a longitudinal Péclet number of order $\varepsilon^{-1}$ and interacts with the bulk reactivity through the flux boundary conditions \eqref{int-2}. This requires a two‑scaled ansatz (see \eqref{regul-1+}) that separates slow variation along the branch from rapid variation across its thickness.
	
	(iii) Near the junction node, the solution develops a \textit{node-layer} \eqref{junc} whose properties reflect the local geometry and the nonlinear flux interactions. This layer is essential for determining the coupling conditions on the limiting graph, since its matching with the two‑scaled ansatzes in the branches produces the correct transmission behavior at the vertex.
	As a consequence, the fracture leading-order terms $\{w_0^{(j)}\}$ satisfy a first-order hyperbolic problem \eqref{limit_prob} on the graph $\mathcal{I}$ with the classical Kirchhoff condition at the vertex:
	\begin{equation}\label{Kirch-1}
		\sum_{j=1}^3 v_j h_j\, w_0^{(j)}(0,t)=0.
	\end{equation}
	
	Problems \eqref{relations+} and \eqref{limit_prob} thus form a recurrent structure of coupled limit problems: a hyperbolic equation posed on the graph and, simultaneously, reaction–diffusion problems in the bulk whose boundary conditions \eqref{int-3} explicitly involve the graph solution.
	
	Crucially, the node boundary condition \eqref{int-1} and the second lateral conditions in \eqref{int-2} have no influence on these leading-order limit problems. This is natural: for problems such as \eqref{relations+}, a condition imposed at a single point has no mathematical meaning.
	
	Thus, the main novelty of the present paper lies in finding the subsequent terms of the asymptotics and in constructing  an asymptotic approximation that rigorously captures the influence of the node-boundary conditions.
	
	As will be demonstrated, this influence is most clearly manifested in a nonstandard Kirchhoff transmission condition
		\begin{equation}\label{Kirch-2}
			\sum_{j=1}^{3} 2 h_j\,\mathrm{v}_j\, w_1^{(j)}(0,t)
			= w^{(1)}_0(0,t)\, |\Xi^{(0)}|
			- \sum_{i=1}^{3}\int_{\Gamma^{(i,0)}}\Phi_0\big(u^{(i)}_0, w^{(1)}_0(0,t), \xi,t\big)\, dl_\xi
		\end{equation}
		for the fracture second-order terms $\{w_1^{(j)}\},$  forming the mixed hyperbolic problem \eqref{prob_w_1} on the graph (here $|\Xi^{(0)}|$ denotes the Lebesgue measure of the rescaled node).
		
	(iv) To capture this influence in the bulk domains, specially designed  \textit{corner-layer} ansatzes are introduced in 
	Subsection~\ref{corner-2-4}. Their coefficients $\{\mathcal{K}^{(i)}_1\}_{i=1}^3$  solve auxiliary
	boundary-value problems in unbounded  corner-type domains with the Neumann condition
	\[
	D_i \nabla_\zeta \mathcal{K}^{(i)}_1 \cdot \boldsymbol{\nu}_\zeta
	= - \Psi_0\big(u_0^{(i)}(0,t), w_0^{(1)}(0,t),\ldots\big)
	\]
	on the corresponding rescaled part of the node boundary. An important feature of these corner-layer correctors is their 
	logarithmic behavior at infinity, governed by Kondrat’ev’s corner‑singularity theory.  
	
	(v) To neutralize this growth, additional \emph{bulk correctors} are introduced in \S~\ref{bulk-corr}, whose leading behavior near the
	vertex matches the logarithmic term generated by the corner-layers. This ensures that the bulk and inner-corner asymptotics are compatible in the intermediate region and that the composite expansion remains uniformly valid
	throughout $\Omega^{(i)}$.
	
	(vi) In the two outflow branches (with respect to the advection field), the two‑scaled ansatz fails to satisfy the Dirichlet outflow conditions on the outer sides. To restore the correct boundary behavior, we construct \textit{boundary‑layer} correctors in Subsection~\ref{subsec_Bound_layer}, which remove the resulting residuals and exponentially vanish at infinity. These layers complete the proposed multiscale scheme by enforcing the microscopic outflow conditions.
	
	The main result of the paper is the construction of a complete multiscale approximation incorporating bulk, branch, boundary-layer, node-layer, and corner-layer contributions, together with quantitative error energy estimates (Theorem~\ref{apriory_estimate}) demonstrating the accuracy of the approximation. These estimates depend explicitly on the corner angles
	$\theta_1,\theta_2,\theta_3$ of the graph, reflecting the geometric complexity
	of the fracture network.

	The paper is organized as follows. Section~\ref{Sec-2} introduces the geometric configuration of the bulk domain and the thin fracture network, formulates the assumptions on the data, and presents the microscopic problem in precise form. Section~\ref{Sect-3} develops the asymptotic analysis for the various parts of the fractured medium, including the node-layer and corner-layer constructions and the auxiliary boundary-value problems in unbounded domains. It derives the effective limit problems on the graph and in the bulk regions and establishes the well-posedness of these problems and the regularity of their solutions. Section~\ref{Sect-4} assembles the multiscale approximation in the entire domain, computes and estimates the residuals that this approximation leaves in the differential equations and boundary conditions of the original problem, and proves the rigorous error estimates that confirm the accuracy of the approximation and the adequacy of the limit problems. Finally, Section~\ref{Sect-5} discusses the results, possible extensions of the analysis, and directions for future research.

	
	\section{Problem statement}\label{Sec-2}
	
	We begin by describing the geometry of the thin fracture network embedded in the bulk
	domain. A description of the advection field is given in \S~\ref{convextion-flux}, while the
	problem formulation and the assumptions on the data are presented in \S~\ref{ibvp}.
	
	\subsection{Domain structure}\label{subsec-1-1}
	
	\subsubsection*{Thin graph-like fracture}\label{domain}
	
	In the Euclidean plane $\mathbb{R}^2$ consider three unit vectors
	$\mathbf{e}^{(i)}=(e_1^{(i)},e_2^{(i)})$, $i\in\{1,2,3\}$,
	with $\mathbf{e}^{(1)}=(-1,0)$, $e_1^{(i)}>0$ for $i\in\{2,3\}$, and
	$e_2^{(2)}>0$, $e_2^{(3)}<0$.
	
	For each $i\in\{1,2,3\}$ let
	$\boldsymbol{\varsigma}^{(i)}=(\varsigma_1^{(i)},\varsigma_2^{(i)})$
	be the unit vector obtained by rotating $\mathbf{e}^{(i)}$ by $90^\circ$
	counterclockwise. The pair $(\mathbf{e}^{(i)},\boldsymbol{\varsigma}^{(i)})$
	is therefore an orthonormal, positively oriented basis of $\mathbb{R}^2$.
	
	Thus, every point $x=(x_1,x_2)\in\mathbb{R}^2$ has coordinates
	$y^{(i)}=(y_1^{(i)},y_2^{(i)})$ in the basis
	$(\mathbf{e}^{(i)},\boldsymbol{\varsigma}^{(i)})$, related to the standard
	Cartesian coordinates $x$ by
	\begin{equation}\label{matrix}
		y^{(i)}=\mathbb{A}_i\,x,
		\qquad
		\mathbb{A}_i=\begin{pmatrix}
			e_1^{(i)} & e_2^{(i)}
			\\[4pt]
			\varsigma_1^{(i)} & \varsigma_2^{(i)}
		\end{pmatrix}.
	\end{equation}
	Since the rows of $\mathbb{A}_i$ are orthonormal basis vectors, we have
	$\mathbb{A}_i^{-1}=\mathbb{A}_i^{T}$ and $\det\mathbb{A}_i=1$.
	
	In each coordinate system $(\mathbf{e}^{(i)},\boldsymbol{\varsigma}^{(i)})$
	we define the thin rectangle
	\[
	\mathcal{R}_\varepsilon^{(i)}
	=\left\{\,y^{(i)}\in\mathbb{R}^2:\ 
	\varepsilon\ell_0<y_1^{(i)}<\ell_i,\ 
	|y_2^{(i)}|<\varepsilon h_i\,\right\},
	\]
	where $\varepsilon>0$ is a small parameter, $h_i>0$ is a constant,
	$\ell_0\in(0,\tfrac{1}{3})$, and $\ell_i\ge1$ are given numbers. We denote by
	\[
	\Upsilon_\varepsilon^{(i)}(\mu)
	:=\mathcal{R}_\varepsilon^{(i)}\cap
	\{\,y^{(i)}:\ y_1^{(i)}=\mu\,\},
	\qquad
	\mu\in[\varepsilon\ell_0,\ell_i],
	\]
	the cross-section of $\mathcal{R}_\varepsilon^{(i)}$ at $y_1^{(i)}=\mu$.
	
	\begin{remark}
		In what follows, a domain $Q$ described in the coordinates $x$ will be
		denoted by the same symbol $Q$ when described in the coordinates $y^{(i)}$;
		the coordinate system used will always be clear from the context.
	\end{remark}
	
	The thin rectangles are joined through a domain $\mathcal{R}_\varepsilon^{(0)}$
	(referred to as the \emph{node}) obtained by the homothetic transformation
	with coefficient $\varepsilon$ from a bounded domain $\Xi^{(0)}$ containing
	the origin:
	\begin{equation}\label{node}
		\mathcal{R}_\varepsilon^{(0)}=\varepsilon\,\Xi^{(0)}.
	\end{equation}
	We assume that the boundary $\partial\Xi^{(0)}$ contains the segments
	\[
	\Upsilon^{(i)}_1(\ell_0)
	:=\overline{\Xi^{(0)}}\cap\{\,x:\ y_1^{(i)}=\ell_0\,\},
	\qquad i\in\{1,2,3\}.
	\]

	Thus, the thin graph-like fracture $\mathcal{R}_\varepsilon$ is the interior of
	the union $\bigcup_{i=0}^{3}\overline{\mathcal{R}_\varepsilon^{(i)}}$.
	We assume that the surface
	$\partial\mathcal{R}_\varepsilon\setminus\bigcup_{i=1}^{3}\overline{\Upsilon_\varepsilon^{(i)}(\ell_i)}$
	is $C^3$-smooth.
	
	The fracture lies inside a bounded domain $\Omega\subset\mathbb{R}^2$ with
	$C^3$-smooth boundary $\partial\Omega$, and we assume that $\partial\Omega$
	contains three boundary intervals, each containing one of the segments
	$\overline{\Upsilon_\varepsilon^{(1)}(\ell_1)}$,
	$\overline{\Upsilon_\varepsilon^{(2)}(\ell_2)}$,
	$\overline{\Upsilon_\varepsilon^{(3)}(\ell_3)}$.
	Thus $\mathcal{R}_\varepsilon$ divides $\Omega$ into three subdomains
	$\Omega_\varepsilon^{(1)},\Omega_\varepsilon^{(2)},\Omega_\varepsilon^{(3)}$
	(see Fig.~\ref{fig1}), i.e.
	$
	\Omega\setminus\overline{\mathcal{R}_\varepsilon}
	=\bigcup_{i=1}^{3}\Omega_\varepsilon^{(i)}.
	$
	\begin{figure}[htbp]  
		\vspace*{-0.3cm}
		\begin{center}
			\includegraphics[width=6cm]{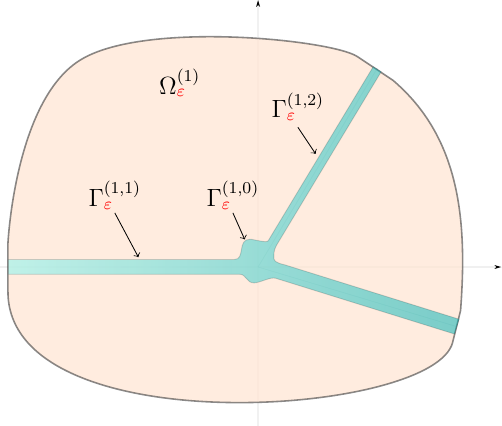}
		\end{center}
		\vspace*{-0.4cm}
		\caption{{\small Parts of the interface between $\Omega_{\rd{\varepsilon}}^{(1)}$ and the thin fracture $\mathcal{R}_{\rd{\varepsilon}}$}}\label{fig2}
	\end{figure}
	
	We represent the common part of the boundary of each domain
	$\Omega_\varepsilon^{(i)}$ and the thin fracture $\mathcal{R}_\varepsilon$ (see Fig.~\ref{fig2} for $\Omega_\varepsilon^{(1)}$)
	as the unions
	\begin{gather}\label{com1}
		\partial\Omega_\varepsilon^{(1)}\cap\partial\mathcal{R}_\varepsilon
		=\Gamma_\varepsilon^{(1,1)}\cup\Gamma_\varepsilon^{(1,0)}\cup\Gamma_\varepsilon^{(1,2)},
		\\[4pt]
		\label{com2}
		\partial\Omega_\varepsilon^{(2)}\cap\partial\mathcal{R}_\varepsilon
		=\Gamma_\varepsilon^{(2,2)}\cup\Gamma_\varepsilon^{(2,0)}\cup\Gamma_\varepsilon^{(2,3)},
		\\[4pt]
		\label{com3}
		\partial\Omega_\varepsilon^{(3)}\cap\partial\mathcal{R}_\varepsilon
		=\Gamma_\varepsilon^{(3,1)}\cup\Gamma_\varepsilon^{(3,0)}\cup\Gamma_\varepsilon^{(3,3)},
	\end{gather}
	where the first index $i$ in $\Gamma_\varepsilon^{(i,j)}$ indicates the
	connection to $\partial\Omega_\varepsilon^{(i)}$, and the second index $j$
	indicates the side of the corresponding rectangle $\mathcal{R}_\varepsilon^{(j)}$
	if $j\neq0$, or the corresponding part of the boundary of the node
	$\mathcal{R}_\varepsilon^{(0)}$ if $j=0$.
	
	As $\varepsilon\to0$, the thin junction $\mathcal{R}_\varepsilon$ shrinks to
	the graph
	\begin{equation}\label{graph}
		\mathcal{I}:=\mathcal{I}_1\cup\mathcal{I}_2\cup\mathcal{I}_3,
	\end{equation}
	where
	$\mathcal{I}_i :=\{\,y^{(i)}\in\mathbb{R}^2:\ 0\le y_1^{(i)}<\ell_i,\ y_2^{(i)}=0\,\},$ and each domain $\Omega_\varepsilon^{(i)}$ is transformed into a domain
	$\Omega^{(i)}$. Thus
	\begin{equation}\label{limit-domain}
		\Omega\setminus\mathcal{I}=\bigcup_{i=1}^{3}\Omega^{(i)}.
	\end{equation}
	The boundary of each domain $\Omega^{(i)}$ contains a corner at the origin,
	where two edges of the graph $\mathcal{I}$ meet with opening angle $\theta_i$.
	The assumptions on the unit vectors $\mathbf{e}^{(1)},\mathbf{e}^{(2)},\mathbf{e}^{(3)}$
	imply $\theta_i<\pi$. We also assume that, for each $i\in\{1,2,3\}$, the curve
	$\Gamma_\varepsilon^{(i,0)}$ lies inside the domain $\Omega^{(i)}$, which means
	that $\Omega_\varepsilon^{(i)}\subset\Omega^{(i)}$.

	
	\subsection{Description of the advection vector field $\overrightarrow{V_\varepsilon}$ in the fracture  $\mathcal{R}_\varepsilon$}\label{convextion-flux}
	
	We assume first that all components of the vector-valued function
	$\overrightarrow{V_\varepsilon}$ belong to the space
	$C^{5}\big(\overline{\mathcal{R}_\varepsilon}\big)$.
	The structure of $\overrightarrow{V_\varepsilon}$ depends on the part of the
	thin junction $\mathcal{R}_\varepsilon$. In the node $\mathcal{R}_\varepsilon^{(0)}$ we set
	\[
	\overrightarrow{V_\varepsilon}(x)
	=\big(v^{(0)}_1(\tfrac{x}{\varepsilon}),\,
	v^{(0)}_2(\tfrac{x}{\varepsilon})\big)
	=: \overrightarrow{V_\varepsilon}^{(0)}(x),
	\]
	and in each thin rectangle we define
	\begin{equation}\label{conv-1}
		\overrightarrow{V_\varepsilon}^{(i)}(y^{(i)})
		:=\mathbb{A}_i\,\overrightarrow{V_\varepsilon}(x)\big|_{x=\mathbb{A}_i^{-1}y^{(i)}}
		=\Big(v^{(i)}_1(y_1^{(i)}),\,
		\varepsilon\,v^{(i)}_2\big(y_1^{(i)},\tfrac{y_2^{(i)}}{\varepsilon}\big)\Big),
		\qquad y^{(i)}\in\mathcal{R}_\varepsilon^{(i)},
	\end{equation}
	in the rotated coordinates $y^{(i)}$, where the matrix $\mathbb{A}_i$ is defined in~\eqref{matrix}.
	
	For each $i\in\{1,2,3\}$ we assume that the longitudinal component
	$v^{(i)}_1(y_1^{(i)})$, $y_1^{(i)}\in[0,\ell_i]$, is equal to a constant
	$\mathrm{v}_i$ on a small interval $[0,\delta_i]$.
	The transversal component
	$v^{(i)}_2\big(y_1^{(i)},\tfrac{y_2^{(i)}}{\varepsilon}\big)$
	has compact support with respect to the longitudinal variable $y_1^{(i)}$;
	in particular, it vanishes on $[0,\delta_i]$.
	Consequently,
	\[
	\overrightarrow{V_\varepsilon}^{(i)}
	=\mathrm{v}_i\,\mathbf{e}^{(i)}
	\quad\text{in the coordinates $x$ near the side }
	\Upsilon_\varepsilon^{(i)}(\varepsilon\ell_0)
	\text{ of }\mathcal{R}_\varepsilon^{(i)}
	\text{ for $\varepsilon$ sufficiently small}.
	\]
	
	By the smoothness of $\overrightarrow{V_\varepsilon}$ we also have
	\begin{equation}\label{V_1}
		\overrightarrow{V_\varepsilon}^{(0)}(x)\big|_{x\in\Upsilon_\varepsilon^{(i)}(\varepsilon\ell_0)}
		=\mathrm{v}_i\,\mathbf{e}^{(i)}.
	\end{equation}
	
	Furthermore, for all $y_1^{(i)}\in[0,\ell_i]$ we assume
	\begin{equation}\label{V_2}
		v^{(1)}_1<0,
		\qquad
		v^{(i)}_1>0\quad\text{for }i\in\{2,3\},
	\end{equation}
	which means that the vector field $\overrightarrow{V_\varepsilon}$ enters the
	thin branch $\mathcal{R}_\varepsilon^{(1)}$ and leaves
	$\mathcal{R}_\varepsilon^{(2)}$ and $\mathcal{R}_\varepsilon^{(3)}$.
	
	For the vector field in the node $\mathcal{R}_\varepsilon^{(0)}$ we impose an
	additional assumption: it is generated by a potential $p$
	solving the boundary-value problem
	\begin{equation}\label{potential0}
		\left\{
		\begin{array}{rcll}
			\Delta_\xi p(\xi) &=& 0, & \xi:=(\xi_1,\xi_2)\in\Xi^{(0)},
			\\[2pt]
			\partial_{\boldsymbol{\nu}_\xi}p(\xi) &=& \mathrm{v}_i,
			& \xi\in\Upsilon^{(i)}_1(\ell_0),\quad i\in\{1,2,3\},
			\\[2pt]
			\partial_{\boldsymbol{\nu}_\xi}p(\xi) &=& 0,
			& \xi\in\partial\Xi^{(0)}\setminus\bigcup_{i=1}^3\Upsilon^{(i)}_1(\ell_0),
		\end{array}
		\right.
	\end{equation}
	where $\Delta_\xi$ is the Laplace operator and
	$\partial_{\boldsymbol{\nu}_\xi}$ denotes the outward normal derivative on
	$\partial\Xi^{(0)}$.
	The Neumann problem \eqref{potential0} has a solution if and only if
	\begin{equation}\label{cond_1}
		\sum_{i=1}^{3} h_i\,\mathrm{v}_i = 0.
	\end{equation}
	To ensure uniqueness we additionally impose the normalization
	$\int_{\Xi^{(0)}} p\,d\xi = 0$.
	Thus,
	\begin{equation}\label{field_0}
		\overrightarrow{V_\varepsilon}^{(0)}(x)
		=\nabla_\xi p(\xi)\big|_{\xi=x/\varepsilon}
		=\varepsilon\,\nabla_x\big(p(\tfrac{x}{\varepsilon})\big),
		\qquad x\in\mathcal{R}_\varepsilon^{(0)}.
	\end{equation}
	
	Since $p$ is harmonic, the advection field is divergence-free and mass-conserving, i.e.,
	$\mathrm{div}_x\overrightarrow{V_\varepsilon}^{(0)}=0$ in
	$\mathcal{R}_\varepsilon^{(0)}$, and the amount of flow entering $\mathcal{R}_\varepsilon^{(0)}$
	equals the amount leaving it.
	\begin{remark}
		The advection field $\overrightarrow{V_\varepsilon}$ may also depend on time
		$t\in[0,T]$, as in \cite{Mel-Roh_JMAA-2024}.
	\end{remark}
	
	
	\subsection{The initial–boundary value problem}\label{ibvp}
	
	We introduce the piecewise-defined function
	\begin{equation}\label{orig-solution}
		U_\varepsilon(x,t):=
		\begin{cases}
			u_\varepsilon^{(i)}(x,t), & (x,t)\in\Omega_\varepsilon^{(i)}\times(0,T),\quad i\in\{1,2,3\},
			\\[3pt]
			w_\varepsilon(x,t), & (x,t)\in\mathcal{R}_\varepsilon\times(0,T),
		\end{cases}
	\end{equation}
	which collects the unknowns in the bulk subdomains and in the thin graph-like junction.
	The evolution of $U_\varepsilon$ is governed by coupled parabolic diffusion-reaction equations and advection–diffusion
	equations posed on the composite geometry of $\Omega$.
	
	In each bulk component $\Omega_{\varepsilon}^{(i)}$, the function
	$u_\varepsilon^{(i)}$ satisfies the parabolic diffusion-reaction equation
	\begin{equation}\label{p-equations}
		\partial_t u_\varepsilon^{(i)} - D_i \Delta_x u_\varepsilon^{(i)}
		+ F_i\big(u_\varepsilon^{(i)}, x, t\big) = f_i(x,t)
		\quad \text{in } \Omega_\varepsilon^{(i)}\times(0,T),\qquad i\in\{1,2,3\},
	\end{equation}
	where $D_i>0$ denotes the diffusion coefficient.
	Within the thin region $\mathcal{R}_\varepsilon$, the component $w_\varepsilon$
	satisfies the advection–diffusion equation
	\begin{equation}\label{p-equations+}
		\partial_t w_\varepsilon - \varepsilon\,\Delta_x w_\varepsilon
		+ \mathrm{div}_x\big(w_\varepsilon\,\overrightarrow{V_\varepsilon}\big)
		= 0
		\quad \text{in } \mathcal{R}_\varepsilon\times(0,T),
	\end{equation}
	with the standard notation
	$\partial_t := \frac{\partial}{\partial t}$ and
	$\Delta_x := \partial_{x_1}^2 + \partial_{x_2}^2$.
	
	On the outer boundary parts we prescribe
	\begin{equation}\label{p-outer-bc}
		u_\varepsilon^{(i)} = 0
		\quad \text{on } \partial\Omega\cap\partial\Omega_\varepsilon^{(i)},
		\qquad
		w_\varepsilon = q_i(t)
		\quad \text{on } \Upsilon_\varepsilon^{(i)}(\ell_i),\ t\in(0,T),\ i\in\{1,2,3\}.
	\end{equation}
	
	On the common interfaces $\partial\Omega_\varepsilon^{(i)}\cap\partial\mathcal{R}_\varepsilon$
	(see \eqref{com1}–\eqref{com3}) we impose nonlinear side-dependent flux coupling
	conditions. On the node interfaces $\Gamma_\varepsilon^{(i,0)}$ we set
	\begin{equation}\label{coup-1}
		\left\{
		\begin{array}{rcll}
			D_i \nabla_x u^{(i)}_{\varepsilon} \cdot \boldsymbol{\nu}_\varepsilon
			&=& \Psi_0\big(u^{(i)}_{\varepsilon}, w_{\varepsilon}, \tfrac{x}{\varepsilon}, t\big)
			& \text{on } \Gamma^{(i,0)}_\varepsilon\times(0,T),
			\\[2mm]
			-\varepsilon\,\nabla_x w_{\varepsilon} \cdot \boldsymbol{\nu}_\varepsilon
			&=& \varepsilon\,\Phi_0\big(u^{(i)}_{\varepsilon}, w_{\varepsilon}, \tfrac{x}{\varepsilon}, t\big)
			& \text{on } \Gamma^{(i,0)}_\varepsilon\times(0,T),
		\end{array}
		\right.
	\end{equation}
	where $\boldsymbol{\nu}_\varepsilon$ denotes the outward unit normal to
	$\partial\mathcal{R}_\varepsilon$.
	On the lateral sides of the thin branches  we impose
	\begin{equation}\label{coupl-1+}
		\left\{
		\begin{array}{rcll}
			D_i \nabla_x u^{(i)}_{\varepsilon} \cdot \boldsymbol{\nu}_\varepsilon
			&=& \Psi^{(i,j)}\big(u^{(i)}_{\varepsilon}, w_{\varepsilon}, y^{(j)}_1, t\big)
			& \text{on } \Gamma^{(i,j)}_\varepsilon\times(0,T),
			\\[2mm]
			\big(-\varepsilon \nabla_x w_\varepsilon + w_\varepsilon \overrightarrow{V_\varepsilon}\big)
			\cdot \boldsymbol{\nu}_\varepsilon
			&=& \varepsilon^2\,\Phi^{(i,j)}\big(u^{(i)}_{\varepsilon}, w_{\varepsilon}, y^{(j)}_1, t\big)
			& \text{on } \Gamma^{(i,j)}_\varepsilon\times(0,T),
		\end{array}
		\right.
	\end{equation}
	where $y^{(j)}=\mathbb{A}_j x$ (see \eqref{matrix}) and the admissible pairs
	$(i,j)$ are given by
	\begin{equation}\label{indexes}
		i=1:\ j\in\{1,2\},\qquad
		i=2:\ j\in\{2,3\},\qquad
		i=3:\ j\in\{1,3\}.
	\end{equation}
	
	\begin{remark}\label{Rem-1-4}
		Whenever functions or boundary components carry indices $(i,j)$, the indices
		are understood to vary according to rule \eqref{indexes}.
	\end{remark}
	
	The initial condition is
	\begin{equation}\label{p-in}
		U_\varepsilon(x,0)=0.
	\end{equation}
	
	The functions appearing on the right-hand sides of
	\eqref{p-equations}–\eqref{p-in} are assumed to be known and satisfy the
	following conditions.
	\begin{description}
		\item[A1]
		For each $i\in\{1,2,3\}$ the source term $f_i(x,t)$,
		$(x,t)\in\Omega^{(i)}\times[0,T]$, and the reaction term
		$F_i(s,x,t)$, $(s,x,t)\in\mathbb{R}\times\Omega^{(i)}\times[0,T]$,
		belong to $C^2$ and are uniformly bounded together with all their derivatives.
		They have compact support in $\Omega^{(i)}$, independent of $s$ and $t$, and
		satisfy the compatibility conditions
		$F_i|_{t=0}=0$, $\partial_s F_i|_{t=0}=0$, and $f_i|_{t=0}=0$.
		In addition,
		\begin{equation}\label{pos-1}
			\partial_s F_i(s,x,t)\ge 0
			\quad\text{on }\mathbb{R}\times\Omega^{(i)}\times[0,T].
		\end{equation}
		
		\item[A2]
		The boundary data $\{q_i(t)\}_{i=1}^3$ in \eqref{p-outer-bc} are nonnegative,
		$q_i\in C^4([0,T])$, $q_i(0)=q_i'(0)=0$ for $i\in\{2,3\}$, and
		\begin{equation}\label{q_1}
			\frac{d^k q_1}{dt^k}(0)=0
			\quad\text{for } k\in\{0,1,2,3,4\}.
		\end{equation}
		
		\item[A3]
		The functions $\Psi_0(s,\omega,\xi,t)$ and $\Phi_0(s,\omega,\xi,t)$,
		$(s,\omega,\xi,t)\in X_0:=\mathbb{R}\times\mathbb{R}\times\overline{\Xi^{(0)}}\times[0,T]$,
		belong to $C^2(X_0)$ and are uniformly bounded together with all their derivatives.
		They vanish uniformly (in $s,\omega,t$) in neighborhoods of the sets
		$\{\Upsilon^{(i)}_1(\ell_0)\}_{i=1}^3$, and satisfy
		\begin{itemize}
			\item $\Psi_0|_{t=0}=0$ and $\Psi_0\ge 0$;
			\item
			\begin{equation}\label{com-con-Phi}
				\Phi_0|_{t=0}
				=\partial_t\Phi_0|_{t=0}
				=\partial_{tt}^2\Phi_0|_{t=0}
				=\partial_s\Phi_0|_{t=0}
				=\partial_{ss}^2\Phi_0|_{t=0}
				=0.
			\end{equation}
		\end{itemize}
		
		\item[A4]
		For each admissible pair $(i,j)$ (see \eqref{indexes}), the functions
		\[
		\Psi^{(i,j)}(s,\omega,y^{(j)}_1,t),\qquad
		\Phi^{(i,j)}(s,\omega,y^{(j)}_1,t),
		\]
		defined on $X_{i,j}:=\mathbb{R}\times\mathbb{R}\times[0,\ell_i]\times[0,T]$,
		belong to $C^2(X_{i,j})$ and are uniformly bounded together with all their
		derivatives. They vanish in neighborhoods of the endpoints of $[0,\ell_i]$
		uniformly in the other variables, and satisfy:
		\begin{itemize}
			\item
			\begin{equation}\label{pos-2}
				\partial_s\Psi^{(i,j)}(s,\omega,y^{(j)}_1,t)\ge 0
				\quad\text{on } X_{i,j};
			\end{equation}
			\item
			there exists $\delta>0$ such that
			\begin{equation}\label{comp-cond-Psi}
				\Psi^{(i,j)}(\cdot,\cdot,\cdot,t)=0
				\quad\text{for all } t\in[0,\delta];
			\end{equation}
			\item
			$\Phi^{(i,j)}|_{t=0}=0$.
		\end{itemize}
	\end{description}

	To proceed, we define a weak solution to problem \eqref{p-equations}--\eqref{p-in}, will henceforth be called the problem~$\mathbb {P}_\varepsilon.$
	
	\subsubsection*{Weak formulation}	Denote by $\mathfrak{H}^{i, *}_\varepsilon$ the dual space to the Sobolev space 
	$$
	\mathfrak{H}^{(i)}_\varepsilon := \big\{u\in H^1(\Omega^{(i)}_\varepsilon)\colon \ {u}\big|_{\partial\Omega^{(i)}_\varepsilon \cap\,  \partial\Omega} = 0\big\}.
	$$ 
	Similarly, let $\mathfrak{H}^f_\varepsilon :=H^1(\mathcal{R}_\varepsilon)$
	and $\mathfrak{H}^{f,*}_\varepsilon$ its dual.
	The pairings are denoted by
	$\langle\cdot,\cdot\rangle_\varepsilon^{(i)}$ and
	$\langle\cdot,\cdot\rangle_\varepsilon^{f}$.

	\begin{definition}\label{Def-weak-sol}
		A function $U_\varepsilon = \begin{cases}
			u_\varepsilon^{(i)}, & \text{in} \ \Omega_\varepsilon^{(i)}\times(0,T),\   i\in\{1,2,3\},\\
			w_\varepsilon, & \text{in} \ \mathcal{R}_\varepsilon\times(0,T),
		\end{cases}
		$ 
		\ with 
		$$
		u_\varepsilon^{(i)}\in L^2(0,T;\mathfrak{H}^{(i)}_\varepsilon), \quad
		\partial_t u_\varepsilon^{(i)}\in L^2(0,T;\mathfrak{H}^{i,*}_\varepsilon), \quad 
		w_\varepsilon\in L^2(0,T;\mathfrak{H}^f_\varepsilon), \quad \partial_t w_\varepsilon\in L^2(0,T;\mathfrak{H}^{f,*}_\varepsilon),
		$$
		is a \emph{weak solution} to problem $\mathbb{P}_\varepsilon$ if
		the following identities and conditions hold:
		\begin{itemize}
			\item for a.e.\ $t\in(0,T)$, for all $i\in\{1,2,3\}$ and all
			$\phi\in\mathfrak{H}^{(i)}_\varepsilon$,
			\begin{multline}\label{identity-1}
				\langle\partial_t u^{(i)}_\varepsilon, \phi \rangle^{(i)}_\varepsilon 
				\, + \,  D_i\int_{\Omega^{(i)}_\varepsilon}\nabla_x u^{(i)}_\varepsilon \cdot \nabla_x \phi \, dx 
				\, +\,  \int_{\Omega^{(i)}_\varepsilon} F_i(u^{(i)}_\varepsilon, x, t) \, \phi \, dx  
				\, +\,  \int_{\Gamma^{(i,0)}_\varepsilon} \Psi_0\big(u^{(i)}_{\varepsilon}, w_{\varepsilon}, \tfrac{x}{\varepsilon}, t \big) \, \phi \, dl_{x}
				\\
				\sum_j\int_{\Gamma^{(i,j)}_\varepsilon} \Psi^{(i,j)}\big(u^{(i)}_{\varepsilon}, w_{\varepsilon}, \cdot, t \big)
				\, \phi \, dl_{x} = \int_{\Omega^{(i)}_\varepsilon} f_i(x, t) \, \phi \, dx ;
			\end{multline}
			\item for a.e. \ \(t\in(0,T)\) and all
			$\phi \in \mathfrak{H}^{(f)}_\varepsilon$ with
			\(\phi|_{\Upsilon_\varepsilon^{(i)}(\ell_i)}= 0, \ i \in \{1,2,3\},\)  
			\begin{multline}\label{identity-f}
				\langle\partial_t w_\varepsilon, \phi \rangle^f_\varepsilon 
				\, + \,  \varepsilon \int_{\mathcal{R}_\varepsilon} \nabla_x w_\varepsilon \cdot \nabla_x\phi \, dx \,-\, \int_{\mathcal{R}_\varepsilon}  w_\varepsilon \, \overrightarrow{V_\varepsilon}\cdot \nabla_x\phi \, dx = 
				\\
				- 	\varepsilon \sum_{i=1}^{3}\int_{\Gamma^{(i,0)}_\varepsilon} \Phi_0\big(u^{(i)}_{\varepsilon}, w_{\varepsilon},\tfrac{x}{\varepsilon}, t \big) \, \phi \, dl_{x}
				- 	\varepsilon^2 \sum_{i=1}^{3} \sum_j\int_{\Gamma^{(i,j)}_\varepsilon} \Phi^{(i,j)}\big(u^{(i)}_{\varepsilon}, w_{\varepsilon}, \cdot,  t \big) \, \phi \, dl_{x};
			\end{multline}
			\item and  
			$$
			U_\varepsilon|_{t=0}= 0, \qquad
			w_\varepsilon|_{\Upsilon_\varepsilon^{(i)}(\ell_i)}= q_i(t), \ \ i \in \{1,2,3\}.
			$$ 
		\end{itemize}
	\end{definition}
	
	\begin{remark}
		The summation over $j$ in \eqref{identity-1} and \eqref{identity-f}
		follows the rule specified in \eqref{indexes}. 
	\end{remark}
	
	\begin{remark}
		It is known that if 
		$u^{(i)}_\varepsilon \in L^2(0,T; \mathfrak{H}^{(i)}_\varepsilon)$ and $\partial_t u^{(i)}_\varepsilon \in L^2(0,T; \mathfrak{H}^{i, *}_\varepsilon)$, then the function $u^{(i)}_\varepsilon \in C([0,T]; L^2(\Omega^{(i)}_\varepsilon)).$ Therefore,
		equality $u^{(i)}_\varepsilon|_{t=0}= 0$ is meaningful. The same holds for $w_\varepsilon.$
	\end{remark}
	
	\begin{remark}\label{rem-1-7}
		As will be seen later, the smoothness of the advection field $\overrightarrow{V_\varepsilon}$ and the nonlinear functions is required in order to determine higher-order terms in the asymptotic expansion of the solution to problem~$\mathbb{P}_\varepsilon.$ 
		In particular, this regularity guarantees Lipschitz continuity with respect to the variables 
		$s$ and $w$ with constants independent of the arguments; the compact-support assumptions model localized reactions (e.g., solute penetration from the solid into the fracture); relations~\eqref{comp-cond-Psi} mean that chemical activity in the fracture begins only after a delay, which is natural in applications. All these conditions are also important to enforce the compatibility conditions  in many problems, which determine the coefficients in the asymptotic approximation, and to establish the regularity of solutions to the limit problems in the bulk regions, see Remark \ref{rem-2-2} and  Proposition~\ref{prop-2-1}. 
	\end{remark}

	The existence and uniqueness of a weak solution to problem $\mathbb{P}_\varepsilon$ for each fixed $\varepsilon>0,$ under the above assumptions on the nonlinear terms, is standard; it may be obtained, for example, by a Galerkin approximation (cf.\ \cite[Sect.~3]{Mel-Roh_Non-Diff-2024}) or by Sch\"afer's fixed point theorem (cf.\ \cite[Prop.~1]{Gahn-Radu-2018}).
	
	\medskip
	
	Our goal is to study the asymptotic behavior of the solution $U_\varepsilon$ as
	$\varepsilon\to0$, i.e., when the thin graph-like fracture $\mathcal{R}_\varepsilon$ shrinks to the one-dimensional graph
	$\mathcal{I}$ (see~\eqref{graph}) and 
	the domain $\Omega$ becomes the union of three bulk regions $\{\Omega^{(i)}\}_{i=1}^3$
	separated by the graph. As part of this analysis, we
	\begin{itemize}
		\item 
		derive effective recurrent coupled limit problems: a first-order hyperbolic problem posed on the graph and, simultaneously, limit problems in the bulk domains whose boundary conditions explicitly involve the graph solution;
		\item 
		establish the well-posedness of these limit problems, including regularity of their solutions;
		\item 
		determine the subsequent terms of the asymptotic expansion to capture the influence of microscale features such as the node geometry, physical processes on and inside the node, as well as the branch structure of the fracture, on the macroscopic behavior of the diffusion–reaction–advection model;
		\item 
		construct an asymptotic approximation of the solution to problem $\mathbb{P}_\varepsilon$ and prove the corresponding error estimates.
	\end{itemize}
	
	\begin{remark}
		The present work focuses on a simplified fracture geometry and problem
		formulation in order to isolate and analyze the processes occurring near and
		within a small junction node. More general geometries and formulations can be
		treated by the same multiscale methodology and are discussed in Section~\ref{Sect-5}, but their analysis would require additional technical constructions
		and is therefore not pursued here.
	\end{remark}


	\section{Formal asymptotic analysis}\label{Sect-3}
	
	In each bulk domain $\Omega_\varepsilon^{(i)}$ $(i\in\{1,2,3\})$ we propose the following ansatz for the solution:
	\begin{equation}\label{exp-1}
		u^{(i)}_\varepsilon(x,t)
		\approx u^{(i)}_0(x,t) + \varepsilon\,u^{(i)}_1(x,t) + \ldots
	\end{equation}
	Since $\Omega_\varepsilon^{(i)}$ transforms into $\Omega^{(i)}$ as $\varepsilon\to0$, and using Taylor’s formula for the functions $F_i$ and $\Psi^{(i,j)}$, we obtain the relations satisfied by the leading term of~\eqref{exp-1}:
	\begin{equation}\label{relations+}
		\left\{
		\begin{array}{rcll}
			\partial_t u^{(i)}_0 - D_i \Delta_x u^{(i)}_0
			+ F_i(u^{(i)}_0,x,t) &=& f_i(x,t)
			& \text{in }\ \Omega^{(i)}\times(0,T),
			\\[2mm]
			u^{(i)}_0 &=& 0
			& \text{on } \ \ \mathcal{B}_i\times(0,T),
			\\[2mm]
			D_i\nabla_x u^{(i)}_0\cdot\boldsymbol{\nu}_0
			&=& \Psi^{(i,i)}\big(u^{(i)}_0,w^{(i)}_0,y^{(i)}_1,t\big)
			& \text{on } \ \mathcal{I}_i\times(0,T),
			\\[2mm]
			D_i\nabla_x u^{(i)}_0\cdot\boldsymbol{\nu}_0
			&=& \Psi^{(i,j)}\big(u^{(i)}_0,w^{(j)}_0,y^{(j)}_1,t\big)
			& \text{on } \ \mathcal{I}_j\times(0,T),
			\\[2mm]
			u^{(i)}_0(x,0) &=& 0 & \text{in} \ \ \Omega^{(i)}.
		\end{array}
		\right.
	\end{equation}
	Here $\mathcal{I}_i$ and $\mathcal{I}_j$ are the edges of the graph
	$\mathcal{I}$ lying on $\partial\Omega^{(i)}$ (see Remark~\ref{Rem-1-4}), 
	$\mathcal{B}_i := \partial\Omega^{(i)}\setminus(\mathcal{I}_i\cup\mathcal{I}_j),$ 
	$\boldsymbol{\nu}_0$ denotes
	the inward unit normal to $\partial\Omega^{(i)}$, and $w^{(i)}_0$ is the leading terms in the asymptotic expansion of
	$w_\varepsilon$ on the edge $\mathcal{I}_i.$ The coefficient $u^{(i)}_1$ is determined in \S~\ref{corner-2-4}.
	
	
	\subsection{Analysis in the thin fracture $\mathcal{R}_\varepsilon$}\label{Par-3-1}
	
	We begin with a simple observation.  
	Let \(u:\mathbb{R}^2\to\mathbb{R}\) be a scalar field and
	\(\bm{\mathcal V}:\mathbb{R}^2\to\mathbb{R}^2\) a vector field.  
	For each \(i\in\{1,2,3\}\) define the transformed functions
	\[
	u^{(i)}(y):=u(\mathbb{A}_i^{-1}y),\qquad
	\bm{\mathcal V}^{(i)}(y):=\mathbb{A}_i\,\bm{\mathcal V}(\mathbb{A}_i^{-1}y).
	\]

	Since the matrix \(\mathbb{A}_i\) is orthogonal (see~\eqref{matrix}), the
	standard differential operators are invariant under this orthonormal change of
	coordinates. In particular,
	\begin{equation}\label{lap_1}
		\Delta_x u(x)
		=\frac{\partial^2 u}{\partial x_1^2}
		+\frac{\partial^2 u}{\partial x_2^2}
		=\frac{\partial^2 u^{(i)}}{\partial (y_1^{(i)})^2}
		+\frac{\partial^2 u^{(i)}}{\partial (y_2^{(i)})^2}
		=\Delta_{y^{(i)}}u^{(i)}(y^{(i)}),
	\end{equation}
	\begin{equation}\label{div_1}
		\mathrm{div}_x\bm{\mathcal V}(x)
		=\frac{\partial v_1}{\partial x_1}
		+\frac{\partial v_2}{\partial x_2}
		=\frac{\partial v^{(i)}_1}{\partial y_1^{(i)}}
		+\frac{\partial v^{(i)}_2}{\partial y_2^{(i)}}
		=\mathrm{div}_{y^{(i)}}\bm{\mathcal V}^{(i)}(y^{(i)}),
	\end{equation}
	and for any vector \(\boldsymbol{\eta}=(\eta_1,\eta_2)^T\),
	\begin{equation}\label{grad_1}
		\boldsymbol{\eta}\cdot\nabla_x u(x)
		=\eta_1\frac{\partial u}{\partial x_1}
		+\eta_2\frac{\partial u}{\partial x_2}
		=\eta_1^{(i)}\frac{\partial u^{(i)}}{\partial y_1^{(i)}}
		+\eta_2^{(i)}\frac{\partial u^{(i)}}{\partial y_2^{(i)}}
		=\boldsymbol{\eta}^{(i)}\cdot\nabla_{y^{(i)}}u^{(i)}(y^{(i)}),
	\end{equation}
	where \(\boldsymbol{\eta}^{(i)}:=\mathbb{A}_i\boldsymbol{\eta}\).
	
	Using these identities, the differential equation for $w_\varepsilon$ in the
	thin branch $\mathcal{R}_\varepsilon^{(i)}$ takes the form
	\begin{equation}\label{eq-thin-rect}
		\partial_t w^{(i)}_\varepsilon
		+\mathrm{div}_{y^{(i)}}\!\big(w^{(i)}_\varepsilon\,\overrightarrow{V_\varepsilon}^{(i)}\big)
		=\varepsilon\,\Delta_{y^{(i)}} w^{(i)}_\varepsilon,
		\qquad
		\text{in }\mathcal{R}_\varepsilon^{(i)}\times(0,T),
	\end{equation}
	where $\overrightarrow{V_\varepsilon}^{(i)}$ is defined in~\eqref{conv-1}.
	
	Since the analysis in each thin branch $\mathcal{R}_\varepsilon^{(i)}$ is analogous, we restrict ourselves, for definiteness, to $\mathcal{R}_\varepsilon^{(1)}$.  
	The boundary conditions on the sides $\Gamma_\varepsilon^{(1,1)}$ and $\Gamma_\varepsilon^{(3,1)}$ of $\mathcal{R}_\varepsilon^{(1)}$ can be rewritten as
	\begin{gather}\label{lbc-1}
		-\Big(-\partial_{y^{(1)}_2}w^{(1)}_\varepsilon + w^{(1)}_\varepsilon v^{(1)}_2\Big)
		= \varepsilon\,
		\Phi^{(1,1)}\!\left(u^{(1)}_\varepsilon,w^{(1)}_\varepsilon,
		y^{(1)}_1,t\right)
		\quad\text{on }\Gamma_\varepsilon^{(1,1)}\times(0,T),
		\\[4pt]\label{lbc-2}
		\Big(-\partial_{y^{(1)}_2}w^{(1)}_\varepsilon + w^{(1)}_\varepsilon v^{(1)}_2\Big)
		= \varepsilon\,
		\Phi^{(3,1)}\!\left(u^{(3)}_\varepsilon,w^{(1)}_\varepsilon,
		y^{(1)}_1,t\right)
		\quad\text{on }\Gamma_\varepsilon^{(3,1)}\times(0,T).
	\end{gather}
	
	We use the two-scaled ansatz
	\begin{equation}\label{regul-1+}
		\mathfrak{W}^{(1)}_\varepsilon =
		w_0^{(1)}(y^{(1)}_1,t)
		+ \varepsilon\Big(w_1^{(1)}(y^{(1)}_1,t)
		+ z_1^{(1)}\big(y^{(1)}_1,\tfrac{y^{(1)}_2}{\varepsilon},t\big)\Big)
		+ \varepsilon^2 z_2^{(1)}\big(y^{(1)}_1,\tfrac{y^{(1)}_2}{\varepsilon},t\big)
		\quad\text{in }\mathcal{R}^{(1)}_\varepsilon.
	\end{equation}
	Substituting this into \eqref{eq-thin-rect} for $i=1$, collecting terms of order $\varepsilon^0$ and $\varepsilon$, and equating them to zero, we obtain differential equations in $\xi_2=\tfrac{y^{(1)}_2}{\varepsilon}\in(-h_1,h_1)$ for the coefficients $z_1^{(1)}$ and $z_2^{(1)}$:
	\begin{equation}\label{eq_1}
		\partial_{\xi_2}^2 z^{(1)}_1
		= \partial_t w^{(1)}_0
		+ \partial_{y^{(1)}_1}\!\big(v_1^{(1)} w^{(1)}_0\big)
		+ w^{(1)}_0\,\partial_{\xi_2}v_2^{(1)},
	\end{equation}
	\begin{equation}\label{eq-2}
		\partial_{\xi_2}^2 z^{(1)}_2
		= \partial_t w^{(1)}_1
		+ \partial_{y^{(1)}_1}\!\big(v_1^{(1)} w^{(1)}_1\big)
		+ \partial_t z^{(1)}_1
		+ \partial_{y^{(1)}_1}\!\big(v_1^{(1)} z^{(1)}_1\big)
		+ \partial_{\xi_2}\!\big(v_2^{(1)}(w^{(1)}_1 + z^{(1)}_1)\big)
		- \partial_{y^{(1)}_1}^2 w^{(1)}_0.
	\end{equation}
	
	Next, substituting the ansatz \eqref{regul-1+} into the boundary conditions \eqref{lbc-1}–\eqref{lbc-2} and using Taylor’s formula for $\Phi^{(1,1)}$ and $\Phi^{(3,1)}$, we obtain
	\begin{gather}\label{bc-1}
		-\big(-\partial_{\xi_2}z^{(1)}_1 + v_2^{(1)} w^{(1)}_0\big)=0,
		\qquad \xi_2=h_1,
		\\[2mm]\label{bc-2}
		-\partial_{\xi_2}z^{(1)}_1 + v_2^{(1)} w^{(1)}_0 = 0,
		\qquad \xi_2=-h_1,
	\end{gather}
	\begin{gather}\label{bc-3}
		-\big(-\partial_{\xi_2}z^{(1)}_2 + v_2^{(1)}(w^{(1)}_1 + z^{(1)}_1)\big)
		= \Phi^{(1,1)}(u^{(1)}_0,w^{(1)}_0,y^{(1)}_1,t),
		\qquad \xi_2=h_1,
		\\[2mm]\label{bc-4}
		-\partial_{\xi_2}z^{(1)}_2 + v_2^{(1)}(w^{(1)}_1 + z^{(1)}_1)
		= \Phi^{(3,1)}(u^{(3)}_0,w^{(1)}_0,y^{(1)}_1,t),
		\qquad \xi_2=-h_1.
	\end{gather}
	
	Thus, for each fixed $(y^{(1)}_1,t)$ we obtain linear Neumann problems on $(-h_1,h_1)$ for $z_1^{(1)}$ and $z_2^{(1)}$, namely \eqref{eq_1}, \eqref{bc-1}, \eqref{bc-2} and \eqref{eq-2}, \eqref{bc-3}, \eqref{bc-4}.  
	To ensure uniqueness, each problem is supplemented with the normalization
	\begin{equation}\label{uniq_1}
		\int_{-h_1}^{h_1} z_k^{(1)}(y^{(1)}_1,\xi_2,t)\,d\xi_2 = 0,
		\qquad k\in\{1,2\}.
	\end{equation}
	
	The solvability conditions for these Neumann problems yield the following
	equations for the coefficients $w^{(1)}_0$ and $w^{(1)}_1$:
	\begin{equation}\label{lim_0}
		\partial_t w^{(1)}_0
		+ \partial_{y^{(1)}_1}\!\big(v_1^{(1)} w^{(1)}_0\big)
		= 0,
		\qquad (y^{(1)}_1,t)\in\mathcal{I}_1\times(0,T),
	\end{equation}
	\begin{equation}\label{lim_1}
		\partial_t w^{(1)}_1
		+ \partial_{y^{(1)}_1}\!\big(v_1^{(1)} w^{(1)}_1\big)
		= \partial_{y^{(1)}_1}^2 w^{(1)}_0 + \widehat{\Phi}^{(1)},
		\qquad (y^{(1)}_1,t)\in\mathcal{I}_1\times(0,T),
	\end{equation}
	where
	\begin{equation}\label{fun_1}
		\widehat{\Phi}^{(1)}(y^{(1)}_1,t)
		:= \frac{1}{2h_1}\Big(
		\Phi^{(1,1)}(u^{(1)}_0,w^{(1)}_0,y^{(1)}_1,t)
		+ \Phi^{(3,1)}(u^{(3)}_0,w^{(1)}_0,y^{(1)}_1,t)
		\Big).
	\end{equation}
	
	Equations \eqref{lim_0}–\eqref{lim_1} are linear first‑order hyperbolic equations on the edge $\mathcal{I}_1$.  
	Analogous equations are obtained on the edges $\mathcal{I}_2$ and $\mathcal{I}_3$ of the graph $\mathcal{I}$.
	
	
	\subsubsection{Inner node-layer part of the approximation}\label{subsec_Inner_part}
	
	To derive the gluing conditions for the coefficients
	$\{w_0^{(i)}\}_{i=1}^3$ and $\{w_1^{(i)}\}_{i=1}^3$ at the origin (the vertex of the graph $\mathcal{I}$),
	we introduce, following the approach of
	\cite{Mel-Roh_AA-2024,Mel-Roh_JMAA-2024}, the inner node-layer ansatz
	\begin{equation}\label{junc}
		\mathfrak{N}_\varepsilon(x,t)
		= N_0\!\left(\frac{x}{\varepsilon},t\right)
		+ \varepsilon\,N_1\!\left(\frac{x}{\varepsilon},t\right)
	\end{equation}
	in a neighborhood of the node $\mathcal{R}^{(0)}_\varepsilon$.
	
	Passing to the scaled variables
	$
	\xi=\frac{x}{\varepsilon},
	$
	and then letting $\varepsilon\to0$, the thin graph-like junction
	$\mathcal{R}_\varepsilon$ transforms into the unbounded domain
	\[
	\Xi=\operatorname{int}\Big(\bigcup_{i=0}^3\overline{\Xi^{(i)}}\Big),
	\]
	where $\Xi^{(0)}$ is the node domain (see~\eqref{node}) and
	\[
	\Xi^{(i)}
	=\big\{\xi^{(i)}=(\xi^{(i)}_1,\xi^{(i)}_2)\in\mathbb{R}^2:\ 
	\ell_0<\xi^{(i)}_1<+\infty,\ |\xi^{(i)}_2|<h_i\big\},
	\qquad i\in\{1,2,3\},
	\]
	are semi-infinite strips written in the scaled coordinates
	$\xi^{(i)}=y^{(i)}/\varepsilon$.
	
	We denote the lateral sides of the strips by  
	$\Gamma^{(1,1)}$ and $\Gamma^{(3,1)}$ for $\Xi^{(1)}$,
	$\Gamma^{(1,2)}$ and $\Gamma^{(2,2)}$ for $\Xi^{(2)}$, and
	$\Gamma^{(2,3)}$ and $\Gamma^{(3,3)}$ for $\Xi^{(3)}$.
	The remaining parts of the boundary of $\Xi$ form the node boundary:
	\[
	\partial\Xi\setminus\Big(\bigcup_{i=1}^3\bigcup_j\Gamma^{(i,j)}\Big)
	=\Gamma^{(1,0)}\cup\Gamma^{(2,0)}\cup\Gamma^{(3,0)}.
	\]
	
	Substituting  \eqref{junc} into the differential equation
	\eqref{p-equations+} of problem $\mathbb{P}_\varepsilon$ and into the boundary
	conditions \eqref{coup-1}–\eqref{coupl-1+}, and collecting terms of equal
	powers of $\varepsilon$, we obtain the following problems:
	\begin{equation}\label{N_0_prob}
		\left\{
		\begin{array}{rcll}
			- \Delta_\xi N_0(\xi,t) + \overrightarrow{V}(\xi)\cdot\nabla_\xi N_0(\xi,t)
			&=& 0, & \xi\in\Xi^{(0)},
			\\[2mm]
			\partial_{\boldsymbol{\nu}_\xi} N_0(\xi,t)
			&=& 0, & \xi\in\Gamma^{(1,0)}\cup\Gamma^{(2,0)}\cup\Gamma^{(3,0)},
			\\[2mm]
			- \Delta_{\xi^{(i)}} N_0(\xi^{(i)},t)
			+ \mathrm{v}_i\,\partial_{\xi^{(i)}_1}N_0(\xi^{(i)},t)
			&=& 0, & \xi^{(i)}\in\Xi^{(i)},
			\\[2mm]
			\partial_{\xi^{(i)}_2}N_0(\xi^{(i)},t)\big|_{\xi^{(i)}_2=\pm h_i}
			&=& 0, & \xi^{(i)}\in\Gamma^{(i,j)},
			\\[2mm]
			N_0(\xi^{(i)},t)\sim w^{(i)}_0(0,t)
			&\text{as}& \xi^{(i)}_1\to+\infty, & \xi^{(i)}\in\Xi^{(i)},\ i\in\{1,2,3\},
		\end{array}
		\right.
	\end{equation}
	and
	\begin{equation}\label{N_1_prob}
		\left\{
		\begin{array}{rcll}
			- \Delta_\xi N_1(\xi,t) + \overrightarrow{V}(\xi)\cdot\nabla_\xi N_1(\xi,t)
			&=& -\partial_t N_0(\xi,t), & \xi\in\Xi^{(0)},
			\\[2mm]
			- \partial_{\boldsymbol{\nu}_\xi}N_1(\xi,t)
			&=& \Phi_0\big(u^{(i)}_0(0,t),N_0,\xi,t\big), & \xi\in\Gamma^{(i,0)},
			\\[2mm]
			- \Delta_{\xi^{(i)}} N_1(\xi^{(i)},t)
			+ \mathrm{v}_i\,\partial_{\xi^{(i)}_1}N_1(\xi^{(i)},t)
			&=& -\partial_t N_0(\xi^{(i)},t), & \xi^{(i)}\in\Xi^{(i)},
			\\[2mm]
			\partial_{\xi^{(i)}_2}N_1(\xi^{(i)},t)\big|_{\xi^{(i)}_2=\pm h_i}
			&=& 0, & \xi^{(i)}\in\Gamma^{(i,j)},
			\\[2mm]
			N_1(\xi^{(i)},t)\sim w^{(i)}_1(0,t)+\Theta^{(i)}_1(\xi^{(i)},t)
			&\text{as}& \xi^{(i)}_1\to+\infty, & \xi^{(i)}\in\Xi^{(i)},\ i\in\{1,2,3\}.
		\end{array}
		\right.
	\end{equation}	
	Here
	\[
	\overrightarrow{V}(\xi)
	=\overrightarrow{V}^{(0)}_\varepsilon(x)\big|_{x=\varepsilon\xi}
	\quad\text{in }\Xi^{(0)},\qquad
	\overrightarrow{V}(\xi^{(i)})=(\mathrm{v}_i,0)
	\quad\text{in }\Xi^{(i)},
	\]
	and
	\begin{equation}\label{Psi_1}
		\Theta^{(i)}_1(\xi^{(i)}_1,t)
		= \xi^{(i)}_1\,\frac{\partial w_0^{(i)}}{\partial y^{(i)}_1}(0,t),
	\end{equation}
	for $i\in\{1,2,3\}$.  
	The function $\Theta^{(i)}_1$ arises from matching with the regular expansion \eqref{regul-1+}.  
	Note that the variable $t$ enters only as a parameter in the steady advection–diffusion problems \eqref{N_0_prob} and \eqref{N_1_prob}.
	
	In each semi-strip $\Xi^{(i)}$ we look for solutions in the form
	\begin{equation}\label{new-solution_k}
		N_k(\xi^{(i)},t)
		= \big(w^{(i)}_k(0,t)+\Theta^{(i)}_k(\xi^{(i)}_1,t)\big)\,\chi_{\ell_0}(\xi^{(i)}_1)
		+ \widetilde{N}_k(\xi^{(i)},t),
		\qquad k\in\{0,1\},
	\end{equation}
	where $\Theta^{(i)}_0\equiv0$, and $\chi_{\ell_0}\in C^\infty(\mathbb{R})$ is a smooth cut-off function such that  
	\[
	0\le\chi_{\ell_0}\le1,\qquad
	\chi_{\ell_0}(s)=0\ \text{for }s\le2\ell_0,\qquad
	\chi_{\ell_0}(s)=1\ \text{for }s\ge3\ell_0.
	\]
	
	Consequently, the remainders combine into a solution on the whole domain $\Xi$, again denoted by $\widetilde{N}_k$, which must satisfy
	\begin{equation}\label{N_k_prob}
		\left\{
		\begin{array}{rcll}
			- \Delta_\xi \widetilde{N}_k + \overrightarrow{V}(\xi)\cdot\nabla_\xi \widetilde{N}_k
			&=& \mathfrak{J}^{(0)}_k, & \xi\in\Xi^{(0)},
			\\[2mm]
			-\partial_{\boldsymbol{\nu}_\xi}\widetilde{N}_k
			&=& \mathfrak{F}^{(i)}_k, & \xi\in\Gamma^{(i,0)},
			\\[2mm]
			- \Delta_{\xi^{(i)}}\widetilde{N}_k
			+ \mathrm{v}_i\,\partial_{\xi^{(i)}_1}\widetilde{N}_k
			&=& \mathfrak{G}^{(i)}_k, & \xi^{(i)}\in\Xi^{(i)},
			\\[2mm]
			\partial_{\xi^{(i)}_2}\widetilde{N}_k\big|_{\xi^{(i)}_2=\pm h_i}
			&=& 0, & \xi^{(i)}\in\Gamma^{(i,j)},
			\\[2mm]
			\widetilde{N}_k(\xi^{(i)},t)\longrightarrow0
			&\text{as}& \xi^{(i)}_1\to+\infty, & \xi^{(i)}\in\Xi^{(i)},\ i\in\{1,2,3\}.
		\end{array}
		\right.
	\end{equation}
	Here
	\[
	\mathfrak{J}^{(0)}_0\equiv0,\qquad
	\mathfrak{J}^{(0)}_1=-\partial_t N_0,
	\qquad
	\mathfrak{F}^{(i)}_0\equiv0,\qquad
	\mathfrak{F}^{(i)}_1=\Phi_0\big(u^{(i)}_0(0,t),N_0,\xi,t\big),
	\]
	\[
	\mathfrak{G}^{(i)}_0
	= w^{(i)}_0(0,t)\,\chi''_{\ell_0}(\xi^{(i)}_1)
	- \mathrm{v}_i\,w^{(i)}_0(0,t)\,\chi'_{\ell_0}(\xi^{(i)}_1),
	\]
	and
	\begin{align*}
		\mathfrak{G}^{(i)}_1
		= & -\partial_t N_0
		- \mathrm{v}_i\big(\Theta^{(i)}_1(\xi^{(i)}_1,t)\chi_{\ell_0}(\xi^{(i)}_1)\big)'
		\\
		& + w^{(i)}_1(0,t)\,\chi''_{\ell_0}(\xi^{(i)}_1)
		- \mathrm{v}_i\,w^{(i)}_1(0,t)\,\chi'_{\ell_0}(\xi^{(i)}_1)
		+ \big(\Theta^{(i)}_1(\xi^{(i)}_1,t)\chi_{\ell_0}(\xi^{(i)}_1)\big)''.
	\end{align*}
	
	We seek $\widetilde{N}_k$ in the weighted space $\mathcal{H}_\beta$ of
	exponentially decaying functions from $H^1(\Xi)$, equipped with the norm
	\[
	\|u\|_{\beta}
	:=\Bigg(
	\int_{\Xi^{(0)}}\big(|\nabla_\xi u|^2+|u|^2\big)\,d\xi
	+\sum_{i=1}^3\int_{\Xi^{(i)}}\varrho_\beta(\xi^{(i)})
	\big(|\nabla_{\xi^{(i)}}u|^2+|u|^2\big)\,d\xi^{(i)}
	\Bigg)^{1/2},
	\]
	where $\beta>0$ and $\varrho_\beta$ is a smooth positive function such that
	\[
	\varrho_\beta =
	\begin{cases}
		1, & \xi\in\Xi^{(0)},
		\\[2pt]
		e^{\beta\xi^{(i)}_1}, & \xi^{(i)}_1\ge2\ell_0,\ \xi^{(i)}\in\Xi^{(i)},\ i\in\{1,2,3\}.
	\end{cases}
	\]
	
	As in \cite[Lemma 3.1]{Mel-Kle-2022}, one shows that if
	$\mathfrak{J}_k^{(0)}\in L^2(\Xi^{(0)})$, $\mathfrak{F}_k^{(i)}\in L^2(\Gamma^{(i,0)})$, and
	\[
	\int_{\Xi^{(i)}} e^{\beta\xi^{(i)}_1}\big(\mathfrak{G}^{(i)}_k\big)^2\,d\xi^{(i)}<+\infty,
	\qquad i\in\{1,2,3\},
	\]
	for some $\beta>0$, then there exists $\beta_0>0$ such that problem \eqref{N_k_prob}
	has a unique solution in $H_{\beta_0}$ if and only if the solvability condition
	\begin{equation}\label{solvability-cond}
		\int_{\Xi^{(0)}}\mathfrak{J}^{(0)}_k\,d\xi
		+ \sum_{i=1}^3\int_{\Xi^{(i)}}\mathfrak{G}^{(i)}_k\,d\xi^{(i)}
		= \sum_{i=1}^3\int_{\Gamma^{(i,0)}}\mathfrak{F}^{(i)}_k\,dl_\xi
	\end{equation}
	is satisfied.
	
	The solvability condition for \eqref{N_k_prob} with $k=0$ takes the form 
	\begin{equation}\label{solv-cond}
		\sum_{i=1}^{3}\int_{\Xi^{(i)}} \mathfrak{G}^{(i)}_0 \, d\xi^{(i)} = 0 \ \ \Longleftrightarrow \ \
		2 \sum_{i=1}^{3} \mathrm{v}_i  \, h_i \, w^{(i)}_0(0,t)   = 0.
	\end{equation}
	This equality provides the coupling condition at the vertex of the graph $\mathcal{I}$ for the leading terms $\{w^{(i)}_0\}_{i=1}^3$ of the regular asymptotics in the thin branch $\mathcal{R}^{(i)}_\varepsilon,$ $i\in\{1,2,3\},$ respectively. 
	
	\subsubsection{Determination of the leading-order terms} Thus, the functions $\{w^{(i)}_0\}_{i=1}^3$ satisfy  the following first-order hyperbolic problem on the graph $\mathcal{I}$:
	\begin{equation}\label{limit_prob}
		\left\{\begin{array}{rcll}
			\partial_t{w}^{(i)}_0(y_1^{(i)},t) + \partial_{y_1^{(i)}}\big( v_1^{(i)}(y_1^{(i)})\,  w^{(i)}_0(y_1^{(i)},t) \big) &=& 0,& (y_1^{(i)}, t) \in \mathcal{I}_i \times (0, T), \ \ i \in \{1,2,3\},
			\\[2pt]
			\sum_{i=1}^{3}  \mathrm{v}_i \, h_i\,  w_0^{(i)}(0, t) & = & 0,& t \in (0, T),
			\\[2mm]
			w_0^{(1)}(\ell_1,  t) & =  & q_1(t), & t \in [0, T], 
			\\[2mm]
			w_0^{(i)}(y_1^{(i)},  0) & = & 0, & y_1^{(i)} \in [0, \ell_i],  \ \ i \in \{1,2,3\}.
		\end{array}\right.
	\end{equation}
	
	Problems of this type were studied in \cite[Sect.~3]{Mel-Roh_AA-2024} and 
	\cite[\S 3.3]{Mel-Roh_JMAA-2024}.  
	Under our assumptions on the advection field and condition~\eqref{cond_1}, 
	problem~\eqref{limit_prob} admits a unique solution that is continuous at the vertex,
	that is, $w_0^{(1)}(0, t)=w_0^{(2)}(0, t)=w_0^{(3)}(0, t).$ 
	
	Moreover, explicit formulas for $w^{(i)}_0$ can be obtained by integrating the 
	corresponding mixed problems along characteristic curves.
	For the first edge $\mathcal{I}_1$, we solve the mixed problem
	\begin{equation}\label{limit_prob-1a}
		\left\{
		\begin{array}{l}
			\partial_t w^{(1)}_0(y_1^{(1)},t)
			+ \partial_{y_1^{(1)}}\!\big(v_1^{(1)}(y_1^{(1)})\, w^{(1)}_0(y_1^{(1)},t)\big)
			= 0,
			\quad (y_1^{(1)},t)\in\mathcal{I}_1\times(0,T),
			\\[2mm]
			w^{(1)}_0(\ell_1,t)=q_1(t), \ \ t \in [0,T], \qquad
			w^{(1)}_0(y_1^{(1)},0)=0,\ \ y_1^{(1)}\in[0,\ell_1].
		\end{array}
		\right.
	\end{equation}
	Integrating along characteristics and incorporating the boundary and initial data 
	yields an explicit representation of $w^{(1)}_0$ on $\mathcal{I}_1\times[0,T]$.
	For the remaining edges $\mathcal{I}_2$ and $\mathcal{I}_3$, we similarly solve
	\begin{equation}\label{limit_prob-1b}
		\left\{
		\begin{array}{l}
			\partial_t w^{(i)}_0(y_1^{(i)},t)
			+ \partial_{y_1^{(i)}}\!\big(v_1^{(i)}(y_1^{(i)})\, w^{(i)}_0(y_1^{(i)},t)\big)
			= 0,
			\quad (y_1^{(i)},t)\in\mathcal{I}_i\times(0,T),
			\\[2mm]
			w^{(i)}_0(0,t)=w^{(1)}_0(0,t), \ \ t \in [0,T],\qquad
			w^{(i)}_0(y_1^{(i)},0)=0,\ \ y_1^{(i)}\in[0,\ell_i]; \quad i \in\{2,3\}.
		\end{array}
		\right.
	\end{equation}
	The resulting formulas are lengthy and therefore omitted here; see \cite[Sect.~3]{Mel-Roh_AA-2024} for details.
	
	Having the solution to problem \eqref{limit_prob} implies the following:
	\begin{enumerate}
		\item 
		the solvability condition~\eqref{solv-cond} is satisfied, 
		and continuity at the vertex implies $N_0 \equiv w_0^{(1)}(0, t);$ 
		\item 
		the leading term $u_0^{(i)}$ of the bulk expansion~\eqref{exp-1} in the
		bulk domain $\Omega^{(i)}$ is uniquely determined as the weak solution of the
		corresponding initial--boundary value problem in $\Omega^{(i)}$ (for
		$u_0^{(1)}$, see problem~\eqref{relations+});
		\item 
		the functions $\{z_1^{(i)}\}_{i=1}^3$ in the regular expansion in  $\mathcal{R}^{(i)}_\varepsilon$ are uniquely determined as
		solutions of the associated Neumann problems on the rescaled cross-sections.
		For example, the Neumann problem \eqref{eq_1}, \eqref{bc-1}, \eqref{bc-2}
		takes the form
		\begin{equation}\label{prob-z-1}
			\left\{
			\begin{array}{l}
				\partial_{\xi_2^2}^2 z^{(1)}_1\big(y^{(1)}_1,\xi_2,t\big)
				= w^{(1)}_0(y^{(1)}_1,t)\,\partial_{\xi_2} v_2^{(1)}(y^{(1)}_1,\xi_2),
				\qquad \xi_2\in(-h_1,h_1),
				\\[2mm]
				-\partial_{\xi_2} z^{(1)}_1(y^{(1)}_1,\xi_2,t)
				+ v_2^{(1)}(y^{(1)}_1,\xi_2)\, w^{(1)}_0(y^{(1)}_1,t)
				=0,
				\qquad \xi_2=\pm h_1,
				\\[2mm]
				\displaystyle \int_{-h_1}^{h_1} z_1^{(1)}(y^{(1)}_1,\xi_2,t)\,d\xi_2 = 0.
			\end{array}
			\right.
		\end{equation}
	\end{enumerate}
	
	Thus, the leading terms of the asymptotic expansions in all parts of the
	domain $\Omega$ are completely determined.

	\subsubsection{Determination of the second-order terms of the asymptotics}
	We now turn to the determination of the second-order coefficients $\{w^{(i)}_1\}_{i=1}^3$ and the corresponding node-layer corrections.  
	To this end, we first rewrite the right-hand sides $\{\mathfrak{G}^{(i)}_1\}_{i=1}^3$ in problem~\eqref{N_k_prob} (with $k=1$) using the already determined functions $\{w^{(i)}_0\}_{i=1}^3$:
	\begin{align} \label{G1}
		\mathfrak{G}^{(i)}_1 = & - \partial_t w^{(i)}_0(0,t) \big(1 - \chi_{\ell_0}(\xi^{(i)}_1)\big) - \mathrm{v}_i\, \Theta^{(i)}_{1}(\xi^{(i)}_1,t) \,\chi'_{\ell_0}(\xi^{(i)}_1) \notag
		\\
		& + w^{(i)}_1(0,t) \, \chi''_{\ell_0}(\xi^{(i)}_1) 
		- \mathrm{v}_i  \, w^{(i)}_1(0,t) \, \chi'_{\ell_0}(\xi^{(i)}_1) + \Big(\Theta^{(i)}_{1}(\xi^{(i)}_1,t) \,\chi_{\ell_0}(\xi^{(i)}_1)\Big)'' .
	\end{align}
	Each $\mathfrak{G}^{(i)}_1$ has compact support in $\Xi^{(i)}$, which simplifies the solvability condition.
	Using these representations, the solvability condition~\eqref{solvability-cond} for $k=1$ becomes
	\[
	\int_{\Xi^{(0)}} \partial_t w^{(1)}_0(0,t) \, d\xi
	\;+\;
	\sum_{i=1}^{3}\int_{\Xi^{(i)}} \mathfrak{G}^{(i)}_1 \, d\xi^{(i)}
	\;=\;
	\sum_{i=1}^{3}\int_{\Gamma^{(i,0)}}
	\Phi_0\big(u^{(i)}_0, w^{(1)}_0(0,t), \xi,t\big)\, dl_\xi.
	\]
	Evaluating the integrals $\int_{\Xi^{(i)}} \mathfrak{G}^{(i)}_1\,d\xi^{(i)}$ yields $2h_i\,\mathrm{v}_i\,w^{(i)}_1(0,t)$, and therefore the solvability condition reduces to
	\[
	\sum_{i=1}^{3} 2 h_i\,\mathrm{v}_i\, w_1^{(i)}(0,t)
	= {\bf d}_1(t),
	\]
	where
	\begin{equation}\label{d_1}
		{\bf d}_1(t)
		:= w^{(1)}_0(0,t)\, |\Xi^{(0)}|
		- \sum_{i=1}^{3}\int_{\Gamma^{(i,0)}}
		\Phi_0\big(u^{(i)}_0, w^{(1)}_0(0,t), \xi,t\big)\, dl_\xi.
	\end{equation}
	Here, $|\Xi^{(0)}|$ denotes the Lebesgue measure of the set $\Xi^{(0)}.$
	
	Thus, the functions $\{w^{(i)}_1\}_{i=1}^3$ satisfy the  mixed hyperbolic problem
	\begin{equation}\label{prob_w_1}
		\left\{
		\begin{array}{rcll}
			\partial_t w^{(i)}_1
			+ \partial_{y_1^{(i)}}\!\big(v_i^{(i)}\, w^{(i)}_1\big)
			&=&
			\partial_{y^{(i)}_1}^2 w^{(i)}_0
			+ \widehat{\Phi}^{(i)},
			& (y_1^{(i)},t)\in\mathcal{I}_i\times(0,T),
			\\[2mm]
			\displaystyle \sum_{i=1}^{3} 2 h_i\,\mathrm{v}_i\, w_1^{(i)}(0,t)
			&=& {\bf d}_1(t),
			& t\in(0,T),
			\\[2mm]
			w_1^{(1)}(\ell_1,t) &=& 0,
			& t\in[0,T],
			\\[2mm]
			w_1^{(i)}(y_1^{(i)},0) &=& 0,
			& y_1^{(i)}\in[0,\ell_i],\ i=1,2,3.
		\end{array}
		\right.
	\end{equation}
	Here $\widehat{\Phi}^{(1)}$ is given in~\eqref{fun_1}, and
	\begin{gather}\label{fun_2}
		\widehat{\Phi}^{(2)}(y^{(2)}_1,t)
		:= \frac{1}{2h_2}\Big(
		\Phi^{(1,2)}(u^{(1)}_0,w^{(2)}_0,y^{(2)}_1,t)
		+
		\Phi^{(2,2)}(u^{(2)}_0,w^{(2)}_0,y^{(2)}_1,t)
		\Big),
		\\[2mm]
		\label{fun_3}
		\widehat{\Phi}^{(3)}(y^{(3)}_1,t)
		:= \frac{1}{2h_3}\Big(
		\Phi^{(2,3)}(u^{(2)}_0,w^{(3)}_0,y^{(3)}_1,t)
		+
		\Phi^{(3,3)}(u^{(3)}_0,w^{(3)}_0,y^{(3)}_1,t)
		\Big).
	\end{gather}
	
	Using the approach of \cite[Sect.~3]{Mel-Roh_AA-2024} and \cite[\S\,3.3]{Mel-Roh_JMAA-2024}, one obtains that problem~\eqref{prob_w_1} admits a unique solution.  
	Its explicit representation follows from integrating the corresponding mixed problems along characteristic curves.  
	To satisfy the vertex condition, the mixed problems for $w^{(2)}_1$ and $w^{(3)}_1$ must be solved with boundary data
	\[
	w^{(i)}_1(0,t)
	=
	\frac{1}{4 h_i\,\mathrm{v}_i}
	\Big(
	{\bf d}_1(t) - 2 h_1\,\mathrm{v}_1\, w_1^{(1)}(0,t)
	\Big),
	\qquad t\in[0,T],\ i\in\{2,3\}.
	\]

	Consequently, both problem~\eqref{N_k_prob} (with $k=1$) and
	problem~\eqref{N_1_prob} have unique solutions, and
	\begin{equation}\label{exp-decrease+1}
		N_1(\xi,t)
		=
		w^{(i)}_1(0,t)
		+
		\Theta^{(i)}_1(\xi^{(i)}_1,t)
		+
		\mathcal{O}\!\left(e^{-\beta_0 \xi^{(i)}_1}\right)
		\quad\text{as } \quad \xi^{(i)}_1\to+\infty,
		\quad \xi\in\Xi^{(i)},\ i=1,2,3,
	\end{equation}
	for some $\beta_0>0.$ Thus, the inner node-layer ansatz \eqref{junc} becomes 
	\begin{equation}\label{junc+1}
		\mathfrak{N}_\varepsilon(x,t)
		= w_0^{(1)}(0,t) + \varepsilon\,N_1\!\left(\frac{x}{\varepsilon},t\right).
	\end{equation}

	Finally, the coefficients $\{z_2^{(i)}\}_{i=1}^3$ are uniquely determined as
	solutions of the corresponding Neumann problems on the rescaled
	cross-sections.  
	For instance, the problem \eqref{eq-2}, \eqref{bc-3}, \eqref{bc-4} becomes
	\begin{equation}\label{prob-z-2}
		\left\{
		\begin{array}{ll}
			\partial_{\xi_2^2}^2 z^{(1)}_2
			=
			\widehat{\Phi}^{(1)}
			+ \partial_t z^{(1)}_1
			+ \partial_{y^{(1)}_1}\!\big(v_1^{(1)} z^{(1)}_1\big)
			+ \partial_{\xi_2}\!\big(v_2^{(1)}(w^{(1)}_1 + z^{(1)}_1)\big),
			& \xi_2\in(-h_1,h_1),
			\\[2mm]
			-\partial_{\xi_2} z^{(1)}_2
			+ v_2^{(1)}(w^{(1)}_1 + z^{(1)}_1)
			= \Phi^{(1,1)}(u^{(1)}_0,w^{(1)}_0,y^{(1)}_1,t),
			& \xi_2=h_1,
			\\[2mm]
			-\partial_{\xi_2} z^{(1)}_2
			+ v_2^{(1)}(w^{(1)}_1 + z^{(1)}_1)
			= \Phi^{(3,1)}(u^{(3)}_0,w^{(1)}_0,y^{(1)}_1,t),
			& \xi_2=-h_1,
			\\[2mm]
			\displaystyle \int_{-h_1}^{h_1} z_2^{(1)}(y^{(1)}_1,\xi_2,t)\,d\xi_2 = 0.
			&
		\end{array}
		\right.
	\end{equation}
	
	Determining all terms of the ansatz \eqref{regul-1+} allows us to compute the
	residuals it leaves in equation~\eqref{eq-thin-rect}. We obtain
	\begin{equation}\label{eq-thin-rect+}
		\partial_t \mathfrak{W}^{(1)}_\varepsilon
		+ \mathrm{div}_{y^{(1)}}\!\big(\mathfrak{W}^{(1)}_\varepsilon\,\overrightarrow{V_\varepsilon}^{(1)}\big)
		- \varepsilon\,\Delta_{y^{(1)}}\mathfrak{W}^{(1)}_\varepsilon
		= \varepsilon^2\,\digamma^{(1)}_\varepsilon
		\quad\text{in }\mathcal{R}_\varepsilon^{(1)}\times(0,T),
	\end{equation}
	where
	\begin{equation}\label{digamma}
		\digamma^{(1)}_\varepsilon\big(y^{(1)}_1, \tfrac{{y}^{(1)}_2}{\varepsilon}, t\big)
		:= \partial_t z_2^{(1)} + \partial_{y^{(1)}_1}\big(v_1^{(1)} \, z^{(1)}_2\big) + 
		\partial_{\xi_2}\big(v_2^{(1)} \,z^{(1)}_2\big) - \partial^2_{y^{(1)}_1}w^{(1)}_1
		- \partial^2_{y^{(1)}_1}z^{(1)}_1 .
	\end{equation}
	
	\begin{remark}\label{rem-2-2}
		From \eqref{eq-thin-rect+}–\eqref{digamma} it follows that the functions
		$\{w^{(i)}_1\}_{i=1}^3$ must be of class $C^2$, that is,
		$w^{(i)}_1\in C^2([0,\ell_i]\times[0,T])$.  
		In particular, this regularity holds provided the right-hand sides in the
		differential equations of problem~\eqref{prob_w_1} are $C^2$–smooth; see, for
		example, \cite[Th.~5.4]{Friedrichs-1948} for one sufficient condition.  
		Consequently, the leading-order terms $\{w^{(i)}_0\}_{i=1}^3$ must be of class
		$C^4$ on the edges of the graph, and the bulk solutions $u^{(i)}_0$ to
		problem~\eqref{relations+} must possess the corresponding regularity for each
		$i\in\{1,2,3\}$.
		
		Using the smoothness of the advection field, assumption~\eqref{q_1}, and the
		explicit representation of the solution obtained by integrating the equations in
		problem~\eqref{limit_prob} along the characteristic curves and incorporating the
		boundary and initial conditions, one verifies that the functions $\{w^{(i)}_0\}$
		indeed possess this regularity; see \cite[\S 3.2.1]{Mel-Roh_AA-2024} and
		\cite[\S 3.3]{Mel-Roh_JMAA-2024} for details.
		
		The regularity of the bulk solutions $u^{(i)}_0$ is addressed in the next
		subsection.
	\end{remark}
	
	\subsection{Regularity of the bulk solutions $\{u_0^{(i)}\}$}
	Since the functions $\{w^{(i)}_0\}$ are $C^4$-smooth,  the right-hand sides
	of the nonlinear Robin conditions on $\mathcal{I}_i$ and $\mathcal{I}_j$ in
	problem~\eqref{relations+} are $C^2$-smooth in all variables due to  assumptions~${\bf A4}.$
	We now restrict our consideration to problem~\eqref{relations+} with $i=1$ and establish the following statement.

	\begin{proposition}[Solvability and regularity]\label{prop-2-1}
		Under assumptions $\mathbf{A1}$ and $\mathbf{A4}$, 
		problem~\eqref{relations+} for $i=1$ admits a unique weak solution 
		possessing $W^{2,1}_2$--regularity and Hölder regularity
		\begin{equation}\label{hoelder-0}
			u\in H^{\mu,\mu/2}\big(\overline{\Omega^{(1)}}\times[0,T]\big)
			\qquad\text{for every }\mu< \frac{\pi}{\theta_1},
		\end{equation}
		where $\frac{\pi}{\theta_1}>1$ is the corner exponent at the origin.
		In addition, for  sufficiently small $\delta>0$,
		\begin{equation}\label{hoelder}
			u\in H^{2+\mu,\,1+\mu/2}\big((\overline{\Omega^{(1)}}\setminus B_\delta(0))\times[0,T]\big)
			\qquad\text{for all }\mu\in(0,1).
		\end{equation}
	\end{proposition}
	\begin{proof}
		Since $F_1$, $\Psi^{(1,1)}$ and $\Psi^{(1,2)}$ are globally Lipschitz in $u$, 
		the variational formulation is well posed in the space
		\[
		u\in L^2\big(0,T; H^1(\Omega^{(1)};\mathcal{B}_1)\big),
		\qquad 
		\partial_t u\in L^2\big(0,T; H^{-1}(\Omega^{(1)};\mathcal{B}_1)\big),
		\]
		where $H^1(\Omega^{(1)};\mathcal{B}_1) := \{v\in H^1(\Omega^{(1)}): v|_{\mathcal{B}_1}=0\}$ and $H^{-1}(\Omega^{(1)};\mathcal{B}_1)$ is the dual space to $H^1(\Omega^{(1)};\mathcal{B}_1).$
		A standard Galerkin approximation combined with Grönwall’s inequality yields existence and uniqueness of a weak solution.
		
		\smallskip
		\noindent\textsf{$W^{2,1}_2$--regularity.}
		Near the origin, the boundary of $\Omega^{(1)}$ consists only of the sides 
		$\mathcal I_1$ and $\mathcal I_2$.  
		By assumption, the fluxes $\Psi^{(1,1)}$ and $\Psi^{(1,2)}$ vanish in a neighbourhood 
		of the origin as well as the source term $f_1$ and the reaction term
		$F_1$, so the local model is the \emph{Neumann problem in a sector} of opening 
		angle $\theta_1\in(\frac{\pi}{2},\pi)$ (see \S\ref{subsec-1-1}).  
		According to the classical theory of elliptic problems in angular domains 
		\cite{Kondratev1967,Grisvard1985,Dauge1988,Maz-Naz-Pla-2012}, the corner singular 
		exponents are
		\[
		\lambda_k=\frac{k\pi}{\theta_1},\qquad k=0,1,2,\dots,
		\]
		with $\lambda_0=0$ and $\lambda_k>1$ for all $k\ge1$.  
		Hence all second derivatives of solutions to the elliptic Neumann problem belong to 
		$L^2$ near the origin.
		
		The other two corner points of $\Omega^{(1)}$ are 
		$\overline{\mathcal{B}_1}\cap\overline{\mathcal{I}_1}$ and 
		$\overline{\mathcal{B}_1}\cap\overline{\mathcal{I}_2}$.  
		These are right angles with Dirichlet data on $\mathcal{B}_1$ and Neumann data on 
		$\mathcal{I}_j$.  
		Using odd reflection across the Dirichlet boundary, each such corner becomes a 
		\emph{smooth Neumann boundary point} for the extended problem, and therefore does not 
		produce any singularity in $u$.
		
		Assumptions \eqref{comp-cond-Psi} guarantee full compatibility of the Neumann boundary 
		conditions at $t=0$ and prevent the formation of an initial boundary layer.  
		Therefore, the classical parabolic $L^2$--regularity theory 
		\cite{LSU,Lions-Magenes} yields
		\[
		u\in W^{2,1}_2(\Omega^{(1)}\times(0,T)).
		\]
		
		\smallskip
		\noindent\textsf{Hölder regularity.}
		Solutions of the elliptic Neumann problem in a sector admit the asymptotic expansion
		\begin{equation}\label{asym-in-coner}
			u(r,\phi)=c_0 + c_1 r^{\lambda_1}\varPhi_1(\phi) + \cdots,
			\qquad \lambda_1=\frac{\pi}{\theta_1}>1,
		\end{equation}
		(see \cite{Kondratev1967,Grisvard1985,Dauge1988,Maz-Naz-Pla-2012}; here $(r,\phi)$ are the polar coordinates),  which implies the optimal spatial regularity
		\[
		u(\cdot,t)\in C^{1,\beta}(\overline{\Omega^{(1)}}),
		\qquad 
		\beta=\lambda_1-1=\frac{\pi}{\theta_1}-1>0.
		\]
		
		At the two corner points 
		$\overline{\mathcal{B}_1}\cap\overline{\mathcal{I}_1}$ and 
		$\overline{\mathcal{B}_1}\cap\overline{\mathcal{I}_2}$, 
		the above reflection argument shows that the solution behaves as in a smooth Neumann 
		boundary situation; in particular, no singular exponents appear there.  
		Thus, away from the vertex at the origin, the boundary is effectively smooth for the 
		extended problem, and the standard parabolic Schauder theory 
		(see \cite[Ch.~IV]{LSU}, \cite[Ch.~IV]{Lieberman}) applies up to these points.  
		We obtain	
		\[
		u\in H^{\mu,\mu/2}\big(\overline{\Omega^{(1)}}\times[0,T]\big)
		\qquad\text{for every }\mu<1+\beta = \frac{\pi}{\theta_1}.
		\]
		
		Finally, since $\Omega^{(1)}\setminus B_\delta(0),$ where $B_\delta(0) = \{x\colon \ |x| < \delta\}$ and
		$\delta$ is chosen such that all the supports of the functions $F_1,$ $f_1,$ $\Psi^{(i,j)}$ belong to $\overline{\Omega}^{(1)}\setminus B_\delta(0),$ the Schauder estimates give \eqref{hoelder}.	
	\end{proof}

	\begin{remark}\label{rem-2-3}
		Assumption~${\bf A4}$ together with the bulk solution regularity
		\eqref{hoelder} ensures that the functions $\{\widehat{\Phi}^{(i)}\}$ appearing
		in the right-hand sides of problem~\eqref{prob_w_1} are compactly supported in
		the longitudinal variables on their respective edges of the graph and are
		$C^2$–smooth on these supports.  
		The compact supports, combined with assumptions~\eqref{com-con-Phi}, are
		essential for satisfying the compatibility conditions (see
		\cite[\S 3.2.1]{Mel-Roh_AA-2024} for a detailed discussion).  
		Consequently, the correctors $\{w^{(i)}_1\}_{i=1}^3$ inherit the required
		$C^2$–regularity.
		
		Furthermore, the smoothness of $\{w^{(i)}_0\}_{i=1}^3$ and
		$\{w^{(i)}_1\}_{i=1}^3$ implies that the coefficients
		$\{z^{(i)}_1\}_{i=1}^3$ and $\{z^{(i)}_2\}_{i=1}^3$ in the ansatz
		\eqref{regul-1+} possess the necessary regularity.  
		Since the transversal components of the advection field and the functions
		$\{\widehat{\Phi}^{(i)}\}$ and $\{\Phi^{(i,j)}\}$ are compactly supported in
		the longitudinal variables on their respective edges, the coefficients
		$\{z^{(i)}_1\}$ and $\{z^{(i)}_2\}$ inherit the same compact supports in the
		corresponding longitudinal variables $\{y^{(i)}_1\}_{i=1}^3$.
		
		As a consequence, the function $\digamma_\varepsilon^{(1)}$ defined in
		\eqref{digamma} is uniformly bounded, namely,
		\begin{equation}\label{uniform-est-1}
			\sup_{\mathcal{R}_\varepsilon^{(1)}\times(0,T)}
			\left|\digamma_\varepsilon^{(1)}\!\left(y^{(1)}_1,
			\tfrac{y^{(1)}_2}{\varepsilon}, t\right)\right| \le C_1.
		\end{equation}
	\end{remark}

	\begin{remark}
		Throughout the paper, all constants in the inequalities are positive and
		independent of the parameter $\varepsilon$.  
		Constants carrying the same index in different inequalities may represent
		different numerical values.
	\end{remark}
	
	\subsection{Boundary-layer parts of the approximation}\label{subsec_Bound_layer}
	
	The regular ansatzes
	\begin{equation}\label{regul-2+}
		\mathfrak{W}^{(i)}_\varepsilon =
		w_0^{(i)}(y^{(i)}_1,t)
		+ \varepsilon\left(w_1^{(i)}(y^{(i)}_1,t)
		+ z_1^{(i)}\!\Big(y^{(i)}_1,\tfrac{y^{(i)}_2}{\varepsilon},t\Big)\right)
		+ \varepsilon^2 z_2^{(i)}\!\Big(y^{(i)}_1,\tfrac{y^{(i)}_2}{\varepsilon},t\Big),
		\qquad i\in\{2,3\},
	\end{equation}
	do not satisfy the Dirichlet outflow conditions on the sides
	$\{\Upsilon_{\varepsilon}^{(i)}(\ell_i)\}_{i=2}^3$ of the outlet thin
	rectangles with respect to the advection field.  
	Therefore, we introduce boundary-layer corrections that compensate the
	residuals of the regular ansatzes on these sides.  
	We use the boundary-layer ansatz
	\begin{equation}\label{prim+}
		\mathfrak{B}^{(i)}_\varepsilon(y^{(i)},t)
		:= \sum_{k=0}^{2}\varepsilon^{k}\,
		\Pi_k^{(i)}\!\left(
		\frac{\ell_i - y^{(i)}_1}{\varepsilon},
		\frac{y^{(i)}_2}{\varepsilon},
		t
		\right),
		\qquad i\in\{2,3\}.
	\end{equation}
	
	We additionally assume that the longitudinal component $v_1^{(i)}$ of the
	vector field $\overrightarrow{V_\varepsilon}^{(i)}$ is independent of
	$y_1^{(i)}$ in a small neighborhood of $\Upsilon_{\varepsilon}^{(i)}(\ell_i)$,
	i.e.,
	$
	\overrightarrow{V_\varepsilon}^{(i)} = \big(v_1^{(i)}(\ell_i),\,0\big),$
	$ i\in\{2,3\}.$
	This is a technical assumption; in the general case, $v_1^{(i)}$ should be
	expanded in a Taylor series near $y_1^{(i)}=\ell_i$.
	
	Substituting~\eqref{prim+} into \eqref{eq-thin-rect} and the corresponding boundary conditions, and using the assumption on the transversal component $v_2^{(i)},$ we obtain, after collecting coefficients of equal powers of 
	$\varepsilon,$ the following problems.
	For $k=0$:
	\begin{equation}\label{prim+probl+0}
		\left\{
		\begin{array}{rcll}
			\Delta_\eta \Pi_0^{(i)} + v_1^{(i)}(\ell_i)\,\partial_{\eta_1}\Pi_0^{(i)}
			&=& 0,
			& \eta\in\mathfrak{C}_+^{(i)},
			\\[2mm]
			\partial_{\eta_2}\Pi_0^{(i)} &=& 0,
			& \eta\in\partial\mathfrak{C}_+^{(i)}\setminus\Upsilon^{(i)},
			\\[2mm]
			\Pi_0^{(i)}(0,\eta_2,t) &=& \Lambda^{(i)}_0(t),
			& \eta_2\in(-h_i,h_i),
			\\[2mm]
			\Pi_0^{(i)}(\eta,t) &\to& 0,
			& \eta_1\to+\infty,
		\end{array}
		\right.
	\end{equation}
	where $\eta_1 = \frac{\ell_i - y_1^{(i)}}{\varepsilon},$ $\eta_2 = \frac{y_2^{(i)}}{\varepsilon},$
	$\mathfrak{C}_+^{(i)} := \{\eta:\ \eta_1>0,\ |\eta_2|<h_i\},$ $\Upsilon^{(i)} := \{\eta:\ \eta_1=0,\ |\eta_2|<h_i\},$
	$\Lambda^{(i)}_0(t) := q_i(t) - w_0^{(i)}(\ell_i,t). $
	\
	For $k=1,2$:
	\begin{equation}\label{prim+probl+k}
		\left\{
		\begin{array}{rcll}
			\Delta_\eta \Pi_k^{(i)} + v_1^{(i)}(\ell_i)\,\partial_{\eta_1}\Pi_k^{(i)}
			&=& \partial_t \Pi_{k-1}^{(i)},
			& \eta\in\mathfrak{C}_+^{(i)},
			\\[2mm]
			\partial_{\eta_2}\Pi_k^{(i)} &=& 0,
			& \eta\in\partial\mathfrak{C}_+^{(i)}\setminus\Upsilon^{(i)},
			\\[2mm]
			\Pi_k^{(i)}(0,\eta_2,t) &=& \Lambda^{(i)}_k(t),
			& \eta_2\in(-h_i,h_i),
			\\[2mm]
			\Pi_k^{(i)}(\eta,t) &\to& 0,
			& \eta_1\to+\infty,
		\end{array}
		\right.
	\end{equation}
	with $\Lambda^{(i)}_1(t) = -w_1^{(i)}(\ell_i,t),$ $\Lambda^{(i)}_2(t) = 0.$ The variable $t$ enters as a parameter in these problems.
	
	Using the Fourier method, one easily obtains explicit solutions of
	\eqref{prim+probl+0}–\eqref{prim+probl+k}. For example,
	\[
	\Pi_0^{(i)}(\eta_1,t)
	= \Lambda_0^{(i)}(t)\,e^{-v_1^{(i)}(\ell_i)\eta_1},
	\qquad
	\Pi_1^{(i)}(\eta_1,t)
	= \left(\Lambda_1^{(i)}(t)
	- \frac{\partial_t\Lambda_0^{(i)}(t)}{v_1^{(i)}(\ell_i)}\,\eta_1\right)
	e^{-v_1^{(i)}(\ell_i)\eta_1}.
	\]

	Since $v_1^{(i)}(\ell_i)\ge\theta_0>0$ for all $t\in[0,T]$, and the prefactors
	of $e^{-v_1^{(i)}(\ell_i)\eta_1}$ remain bounded, we obtain the uniform decay
	estimate
	\begin{equation}\label{exp_decay}
		\Pi_k^{(i)}(\eta_1,t)
		= \mathcal{O}\!\left(e^{-\frac{\theta_0}{2}\eta_1}\right)
		\quad\text{as }\eta_1\to+\infty,
		\qquad k\in\{0,1,2\},\ i\in\{2,3\}.
	\end{equation}
	Obviously,
	\begin{equation}\label{in_cond}
		\Pi_0^{(i)}|_{t=0}
		= \Pi_1^{(i)}|_{t=0}
		= \Pi_2^{(i)}|_{t=0}
		= 0,
		\qquad i\in\{2,3\}.
	\end{equation}

	
	\subsection{Inner-corner asymptotics}\label{corner-2-4}
	
	Since for each $i\in\{1,2,3\}$ the leading term of the regular expansion \eqref{exp-1} is uniquely determined 
	as a solution to problem~\eqref{relations+}, it does not 
	satisfy the nonlinear Robin condition~\eqref{coup-1} on the  boundary segment 
	$\Gamma^{(i,0)}_\varepsilon$. Therefore, additional \emph{inner-corner correctors} must be introduced.  
	We describe this construction in detail for problem~\eqref{relations+} with $i=1$.  
	
	\begin{remark}
		The general theory for constructing the asymptotics of solutions to boundary-value problems in non-local perturbed  domains with corner points, presented in \cite[\S 5.5]{Maz-Naz-Pla-2012}, provides a comprehensive set of tools for analysis in special weighted spaces. However, its application is excessive for our problem. We employ suitable conformal mappings directly, which allow the required properties of solutions to be obtained much more simply. Furthermore, the examples in the monograph do not include problems with Robin conditions, which further justifies the use of our approach. In addition to the standard regularity and asymptotic behaviour of solutions, this approach also yields detailed information on the behaviour of the normal derivatives to rays parallel to the sides of the angle, and we study this behaviour along those rays. This information is needed to estimate residuals near the node.
	\end{remark}
	
	We denote by $\theta := \theta_1 \in(\frac{\pi}{2},\pi)$ the opening 
	angle between the sides $\mathcal I_1$ and $\mathcal I_2$.
	After introducing the scaled variables 
	\begin{equation}\label{rescaled}
		\zeta_1 = \frac{x_1 - \varepsilon\big(\frac{h_2}{\cos \theta} - h_1\big) \cot\theta}{\varepsilon}, \qquad
		\zeta_2 = \frac{x_2 - \varepsilon \, h_1}{\varepsilon}	
	\end{equation}
	and then letting $\varepsilon \to 0,$  the domain $\Omega_{\varepsilon}^{(1)}$ transforms into the unbounded domain $\mathfrak{K}$ which, outside a bounded neighborhood of the origin, coincides 
	with the sector
	$\{\zeta=(\zeta_1, \zeta_2)\colon \ \varrho > 0, \ \phi \in (\pi -\theta, \pi)\},$ and which belongs to this sector.
	Here $(\varrho,\phi)$ are the polar coordinates in the $\zeta$--plane.

	The boundary of $\mathfrak{K}$ consists of two rays, $\mathfrak{I}_1 =\{\zeta\colon  \phi = \pi, \ \varrho > \ell_0^*\} $ and
	$\mathfrak{I}_2 =\{\zeta\colon  \phi = \pi -\theta, \ \varrho > \ell_0^{**}\},$
	and a smooth curve $\mathfrak{I}_0$ connecting the initial points $P_1(\ell_0^*, \pi)$ and $P_2(\ell_0^{**}, \pi - \theta)$  of these rays, where $\ell_0^*$ and $\ell_0^{**}$ are some positive numbers. The curve $\mathfrak{I}_0$ is the image of ${\Gamma_\varepsilon^{(1,0)}}$ under the mapping~\eqref{rescaled}.
	
	In the intersection of a small neighborhood of the origin with  $\Omega_{\varepsilon}^{(1)}$ we seek an inner-corner ansatz in the form
	\begin{equation}\label{corn-asym}
		\varepsilon \, \mathcal{K}^{(1)}_1\Big(\frac{x_1 - \varepsilon\big(\frac{h_2}{\cos \theta} - h_1\big) \cot\theta}{\varepsilon}, \frac{x_2 - \varepsilon \, h_1}{\varepsilon}, t\Big) \,  + \ldots
	\end{equation}  
	Substituting \eqref{corn-asym} into the differential equation \eqref{p-equations} for 
	$i=1$ and into the corresponding boundary conditions, and using the compact-support 
	assumptions $\mathbf{A4}$ for $\Psi^{(1,1)}$ and 
	$\Psi^{(1,2)}$, we obtain the following boundary-value problem for the coefficient $\mathcal{K}_1 := \mathcal{K}^{(1)}_1$ :
	\begin{equation}\label{corn-probl+0}
		\left\{
		\begin{array}{rcll}
			\Delta_\zeta \mathcal{K}_1(\zeta,t) &=& 0,
			& \zeta\in \mathfrak{K},
			\\[2mm]
			\nabla_\zeta \mathcal{K}_1 \cdot \boldsymbol{\nu}_\zeta &=& 0,
			& \zeta\in \mathfrak{I}_1  \cup \mathfrak{I}_2,
			\\[2mm]
			\nabla_\zeta \mathcal{K}_1 \cdot \boldsymbol{\nu}_\zeta  &=& \varPsi_0(\zeta,t), 
			& \zeta \in \mathfrak{I}_0,
		\end{array}
		\right.
	\end{equation}
	where $\boldsymbol{\nu}_\zeta$ denotes the outward unit normal to $\partial\mathfrak{K}$, 
	the variable $t\in [0,T]$ plays the role of a parameter, and
	\begin{equation*}
		\varPsi_0(\zeta,t) = - \frac{1}{D_1} \, \Psi_0\Big(u_0^{(1)}(0,t), \, w_0^{(1)}(0,t),\,  \zeta_1 + \big(\tfrac{h_2}{\cos \theta} - h_1\big) \cot\theta, \, \zeta_2 + h_2, \, t\Big).
	\end{equation*}
	It should be noted that $\varPsi_0$ vanishes in a small vicinity of both the endpoints $P_1$ and $P_2$ of the curve  $\mathfrak{I}_0.$ 
	\begin{figure}[htbp]  
		\vspace*{-0.3cm}
		\begin{center}
			\includegraphics[width=10cm]{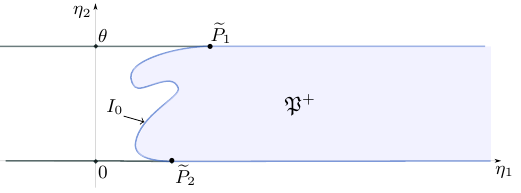}
		\end{center}
		\vspace*{-0.4cm}
		\caption{{\small Semi-strip $\mathfrak{P}^+$}}\label{fig-4}
	\end{figure}
	
	In new variables \ $\eta_1 = \log \varrho, \ \ \eta_2 = \phi - \pi + \theta$ problem~\eqref{corn-probl+0} takes form  
	\begin{equation}\label{strip-probl+0}
		\left\{
		\begin{array}{rcll}
			\Delta_\eta {K}_1(\eta,t) &=& 0,
			& \eta\in \mathfrak{P}^+,
			\\[2mm]
			\partial_{\eta_2}{K}_1|_{\eta_2 = 0}  &=& 0, & \eta_1 \in (\log \ell^{**}, +\infty),
			\\[2mm]
			\partial_{\eta_2}{K}_1|_{\eta_2 = \theta} &=& 0, & \eta_1 \in (\log \ell^{*}, +\infty),
			\\[2mm]
			\nabla_\eta{K}_1 \cdot \boldsymbol{\nu}_\eta &=& \widetilde{\varPsi}_0(\eta,t), & \eta \in I_0,
		\end{array}
		\right.
	\end{equation}
	where $\eta = (\eta_1, \eta_2),$ \ $\widetilde{\varPsi}_0(\eta,t) := e^{\eta_1} \varPsi_0(e^{\eta_1} \cos\eta_2, e^{\eta_1} \sin\eta_2,t),$
	${I}_0$ is the image of the curve $\mathfrak{I}_0.$ The curve ${I}_0$ connects the points $\widetilde{P}_1(\log \ell^{*}, \theta)$ and $\widetilde{P}_2(\log \ell^{**}, 0),$ and divides the strip $\{\eta\colon \ 0  < \eta_2 < \theta\}$ into two semi-strips: the right semi-strip $\mathfrak{P}^+$ (see Fig.~\ref{fig-4}) and the left one.
	
	Problem \eqref{strip-probl+0} will be analyzed through the following model problem:
	\begin{equation}\label{strip-probl+model}
		\left\{
		\begin{array}{rcll}
			\Delta_\eta {K}(\eta) &=& f(\eta),
			& \eta\in \mathfrak{P}^+,
			\\[2mm]
			\partial_{\eta_2}{K}|_{\eta_2 = 0}  &=& 0, & \eta_1 \in (\log \ell^{**}, +\infty),
			\\[2mm]
			\partial_{\eta_2}{K}|_{\eta_2 = \theta} &=& 0, & \eta_1 \in (\log \ell^{*}, +\infty),
			\\[2mm]
			\nabla_\eta{K} \cdot \boldsymbol{\nu}_\eta &=& g(\eta), & \eta \in I_0,
		\end{array}
		\right.
	\end{equation}
	where $f$ and $g$ are given smooth functions.
	We seek a solution in the weighted Sobolev space
	$$
	\mathcal{H}^+ := \big\{\, 
	u \in H^1_{loc,\eta_1}(\mathfrak{P}^+) \; \colon \ 
	\|\nabla_{\eta} u\|_{L^2(\mathfrak{P}^+)} < \infty,\quad  
	\|\rho\, u\|_{L^2(\mathfrak{P}^+)} < \infty
	\big\},
	$$ 
	where the weighted function $\rho(\eta_1)=(1+ |\eta_1|)^{-1},$ and the norm $\| \cdot \|_\mathcal{H}^+$ is generated by the scalar product
	$$
	(u, v)_\mathcal{H}^+
	=  \int_{\mathfrak{P}^+}\nabla_\eta u \cdot \nabla_\eta v \, d\eta + \int_{\mathfrak{P}^+} \rho^2(\eta_1) \, u \, v \, d\eta .
	$$
	Clearly, constant functions belong to $\mathcal{H}$ and the Dirichlet integral of every function from this space is finite.
	A  function $K \in \mathcal{H}^+$ is called a weak solution to  \eqref{strip-probl+model}
	if
	\begin{equation}\label{int-identity-1}
		\int_{\mathfrak{P}^+}\nabla_{\eta}K\cdot\nabla_{\eta}v \, d\eta =    \int_{I_0} g\, v\, dl_\eta - \int_{\mathfrak{P}^+} f \, v\,d\eta \qquad \forall\, v \in \mathcal{H}^+.  
	\end{equation}
	\begin{proposition}\label{prop-2-2}
		Let $\rho^{-1} f\in L^2(\mathfrak{P}^+).$
		Then a weak solution to problem \eqref{strip-probl+model} exists if and only if 
		\begin{equation}\label{solvability-1}
			\int_{I_0} g\, dl_\eta = \int_{\mathfrak{P}^+} f \, d\eta;
		\end{equation}
		this solution is determined up to an additive constant.
		
		If, in addition, $\exp(\beta\eta_1)f \in L^2(\mathfrak{P}^+)$ for some $\beta>0$, then there exists a unique solution satisfying
		the asymptotic behavior
		\begin{equation}\label{asym-1}
			K(\eta) = \mathcal{O}\!\left(\exp(-\beta_0 \eta_1)\right)
			\quad \text{as } \ \ \eta_1 \to +\infty ,
		\end{equation}
		for any $\beta_0 \in (0,\beta)$ such that $\beta_0 \le \sqrt{\lambda_1}$, where $\lambda_1$ is the first positive eigenvalue of the Neumann problem on the interval $(0,\theta)$.
	\end{proposition}
	\begin{proof}
		Necessity follows from \eqref{int-identity-1} with $v \equiv1$.  
		To prove sufficiency, note first that the assumption $\rho^{-1} f \in L^2(\mathfrak{P}^+)$ ensures that the right-hand side of \eqref{int-identity-1} defines a linear bounded functional on $\mathcal{H}^+$.
		
		The left-hand side of \eqref{int-identity-1} can be rewritten as
		\[
		\int_{\mathfrak{P}^+} \nabla_\eta K \cdot \nabla_\eta v \, d\eta
		= \langle K, \phi \rangle - \int_{\mathfrak{P}^+_L} K\, v \, d\eta,
		\]
		where
		\begin{equation}\label{new-scalar-prod}
			\langle u, v  \rangle :=
			\int_{\mathfrak{P}^+} \nabla_\eta u \cdot \nabla_\eta v \, d\eta
			+ \int_{\mathfrak{P}^+_L} u \, v \, d\eta,
		\end{equation}
		$\mathfrak{P}^+_L := \mathfrak{P}^+ \cap \{|\eta_1|<L\}$, and $L > \max\{\log\ell^*, \log\ell^{**}, 1\}$ is chosen so that $I_0 \subset \{|\eta_1|<L\}$.
		
		The scalar product \eqref{new-scalar-prod} generates a norm equivalent to $\|\cdot\|_{\mathcal{H}^+}$.  
		The inequality $\langle u, u \rangle \le c_1 \|u\|_{\mathcal{H}^+}^2$ is immediate.  
		The reverse inequality follows from Hardy’s inequality
		\begin{equation}\label{Hardy}
			\int_0^{+\infty} (1+\eta_1)^{-2} \varphi^2(\eta_1)\, d\eta_1
			\le 4 \int_0^{+\infty} |\partial_{\eta_1}\varphi|^2\, d\eta_1,
			\qquad \varphi \in H^1((0,+\infty)),\ \varphi(0)=0,
		\end{equation}
		see \cite[Lemma 4.1]{Mel-ZAA-1999}.
		Since the embedding $\mathcal{H}^+ \hookrightarrow L^2(\mathfrak{P}^+_L)$ is compact, there exists a positive self-adjoint compact operator $\mathbf{A} : \mathcal{H}^+ \to \mathcal{H}^+$ such that
		\[
		\langle \mathbf{A} u, v \rangle
		= \int_{\mathfrak{P}^+_L} u\, v\, d\eta,
		\qquad u,\, v  \in \mathcal{H}^+.
		\]
		Using the Riesz representation theorem, identity \eqref{int-identity-1} becomes the operator equation
		$K - \mathbf{A}K = \mathcal{F} $ in $\mathcal{H}^+,$
		to which the Fredholm theory applies. The homogeneous problem admits only constant solutions in $\mathcal{H}$, hence the solvability condition is exactly \eqref{solvability-1}.
		
		Finally, if $\exp(\beta \eta_1) f \in L^2(\mathfrak{P}^+)$, standard properties of elliptic problems in semi-strips allow one to select a solution satisfying the decay \eqref{asym-1} for any $\beta_0 \in (0,\beta)$ with $\beta_0 \le \sqrt{\lambda_1}$.
	\end{proof}
	
	Based on Proposition~\ref{prop-2-2}, we obtain the following result.
	\begin{proposition}\label{prop-2-3}
		There exists a solution to problem \eqref{strip-probl+0}  satisfying
		\begin{equation}\label{asym-2}
			K_1(\eta,t) = C_1(t) \eta_1 + \Lambda_1(t) + \mathcal{O}\!\left(\exp(- \tfrac{\pi}{\theta}\,  \eta_1)\right)
			\quad \text{as } \ \ \eta_1 \to +\infty,
		\end{equation}
		where $\Lambda_1(t)$ is any  time-dependent smooth  function and 
		\begin{equation}\label{const-1}
			C_1(t) = - \frac{1}{\theta} \, \int_{I_0} \widetilde{\varPsi}_0(\eta,t)\, dl_\eta .
		\end{equation}
	\end{proposition}
	\begin{proof}
		We look for $K_1$ in the form
		$
		K_1(\eta) = \chi_+(\eta_1)\, C_1(t)\, \eta_1 + \widetilde{K}_1(\eta),
		$
		where $\widetilde{K}_1 \in \mathcal{H}^+$ solves problem \eqref{strip-probl+model} with
		\[
		g = \widetilde{\varPsi}_0(\eta,t), \qquad
		f(\eta) = -C_1(t)\, \big(2\chi_+'(\eta_1) + \chi_+''(\eta_1)\,\eta_1\big).
		\]
		Here $\chi_+$ is a smooth cut-off function such that
		\[
		\chi_+(\eta_1)=
		\begin{cases}
			1, & \eta_1 \ge L+1,\\
			0, & \eta_1 \le L,
		\end{cases}
		\]
		with $L$ chosen as in Proposition~\ref{prop-2-2}.
		The solvability condition \eqref{solvability-1} for $\widetilde{K}_1$ gives
		\[
		\int_{I_0} \widetilde{\varPsi}_0(\eta,t)\, dl_\eta
		= - C_1(t)\, \theta,
		\]
		which yields \eqref{const-1}.  The asymptotics \eqref{asym-2} then follow directly from Proposition~\ref{prop-2-2}.
	\end{proof}
	\begin{corollary}
		In view of problem~\eqref{strip-probl+0} and the asymptotic behavior~\eqref{asym-2}, we have 
		\begin{equation*}
			\partial_{\eta_1} K_1(\eta,t)
			= C_1(t) + \mathcal{O}\!\left(\exp(- \tfrac{\pi}{\theta}\,  \eta_1)\right) \quad \text{and} \quad 
			\partial_{\eta_2} K_1(\eta,t)
			= \mathcal{O}\!\left(\exp(- \tfrac{\pi}{\theta}\,  \eta_1)\right)
			\quad \text{as } \ \ \eta_1 \to +\infty.
		\end{equation*}
	\end{corollary}

	Returning to problem \eqref{corn-probl+0}, we obtain the following.
	\begin{proposition}\label{prop-2-4}
		Problem \eqref{corn-probl+0} admits a solution satisfying
		\begin{equation}\label{asym-3}
			\mathcal{K}^{(1)}_1(\zeta,t)
			\, = \,  C_1^{(1)}(t)\, \log \varrho + \Lambda_1^{(1)}(t) + \mathcal{O}\!\left(\varrho^{-\pi/\theta}\right)
			\quad \text{as } \ \ \varrho = \sqrt{\zeta_1^2 + \zeta_2^2} \to +\infty,
		\end{equation}
		where $\Lambda_1^{(1)}(t)$ is any  time-dependent smooth  function (it will be chosen below) and
		\begin{equation}\label{const-2}
			C_1^{(1)}(t)
			= \frac{1}{\theta\, D_1}
			\int_{\mathfrak{I}_0}
			\Psi_0\!\left(u_0^{(1)}(0,t),\, w_0^{(1)}(0,t),\,
			\zeta_1 + \big(\tfrac{h_2}{\cos\theta}-h_1\big)\cot\theta,\,
			\zeta_2 + h_2,\,
			t \right) dl_\zeta .
		\end{equation}
	\end{proposition}
	
	\subsubsection{Bulk correction}\label{bulk-corr}
	To neutralize the logarithmic growth of $\mathcal{K}_1^{(1)}$ at infinity, we
	need to find a bulk correction $\varepsilon\,u_1^{(1)}(x,t)$ in the domain
	$\Omega^{(1)}$ whose leading behavior near the origin is also logarithmic.

	Based on assumptions $\mathbf{A1}$ and $\mathbf{A4}$ for the functions $F_i$ and $\Psi^{(i,j)}$, 
	we look for  the second term  $u_1^{(i)}$ of the expansion \eqref{exp-1} as a solution to the linear problem
	\begin{equation}\label{bulk-term-1}
		\left\{
		\begin{array}{rcll}
			- D_i \Delta_x u^{(i)}_1
			+ \widehat{F}_i(x,t) \, u^{(i)}_1 &=& 0
			& \text{in }\ \Omega^{(i)},
			\\[2mm]
			u^{(i)}_1 &=& 0
			& \text{on } \ \ \mathcal{B}_i,
			\\[2mm]
			D_i\nabla_x u^{(i)}_1\cdot\boldsymbol{\nu}_0
			&=& \widehat{\Psi}^{(i,i)}(y^{(i)}_1,t)\, u^{(i)}_1 + \widehat{G}^{(i,i)}(y^{(i)}_1,t)
			& \text{on } \ \mathcal{I}_i,
			\\[2mm]
			D_i\nabla_x u^{(i)}_1\cdot\boldsymbol{\nu}_0
			&=& \widehat{\Psi}^{(i,j)}(y^{(j)}_1,t)\, u^{(i)}_1 + \widehat{G}^{(i,j)}(y^{(j)}_1,t)
			& \text{on } \ \mathcal{I}_j,			
		\end{array}
		\right.
	\end{equation}
	supplemented by the logarithmic asymptotics
	\begin{equation}\label{log-sol}
		u^{(i)}_1(x,t) \, \sim \, 
		C^{(i)}_1(t)\, \log\sqrt{x_1^2+x_2^2} \quad \text{as } \ |x|\to 0.
	\end{equation}
	Here the variable $t\in [0,T]$ plays the role of a parameter, $\boldsymbol{\nu}_0$ denotes
	the inward unit normal to $\partial\Omega^{(i)}$, and the coefficients
	\begin{gather*}
		\widehat{F}_i(x,t) := \partial_s F_i\big(u^{(i)}_0,x,t\big)  , \qquad \widehat{\Psi}^{(i,j)}(y^{(j)}_1,t) := \partial_s \Psi^{(i,j)}\big(u^{(i)}_0,w^{(j)}_0,y^{(j)}_1,t\big)
		\\
		\widehat{G}^{(i,j)}(y^{(j)}_1,t) := \partial_w \Psi^{(i,j)}\big(u^{(i)}_0,w^{(j)}_0,y^{(j)}_1,t\big)\, \big(w_1^{(j)}  + z_1^{(j)} \big),	
	\end{gather*}
	are determined by the linearizations of $F_i$ and $\Psi^{(i,j)}$ along the leading-order solutions $u_0^{(i)}$ and $w_0^{(j)}$.

	\begin{remark}\label{rem-2-6}
		Since $\Psi_0|_{t=0}=0$, we obtain $C^{(i)}_1(0)=0$, and by~\eqref{comp-cond-Psi} also $u^{(i)}_1|_{t=0}=0$.
		
		If $\Psi_0 \equiv 0$, then no logarithmic angular asymptotics appear. In this situation the second term in the bulk expansion~\eqref{exp-1} simply solves the corresponding parabolic problem in $\Omega^{(i)} \times (0,T)$, and the node exerts no influence on the transport in the bulk subdomains. Since the main purpose of this work is precisely to analyse the effect of the node, this case is not of interest and will not be considered.
		
		In the opposite case, when $\Psi_0 \ge 0$ (see Assumption~{\bf A3}), the quantity $C^{(i)}_1(t)$ becomes positive after some initial moment of time. Without loss of generality, we may therefore regard $C^{(i)}_1(t) > 0$ for all $t \in (0,T].$
	\end{remark}

	Restricting attention to the case $i=1$, we rewrite problem \eqref{bulk-term-1} in the variables\ $\eta_1 = \log r,$ $\eta_2 = \phi - \pi + \theta$, where  $(r,\phi)$ are the polar coordinates in the $x$--plane.  In these variables, the domain $\Omega^{(1)}$ is mapped onto the semi-infinite strip $\mathfrak{P}^-$, and problem \eqref{bulk-term-1} takes the form  
	\begin{equation}\label{strip-probl+1}
		\left\{
		\begin{array}{rcll}
			D_1\Delta_\eta U_1(\eta,t) &=& \widetilde{F}_1(\eta,t) \, U_1(\eta,t),
			& \eta\in \mathfrak{P}^-,
			\\[2mm]
			D_1\partial_{\eta_2}{U}_1|_{\eta_2 = 0}  &=& \widetilde{\Psi}^{(1,2)}(\eta_1,0,t)\, U_1(\eta_1,0,t) + \widetilde{G}^{(1,2)}(\eta_1,0,t), & \eta_1 \in (-\infty, \log\ell_2),
			\\[2mm]
			-D_1\partial_{\eta_2}{U}_1|_{\eta_2 = \theta} &=& \widetilde{\Psi}^{(1,1)}(\eta_1,\theta,t)\, U_1(\eta_1,\theta,t) + \widetilde{G}^{(1,1)}(\eta_1,\theta,t), & \eta_1 \in (-\infty, \log\ell_1),
			\\[2mm]
			U_1(\eta,t) &=& 0, & \eta \in \mathfrak{B}_1.
		\end{array}
		\right.
	\end{equation}
	The transformed coefficients $\widetilde{F}_1(\eta,t) = e^{2\eta_1} \, \widehat{F}_1(e^{\eta_1},\eta_2 +\pi - \theta),$
	$\widetilde{\Psi}^{(1,j)}(\eta,t) = e^{\eta_1} \widehat{\Psi}^{(i,j)}(e^{\eta_1},\eta_2 +\pi - \theta,t)$ and 
	$\widetilde{G}^{(1,j)}$ inherit compact support in $\eta_1$ from the original
	coefficients, and the curve $\mathfrak{B}_1$ divides the strip
	$\{0<\eta_2<\theta\}$ into two semi-strips, with $\mathfrak{P}^-$ being the left
	one.
	
	The condition \eqref{log-sol} in variables $(\eta_1,\eta_2)$ becomes
	\begin{equation}\label{asymptotics-1}
		U_1(\eta,t) \, \sim \, C_1^{(1)}(t)\, \eta_1 \quad \text{as } \ \ \eta_1 \to -\infty.
	\end{equation}
	
	We look for a such solution in the form 
	$
	U_1(\eta,t) = \chi_-(\eta_1)\, C_1^{(1)}(t) \, \eta_1 + \widetilde{U}_1(\eta,t),
	$
	where $\chi_-$ is a smooth cut-off function such that
	\[
	\chi_-(\eta_1)=
	\begin{cases}
		1, & \eta_1 \le L - 1,\\
		0, & \eta_1 \ge L,
	\end{cases}
	\]
	with $L$ chosen so that the supports of $\widetilde{F}_1$, $\widetilde{\Psi}^{(1,j)}$,
	and $\widetilde{G}^{(1,j)}$ lie inside the rectangle
	$\{\,L \le \eta_1 \le l_0,\; 0 \le \eta_2 \le \theta\,\} \subset \overline{\mathfrak{P}^-}$.
	By assumptions ${\bf A1}$ and ${\bf A4}$, such numbers $L$ and $l_0$ can indeed be chosen.
	
	The function $\widetilde{U}_1$ must solve the equation
	$$
	D_1\Delta_\eta \widetilde{U}_1 = \widetilde{F}_1 \, \widetilde{U}_1 + \tilde{f}_1 \quad \text{in} \ \ \mathfrak{P}^-,
	$$ 
	with the same boundary conditions as in \eqref{strip-probl+1}, and with
	\begin{equation}\label{func-f_1}
		\tilde{f}_1(\eta_1,t) = -C_1^{(1)}(t)\, \big(2\chi_-'(\eta_1) + \chi_-''(\eta_1)\,\eta_1\big).
	\end{equation}
	The source term $\tilde{f}_1$ is compactly supported in the transition region
	$L-1 < \eta_1 < L$, reflecting the fact that the cut-off modification of the
	leading $\eta_1$-term is localized in this strip.

	We seek a solution $\widetilde{U}_1$ in the energy space
	$$
	\mathcal{H}^- := \big\{\, 
	u \in H^1_{loc,\eta_1}(\mathfrak{P}^-) \; \colon \ 
	\|\nabla_{\eta} u\|_{L^2(\mathfrak{P}^-)} < \infty,\quad  
	\|\rho\, u\|_{L^2(\mathfrak{P}^-)} < \infty, \quad u|_{\mathfrak{B}_1}=0
	\big\},
	$$ 
	where the weight is $\rho(\eta_1)=(1+ |\eta_1|)^{-1}.$ Due to the zero Dirichlet condition on $\mathfrak{B}_1$ and corresponding Hardy’s inequality (see \eqref{Hardy}), we can introduce the norm in $\mathcal{H}^-$ generated by the scalar product
	$$
	(u, v)_{\mathcal{H}^-}
	=  \int_{\mathfrak{P}^-}\nabla_\eta u \cdot \nabla_\eta v \, d\eta, \quad  u, \, v \in \mathcal{H}^- .
	$$ 
	The weak formulation is as follows:  \  find  $\widetilde{U}_1 \in \mathcal{H}^-$ such that
	\begin{multline}\label{int-iden-1}
		D_1 \int_{\mathfrak{P}^-} \nabla_\eta\widetilde{U}_1\cdot\nabla_{\eta} v \, d\eta + \int_{\mathfrak{P}^-} \widetilde{F}_1 \, \widetilde{U}_1 \, v \, d\eta + \int_{-\infty}^{\log\ell_1} \widetilde{\Psi}^{(1,1)}(\eta_1,\theta,t) \, \widetilde{U}_1 \, v \, d\eta_1 + 
		\int_{-\infty}^{\log\ell_2} \widetilde{\Psi}^{(1,2)}(\eta_1,0,t) \, \widetilde{U}_1 \, v \, d\eta_1
		\\
		= - \int_{\mathfrak{P}^-}  \tilde{f}_1 \, v \, d\eta - 	
		\int_{-\infty}^{\log\ell_1} \widetilde{G}^{(1,1)}(\eta_1,\theta,t) \,  v \, d\eta_1 - 
		\int_{-\infty}^{\log\ell_2} \widetilde{G}^{(1,2)}(\eta_1,0,t) \,  v \, d\eta_1
	\end{multline}  
	for every test function $v\in \mathcal{H}^-.$

	We consider the bilinear form
	\[
	a(u, v) :=
	D_1 \int\limits_{\mathfrak{P}^-} \nabla_\eta u \cdot\nabla_{\eta} v \, d\eta + \int\limits_{\mathfrak{P}^-} \widetilde{F}_1 \, u \, v \, d\eta + \int\limits_{-\infty}^{\log\ell_1} \widetilde{\Psi}^{(1,1)}(\eta_1,\theta,t) \, u \, v \, d\eta_1 + 
	\int\limits_{-\infty}^{\log\ell_2} \widetilde{\Psi}^{(1,2)}(\eta_1,0,t) \, u \, v \, d\eta_1
	\]
	on $\mathcal{H}^- \times \mathcal{H}^-$, and the linear functional
	$$
	\mathcal{G}(v) := - \int_{\mathfrak{P}^-}  \tilde{f}_1 \, v \, d\eta - 	
	\int_{-\infty}^{\log\ell_1} \widetilde{G}^{(1,1)}(\eta_1,\theta,t) \,  v \, d\eta_1 - 
	\int_{-\infty}^{\log\ell_2} \widetilde{G}^{(1,2)}(\eta_1,0,t) \,  v \, d\eta_1, \quad v \in \mathcal{H}^-. 
	$$
	
	Since $\widetilde{F}_1$, $\widetilde{\Psi}^{(1,j)}$, and $\widetilde{G}^{(1,j)}$ have uniformly compact supports in $\eta_1$ and are uniformly bounded, the Hardy inequality and the trace theorem imply that 
	the bilinear form $a$ and the linear functional $\mathcal{G}$ are bounded on $\mathcal{H}^- \times \mathcal{H}^-$ and on $\mathcal{H}^-$, respectively.
	The coercivity of $a$ follows directly from \eqref{pos-1} and \eqref{pos-2}. Hence, by the Lax--Milgram theorem, there exists a unique
	$\widetilde{U}_1 \in \mathcal{H}^-$ satisfying identity \eqref{int-iden-1}. 
	Moreover, by standard properties of elliptic second-order problems in semi-strips, this solution has the asymptotic behaviour
	\begin{equation}\label{asym-4}
		\widetilde{U}_1(\eta,t) = \widetilde{C}^{(1)}_1(t) + \mathcal{O}\!\left(e^{\beta_0 \eta_1}\right)
		\qquad \text{as } \ \ \eta_1 \to -\infty,
	\end{equation}
	for any $\beta_0 \in (0,\sqrt{\lambda_1}\,]$, where $\lambda_1$ is the first positive eigenvalue of the Neumann problem on $(0,\theta)$.
	
	To determine $\widetilde{C}^{(1)}_1(t)$ in \eqref{asym-4}, we apply the Gauss--Ostrogradsky formula in the truncated domain
	$$
	\mathfrak{P}^-_R := \{\eta \in \mathfrak{P}^-\colon \ \eta_1 > - R\}
	$$ 
	to $\widetilde{U}_1$ and to the solution $\widehat{U}_1$ of the homogeneous problem \eqref{strip-probl+1}, which satisfies asymptotics \eqref{asymptotics-1}. This yields
	$$
	\int_{\mathfrak{P}^-_R}\Big(\widehat{U}_1 \, \Delta_\eta\widetilde{U}_1 - \widetilde{U}_1 \, \Delta_\eta\widehat{U}_1\Big) =
	\int_{\partial\mathfrak{P}^-_R}\Big(\widehat{U}_1 \, \partial_{\boldsymbol{\nu}_\eta}\widetilde{U}_1 - \widetilde{U}_1 \, \partial_{\boldsymbol{\nu}_\eta}\widehat{U}_1\Big),
	$$ 
	where $\partial_{\boldsymbol{\nu}_\eta}$ denotes the outward normal derivative on $\partial\mathfrak{P}^-_R$. 
	Using the differential equations and boundary conditions for these functions, the above identity becomes
	$$
	\int\limits_{\mathfrak{P}^-_R} \widehat{U}_1 \, \tilde{f}_1 \, d\eta + \int\limits_{-R}^{\log\ell_1} \widetilde{G}^{(1,1)}(\eta_1,\theta,t) \,  \widehat{U}_1 \, d\eta_1 + \int\limits_{-R}^{\log\ell_2} \widetilde{G}^{(1,2)}(\eta_1,0,t) \,  \widehat{U}_1\, d\eta_1 =
	\int\limits_{0}^{\theta}\Big(-\widehat{U}_1 \, \partial_{\eta_1}\widetilde{U}_1 + \widetilde{U}_1 \, \partial_{\eta_1}\widehat{U}_1\Big)\Big|_{\eta_1=-R} \, d\eta_2.
	$$
	Passing to the limit as $R\to+\infty$ and using the asymptotics \eqref{asym-4} and \eqref{asymptotics-1}, we obtain
	\begin{multline}\label{constant}
		\widetilde{C}^{(1)}_1(t) = -\frac{1}{\theta} \int_{\mathfrak{P}^-} \widehat{U}_1(\eta,t) \, \big(2\chi_-'(\eta_1) + \chi_-''(\eta_1)\,\eta_1\big)\, d\eta 
		\\
		+ \frac{1}{\theta\, C_1^{(1)}(t)}\int_{-\infty}^{\log\ell_1} \widetilde{G}^{(1,1)}(\eta_1,\theta,t) \,  \widehat{U}_1(\eta,t) \, d\eta_1 + \frac{1}{\theta\, C_1^{(1)}(t)}\int_{-\infty}^{\log\ell_2} \widetilde{G}^{(1,2)}(\eta_1,0,t) \,  \widehat{U}_1(\eta,t)\, d\eta_1.
	\end{multline}
	Taking Remark~\ref{rem-2-6} and \eqref{comp-cond-Psi} 
	into account, we have $\widetilde{C}^{(1)}_1(0) =0.$
	
	Thus, there exists a unique solution to problem \eqref{strip-probl+1} which has the asymptotics
	\begin{equation*}
		U_1(\eta,t) = C_1^{(1)}(t)\, \eta_1 \, + \, \widetilde{C}^{(1)}_1(t)\, + \, \mathcal{O}\!\left(e^{\beta_0 \eta_1}\right)
		\quad \text{as } \ \ \eta_1 \to -\infty,
	\end{equation*}  
	for any $\beta_0 \in (0,\sqrt{\lambda_1}]$.
	Since the right-hand sides in problem \eqref{strip-probl+1} have compact supports, we have 
	\begin{equation}\label{asymptotic-derin-origin}
		\partial_{\eta_1} U_1(\eta,t) = C_1^{(1)}(t) \, + \, \mathcal{O}\!\left(e^{\beta_0 \eta_1}\right) \quad \text{and} \quad 
		\partial_{\eta_2} U_1(\eta,t) =  \mathcal{O}\!\left(e^{\beta_0 \eta_1}\right)
		\quad \text{as } \ \ \eta_1 \to -\infty.	
	\end{equation}
	
	We now return to problem \eqref{bulk-term-1}. The analysis carried out above
	implies the following statement.
	\begin{proposition}\label{propo-2-5}
		For each $i\in\{1,2,3\}$, problem \eqref{bulk-term-1} admits a unique solution
		$u^{(i)}_1$ satisfying the corner asymptotics
		\begin{equation}\label{asym-5}
			u^{(i)}_1(x,t)
			= C^{(i)}_1(t)\,\log r + \widetilde{C}^{(i)}_1(t)
			+ \mathcal{O}\!\left(r^{\delta_i}\right)
			\quad \text{as } \ \  r=\sqrt{x_1^2+x_2^2}\to 0,
		\end{equation}
		where $\delta_i=\pi/\theta_i>1$, $\theta_i$ is the opening angle of
		$\Omega^{(i)}$, and $\widetilde{C}^{(i)}_1(t)$ is given by formula
		\eqref{constant} with the corresponding indices $i$ and $j$.
		
		Moreover, the radial derivative satisfies
		\begin{equation}\label{asym-deriv-r}
			\partial_r u^{(i)}_1(x,t)
			= \frac{C^{(i)}_1(t)}{r}
			+ \mathcal{O}\!\left(r^{\delta_i-1}\right)
			\quad \text{as } \ \ r\to 0,
		\end{equation}
		and for each ray $\phi=\mathrm{const}$ in $\Omega^{(i)}$ the normal derivative
		obeys
		\begin{equation}\label{asym-deriv-normal}
			\partial_{\nu} u^{(i)}_1(x,t)
			= \frac{1}{r}\,\partial_\phi u^{(i)}_1(x,t)
			= \mathcal{O}\!\left(r^{\delta_i-1}\right)
			\quad \text{as } \ \ r\to 0.
		\end{equation}
	\end{proposition}
	\begin{corollary}\label{corollarity-2-3}
		The normal derivative of $u^{(i)}_1$ to the sides $\Gamma_\varepsilon^{(i,i)}$ and $\Gamma_\varepsilon^{(i,j)}$ satisfies
		\begin{equation}\label{asym-normal-line-final}
			\partial_{\boldsymbol{\nu}_\varepsilon} u^{(i)}_1(x,t)
			= C^{(i)}_1(t)
			+ \mathcal{O}\!\left(\varepsilon^{\delta_i/2}\right)
			\quad \text{as } \ \ r_\varepsilon = |x| = \mathcal{O}\!\left(\varepsilon^{1/2}\right) \ \ \text{and} \ \ \varepsilon\to 0.
		\end{equation}
	\end{corollary}
	\begin{proof}
		We restrict ourselves to the case $i = 1$ and the side $\Gamma_\varepsilon^{(1,1)}.$ 
		Since the leading terms in \eqref{asym-5} are independent of $\phi$, all
		$\phi$–dependence is contained in the remainder. Differentiating the remainder
		preserves its order:
		\begin{equation}\label{phi-derivative}
			\partial_\phi^k u^{(1)}_1(r,\phi,t)
			= \mathcal{O}\!\left(r^{\delta_1}\right),
			\qquad k=1,2,
		\end{equation}
		because in strip variables $(\eta_1,\eta_2)$ the decaying part satisfies
		$\partial_{\eta_2}^k \widetilde{R}(\eta,t)=\mathcal{O}(e^{\delta_1\eta_1})$, and
		the change of variables $\eta_1=\log r$, $\eta_2=\phi-\pi+\theta_1$ preserves
		this rate.
		
		Since for small $r$ the boundary condition $\partial_\nu u^{(1)}_1=0$ holds on
		$\mathcal{I}_1$, and the outward normal to $\mathcal{I}_1$ in polar coordinates
		coincides with $\mathbf e_\phi$, we therefore have
		$
		\partial_\phi u^{(1)}_1(r,\theta_1,t)=0.
		$
		For $x\in \Gamma_\varepsilon^{(1,1)}$ and such $r$ we have $\phi(x)=\theta_1+\beta_\varepsilon$, where
		$\beta_\varepsilon=\mathcal{O}(\varepsilon^{1/2})$. Hence, by Taylor expansion and \eqref{phi-derivative},
		\begin{equation}\label{phi-derivative+2}
			\partial_\phi u^{(1)}_1(r_\varepsilon,\phi(x),t)
			= \beta_\varepsilon\,\partial^2_{\phi} u^{(1)}_1(r_\varepsilon,\phi^*,t)
			= \mathcal{O}\!\left(\varepsilon^{1/2} r_\varepsilon^{\delta_1}\right),
		\end{equation}
		for some $\phi^*$ between $\theta_1$ and $\phi(x)$. 
		The unit normal to $\Gamma_\varepsilon^{(1,1)}$ is
		\[
		\boldsymbol{\nu}_\varepsilon
		= \sin\beta_\varepsilon\,\mathbf e_r(\phi)
		+ \cos\beta_\varepsilon\,\mathbf e_\phi(\phi),
		\]
		with $\sin\beta_\varepsilon=\mathcal{O}(\varepsilon^{1/2})$ and
		$\cos\beta_\varepsilon=1+\mathcal{O}(\varepsilon)$.
		Thus
		\[
		\partial_{\boldsymbol{\nu}_\varepsilon} u^{(1)}_1
		= \partial_r u^{(1)}_1\,\sin\beta_\varepsilon
		+ \frac{1}{r_\varepsilon}\,\partial_\phi u^{(1)}_1\,\cos\beta_\varepsilon.
		\]
		
		Since $r_\varepsilon=\mathcal{O}(\varepsilon^{1/2})$, the estimates
		\eqref{asym-deriv-r} and \eqref{phi-derivative+2} yield
		\[
		\partial_r u^{(1)}_1\,\sin\beta_\varepsilon
		= C^{(1)}_1(t) + \mathcal{O}\!\left(\varepsilon^{\delta_1/2}\right),
		\qquad
		\frac{1}{r_\varepsilon}\,\partial_\phi u^{(1)}_1
		= \mathcal{O}\!\left(\varepsilon^{\delta_1/2}\right),
		\]
		which proves \eqref{asym-normal-line-final}.
	\end{proof}

	\begin{remark}
		The behaviour in \eqref{asym-normal-line-final} contrasts with the estimate
		\eqref{asym-deriv-normal} on the ray $\phi=\phi_0$, where the normal derivative
		is purely angular and therefore satisfies
		$\partial_\nu u^{(i)}_1=\mathcal{O}(r^{\delta_i-1})\to 0$ as $r\to 0$.
		Along the parallel sides $\Gamma_\varepsilon^{(i,i)}$ and
		$\Gamma_\varepsilon^{(i,j)}$, however, the outward normal is tilted by
		$\mathcal{O}(\varepsilon^{1/2})$ relative to the angular direction. This tilt
		couples the normal derivative to the radial derivative, whose leading term is
		$C^{(i)}_1(t)/r$. Since $r_\varepsilon=\mathcal{O}(\varepsilon^{1/2})$, the
		normal derivative approaches the finite limit $C^{(i)}_1(t)$ rather than zero.
		Thus, even an arbitrarily small deviation from the ray changes the leading
		asymptotics of the normal derivative near the corner.
	\end{remark}

	\begin{remark}\label{rem-2-8} The solution $u_1^{(i)}$ in $\overline{\Omega^{(i)}}\setminus B_\delta(0)$ has the same Hölder regularity
		as $u_0^{(i)}$ (see \eqref{hoelder}).	
	\end{remark}

	Passing to the scaled variables $(\zeta_1,\zeta_2)$ defined in \eqref{rescaled}, 
	the leading terms in \eqref{asym-5} becomes
	\begin{equation}\label{asym-behav-log}
		C_1^{(i)}(t)\,\big(\log\varrho + \log\varepsilon\big) + \widetilde{C}^{(i)}_1(t) + o(1)
		\qquad \text{as } \ \ \varrho\to\infty.
	\end{equation}
	This expression coincides with the inner asymptotic behaviour \eqref{asym-3}
	provided we choose
	\begin{equation}\label{Lambda}
		\Lambda^{(i)}_1(t)
		= C_1^{(i)}(t)\,\log\varepsilon + \widetilde{C}^{(i)}_1(t).
	\end{equation}
	
	Thus the bulk and inner-corner ansatzes match smoothly up to order
	$\mathcal{O}\!\left(\varepsilon^{\pi/(2\theta_i)}\right)$ in the intermediate
	region
	\begin{equation}\label{region-1}
		\mathcal{D}^{(i)}_{\varepsilon}
		:= \big\{x\in\Omega^{(i)}_\varepsilon :
		2\varepsilon^{1/2} < |x| < 4\varepsilon^{1/2}\big\}.
	\end{equation}

	\section{Construction and Justification of the Approximation}\label{Sect-4}
	
	The goal of this section is to assemble a global approximation in the entire 
	domain $\Omega,$ consisting of the bulk regions \(\Omega_\varepsilon^{(1)},\Omega_\varepsilon^{(2)},\Omega_\varepsilon^{(3)}\) and
	the thin graph-like junction $\mathcal{R}_\varepsilon$ inside (see \S~\ref{domain}), and to rigorously justify  its accuracy as an approximation to solution \eqref{orig-solution} of the original problem $\mathbb{P}_\varepsilon$ (see \S~\ref{ibvp}).
	
	\subsection{Approximation} In each bulk domain $\Omega_\varepsilon^{(i)}$ and for $t\in [0,T],$ we propose the following approximation 
	\begin{equation}\label{approxim-3-1}
		\mathfrak{U}_\varepsilon^{(i)}(x,t) = u_0^{(i)}(x,t) + \varepsilon \, \Big(\chi_b(\tfrac{|x|}{\varepsilon^{1/2}}) \, u_1^{(i)}(x,t) \, + \,  \big(1 - \chi_b(\tfrac{|x|}{\varepsilon^{1/2}}) \big)\, \mathcal{K}^{(i)}_1 \Big), 
	\end{equation}
	where $u_0^{(i)}$ is the solution to problem \eqref{relations+}, $u_1^{(i)}$ is the solution to problem \eqref{bulk-term-1}--\eqref{log-sol}, $\mathcal{K}^{(i)}_1$ is the corresponding corner solution (for $i=1$ see problem \eqref{corn-probl+0}), and $\chi_b$ is a smooth cut-off function such that
	\begin{equation}\label{cut-off_function-d}
		\chi_b(s)=
		\begin{cases}
			1, & s \ge 4,
			\\[2pt]
			0, & s \le 2.
		\end{cases}
	\end{equation}
	
	Each part of the junction $\mathcal{R}_\varepsilon$ -- the three thin 
	rectangular branches and the inner node region -- is governed by a different 
	asymptotic regime, and the approximation must smoothly interpolate between them.
	In the branch $\mathcal{R}^{(i)}_\varepsilon$ we have already constructed the 
	regular expansion $\mathfrak{W}^{(i)}_\varepsilon$ 
	(see~\eqref{regul-1+} and \eqref{regul-2+}), which describes the behavior of the solution 
	away from the node and away from the outflow boundaries.  
	Near the node, the geometry forces a rapid change of scale, and the solution is 
	captured instead by the inner node-layer ansatz $\mathfrak{N}_\varepsilon$ 
	(see~\eqref{junc+1}).  
	Finally, in the two branches with outflow boundaries ($i=2,3$), the regular 
	expansion does not satisfy the Dirichlet condition at $y^{(i)}_1=\ell_i$, and the 
	corresponding boundary-layer corrections $\{\mathfrak{B}^{(i)}_\varepsilon\}_{i=2}^3$ 
	(see~\eqref{prim+}) must be added.
	
	To merge these asymptotic components into a single approximation, we employ smooth 
	cut-off functions.  
	The function $\chi_{\ell_0}$, introduced after~\eqref{new-solution_k}, governs the 
	transition between the node-layer and the regular expansions inside each branch.  
	In addition, for $i\in\{2,3\}$ we use the smooth cut-off functions
	\begin{equation}\label{cut-off_functions}
		\chi_\delta^{(i)}(s)=
		\begin{cases}
			1, & s\ge \ell_i-\delta,\\[2pt]
			0, & s\le \ell_i-2\delta,
		\end{cases}
		\qquad i\in\{2,3\},
	\end{equation}
	where $\delta>0$ is chosen sufficiently small so that $\chi_\delta^{(i)}$ vanishes on 
	the support of $\Phi^{(i,j)}$.  
	These functions ensure that the boundary-layer corrections are activated only in a 
	thin neighborhood of the outflow boundaries and do not interfere with the coupling 
	conditions on the lateral sides.
	
	With these ingredients, the global approximation $\mathfrak{A}_\varepsilon$ in $\mathcal{R}_\varepsilon$ is 
	defined by patching together the regular, inner node-layer, and boundary-layer profiles in the 
	appropriate subregions of the junction:
	\begin{equation}\label{first_app}
		\mathfrak{A}_\varepsilon =
		\left\{
		\begin{array}{ll}
			\mathfrak{W}^{(1)}_\varepsilon(y^{(1)}, t) & \text{in} \  \  \mathcal{R}^{(1)}_{\varepsilon,3\ell_0,\gamma}, 
			\\[4pt]
			\mathfrak{W}^{(i)}_\varepsilon(y^{(i)}, t) + \chi_\delta^{(i)}(y^{(i)}_1) \, \mathfrak{B}^{(i)}_\varepsilon(y^{(i)}, t)& \text{in} \ \
			\mathcal{R}^{(i)}_{\varepsilon,3\ell_0,\gamma}, \ \ i\in\{2, 3\},
			\\[4pt]
			\mathfrak{N}_\varepsilon(x,t) & \text{in} \ \  \mathcal{R}^{(0)}_{\varepsilon, \gamma},
			\\[4pt]
			\chi_{\ell_0}\big(\frac{y^{(i)}_1}{\varepsilon^\gamma}\big)\, \mathfrak{W}^{(i)}_\varepsilon +
			\Big(1- \chi_{\ell_0}\big(\frac{y^{(i)}_1}{\varepsilon^\gamma}\big)\Big) \mathfrak{N}_\varepsilon & \text{in} \ \
			\mathcal{R}^{(i)}_{\varepsilon,2\ell_0,3\ell_0,\gamma }, \ \ i\in\{1,2,3\}.
		\end{array}
		\right.
	\end{equation}
	Here $t\in[0,T]$, $\gamma\in(\tfrac23,1)$ is fixed, and the subregions
	\begin{gather}
		\mathcal{R}^{(i)}_{\varepsilon,3\ell_0,\gamma}
		:= \mathcal{R}^{(i)}_\varepsilon\cap\{y^{(i)}_1\in[3\ell_0\varepsilon^\gamma,\ell_i]\},
		\label{note_doms}\\
		\mathcal{R}^{(0)}_{\varepsilon,\gamma}
		:= \mathcal{R}^{(0)}_\varepsilon
		\cup\Big(\bigcup_{i=1}^3\mathcal{R}^{(i)}_\varepsilon
		\cap\{y^{(i)}_1\in[\varepsilon\ell_0,2\ell_0\varepsilon^\gamma]\}\Big),\notag\\
		\mathcal{R}^{(i)}_{\varepsilon,2\ell_0,3\ell_0,\gamma}
		:= \mathcal{R}^{(i)}_\varepsilon\cap\{y^{(i)}_1\in[2\ell_0\varepsilon^\gamma,3\ell_0\varepsilon^\gamma]\},
		\label{small-rec}
	\end{gather}
	describe the regions where each asymptotic component dominates.
	
	\subsection{Residuals in the bulk regions}	
	Restricting to the case $i=1$, we calculate residuals arising from substituting $\mathfrak{U}_\varepsilon^{(1)}$ into the differential equation and boundary conditions. Clearly, $\mathfrak{U}_\varepsilon^{(1)}\big|_{t=0}=0$ and $\mathfrak{U}_\varepsilon^{(1)}\big|_{x\in \mathcal{B}_1}=0.$ 
	
	In the region $\{x\in \Omega_\varepsilon^{(1)} \colon \ |x| \ge 4 \varepsilon^{\frac12}\}$ the approximation $\mathfrak{U}_\varepsilon^{(1)} = u_0^{(1)} + \varepsilon u_1^{(1)}$ and 
	the supports of functions $f_1, \, F_1$ and $\{\Psi^{(1,j)}\}$ lie entirely in this region. 
	Direct computations show that, for each $t\in (0,T),$
	\begin{equation}\label{res-reg-0}
		\partial_t \mathfrak{U}_\varepsilon^{(1)} - D_1 \Delta_x \mathfrak{U}_\varepsilon^{(1)} + F_1(\mathfrak{U}_\varepsilon^{(1)},x,t) - f_1(x,t) = \varepsilon \, \partial_t u_1^{(i)}(x,t) + \mathcal{O}(\varepsilon^2),
	\end{equation}
	\begin{remark}\label{rem-3-1}
		Throughout the paper, to avoid cumbersome notation, we do not write out the residuals arising from the application of Taylor's theorem (the Taylor formula) to smooth functions.  
		Instead, we indicate only their order in the uniform norm, which follows from the assumptions on the uniform boundedness of the derivatives of the given functions (see~$\mathbf{A1}$ and~$\mathbf{A4}$). We also omit the explicit phrase "as $\varepsilon \to 0$",
		because this limiting process is inherent in all asymptotic relations used below.
		
		Using the asymptotic behavior \eqref{asym-5}, we obtain
		\begin{equation}
			\sup_{\{x\in \Omega_\varepsilon^{(1)} : \ |x| \ge 4 \varepsilon^{1/2}\}\times(0,T)}
			\varepsilon\, |\partial_t u_1^{(i)}(x,t)|
			\le C_0\, \varepsilon |\log\varepsilon|.
		\end{equation}
	\end{remark}
	
	In the region $\mathcal{D}^{(1)}_{\varepsilon}$ (see \eqref{region-1})
	all right-hand sides in the differential equation and boundary conditions vanish, and the approximation $\mathfrak{U}_\varepsilon^{(1)}$ is given by \eqref{approxim-3-1}. Since both $u_1^{(1)}$ and $\mathcal{K}_1^{(1)}$ are harmonic in this region, we obtain 
	\begin{multline}\label{res-reg-1}
		\partial_t \mathfrak{U}_\varepsilon^{(1)} - D_1 \Delta_x \mathfrak{U}_\varepsilon^{(1)}  = \varepsilon \, \partial_t u_1^{(1)}  + \varepsilon\, \partial_t\mathcal{K}_1^{(1)}
		\\
		-\, \varepsilon \, D_1 \mathrm{div}_x\left(\big(u_1^{(1)} - \mathcal{K}_1^{(1)}\big) \, \nabla_x\big(\chi_b(\tfrac{|x|}{\varepsilon^{1/2}})\big)\right)
		- \varepsilon \, D_1  \nabla_x\big(u_1^{(1)} - \mathcal{K}_1^{(1)}\big)\cdot\nabla_x\big(\chi_b(\tfrac{|x|}{\varepsilon^{1/2}})\big).
	\end{multline}
	Obviously, $\big|\nabla_x \big(\chi_b(\frac{|x|}{\varepsilon^{1/2}})\big)\big| \le C_1 \varepsilon^{-1/2}.$ Taking into account 
	the asymptotics established in Propositions~\ref{prop-2-4} and \ref{propo-2-5}, together with the definition of $\Lambda^{(1)}_1(t)$ (see \eqref{Lambda}), we have
	\begin{equation}\label{as-est-1}
		\varepsilon \, \partial_t u_1^{(1)} = \mathcal{O}\big(\varepsilon|\log\varepsilon|\big), \qquad \varepsilon\, \partial_t\mathcal{K}_1^{(1)} = \mathcal{O}\big(\varepsilon|\log\varepsilon|\big),
	\end{equation}
	\begin{equation}\label{as-est-2}
		u_1^{(1)}(x,t) - \mathcal{K}^{(1)}_1\Big(\frac{x_1 - \varepsilon (\frac{h_2}{\cos \theta} - h_1) \cot\theta}{\varepsilon}, \frac{x_2 - \varepsilon \, h_1}{\varepsilon}, t\Big) = \mathcal{O}(\varepsilon^{1/2}) + \mathcal{O}(\varepsilon^{\delta_1/2}), 
	\end{equation}
	\begin{equation}\label{as-est-3}
		\Big|\nabla_x\Big(u^{(1)}_1(x,t) - \mathcal{K}^{(1)}_1\Big(\frac{x_1 - \varepsilon (\frac{h_2}{\cos \theta} - h_1) \cot\theta}{\varepsilon}, \frac{x_2 - \varepsilon \, h_1}{\varepsilon}, t\Big)\Big)\Big|
		= \mathcal{O}\!\left(\varepsilon^{(\delta_1-1)/2}\right), \qquad x\in\mathcal{D}^{(1)}_\varepsilon.
	\end{equation}
	where $\delta_1 = \pi/\theta_1.$

	In the region $\{x\in \Omega_\varepsilon^{(1)} \colon \ |x| < 2 \varepsilon^{\frac12}\}$ the approximation $\mathfrak{U}_\varepsilon^{(1)} = u_0^{(1)} + \varepsilon \, \mathcal{K}^{(i)}_1.$ Direct computations give that, for each $t\in (0,T),$
	\begin{equation}\label{as-est-4}
		\partial_t \mathfrak{U}_\varepsilon^{(1)} - D_1 \Delta_x \mathfrak{U}_\varepsilon^{(1)}  =  \varepsilon\, \partial_t\mathcal{K}_1^{(1)}.
	\end{equation}
	
	Next, we calculate residuals on the sides $\Gamma_\varepsilon^{(1,1)}$ and $\Gamma_\varepsilon^{(1,2)}.$ As the calculations are similar, we only restrict to the side $\Gamma_\varepsilon^{(1,1)}.$	
	On the part interface $\{x\in \Gamma_\varepsilon^{(1,1)}\colon \ |x| \ge \delta\},$ where $\delta$ is chosen such that the support of $\Psi^{(1,1)}$ belong to this part, we find
	\begin{align}\label{as-est-bd-1}
		D_1\nabla_x \mathfrak{U}_\varepsilon^{(1)} \cdot\boldsymbol{\nu}_\varepsilon  & = D_1\Big( \nabla_x u_0^{(1)}(x,t) - \nabla_x u_0^{(1)}(x_1, 0,t) \Big)  \cdot\boldsymbol{\nu}_\varepsilon \notag 
		\\ &+ \varepsilon D_1 \Big( \nabla_x u_1^{(1)}(x,t) - \nabla_x u_1^{(1)}(x_1, 0,t) \Big)  \cdot\boldsymbol{\nu}_\varepsilon
		+ \, \Psi^{(1,1)}\big(\mathfrak{U}_\varepsilon^{(1)}, \mathfrak{W}^{(1)}_\varepsilon, y_1^{(1)},t\big)\,  + \, \mathcal{O}(\varepsilon^2).
	\end{align}
	Recall that $\boldsymbol{\nu}_\varepsilon$ is the inner unit normal. 	
	Taking into account inclusion \eqref{hoelder} and Remark~\ref{rem-2-8}, we obtain 
	\begin{equation}\label{as-est-bd-2}
		D_1\nabla_x \mathfrak{U}_\varepsilon^{(1)} \cdot\boldsymbol{\nu}_\varepsilon = \Psi^{(1,1)}\big(\mathfrak{U}_\varepsilon^{(1)}, \mathfrak{W}^{(1)}_\varepsilon, y_1^{(1)},t\big)\,  + \, \mathcal{O}(\varepsilon).
	\end{equation}
	
	On $\{x\in \Gamma_\varepsilon^{(1,1)}\colon \ 4 \varepsilon^{1/2} \le |x| \le \delta\},$ we 
	have 
	\begin{equation}\label{as-est-bd-2+}
		D_1\nabla_x \mathfrak{U}_\varepsilon^{(1)} \cdot\boldsymbol{\nu}_\varepsilon   = D_1\nabla_x u_0^{(1)}(x,t) \cdot\boldsymbol{\nu}_\varepsilon + \varepsilon D_1\nabla_x u_1^{(1)}(x,t) \cdot\boldsymbol{\nu}_\varepsilon.
	\end{equation}
	Taking into account  Corollary~\ref{corollarity-2-3}, the last  term in \eqref{as-est-bd-2+} has order $\varepsilon.$
	
	On the interface part $\{x\in \Gamma_\varepsilon^{(1,1)}\colon \ |x| \le 4 \varepsilon^{1/2}\},$ we have $\nabla_x \mathcal{K}^{(1)}_1 \cdot\boldsymbol{\nu}_\varepsilon =0.$ Thus, 
	\begin{equation}\label{as-est-bd-3}
		D_1 \nabla_x \mathfrak{U}_\varepsilon^{(1)} \cdot\boldsymbol{\nu}_0 = D_1\nabla_x u_0^{(1)}\cdot\boldsymbol{\nu}_0  + \varepsilon \, D_1 \nabla_x\left(\chi_b(\tfrac{|x|}{\varepsilon^{1/2}}) \, \big(u_1^{(1)} - \mathcal{K}_1^{(1)}\big)\right) \cdot\boldsymbol{\nu}_\varepsilon	
	\end{equation}
	on $\{x\in \Gamma_\varepsilon^{(1,1)}\colon \ 2 \varepsilon^{1/2} \le |x| \le 4 \varepsilon^{1/2}\},$ and 
	\begin{equation}\label{as-est-bd-4}
		D_1 \nabla_x \mathfrak{U}_\varepsilon^{(1)} \cdot\boldsymbol{\nu}_\varepsilon = D_1\nabla_x u_0^{(1)}\cdot\boldsymbol{\nu}_\varepsilon \quad \text{on} \ \
		\{x\in \Gamma_\varepsilon^{(1,1)}\colon \ |x| \le 2 \varepsilon^{1/2}\}.
	\end{equation}
	
	Now it remains to find residuals on $\Gamma_\varepsilon^{(1,0)}.$ Taking into account the boundary condition for $\mathcal{K}^{(1)}_1$ on~$\mathfrak{I}_0$ (see \eqref{corn-probl+0}), we find
	\begin{equation}\label{as-est-bd-5}
		D_1\nabla_x \mathfrak{U}_\varepsilon^{(1)} \cdot\boldsymbol{\nu}_\varepsilon = D_1\nabla_x u_0^{(1)}\cdot\boldsymbol{\nu}_\varepsilon + \Psi_0\big(\mathfrak{U}_\varepsilon^{(1)}, \mathfrak{W}^{(1)}_\varepsilon, \tfrac{x}{\varepsilon},t\big)\,  + \, \mathcal{O}(\varepsilon).
	\end{equation}
	\begin{remark}
		From \eqref{approxim-3-1}, \eqref{as-est-2} and \eqref{as-est-3}, together with the smoothness assumptions on the corresponding right-hand sides, we have that $\mathfrak{U}^{(1)}_\varepsilon	
		\in W^{2,1}_2(\Omega^{(1)}_\varepsilon\times(0,T));$ and the same regularity holds for $\mathfrak{U}^{(2)}_\varepsilon$ and $\mathfrak{U}^{(3)}_\varepsilon.$  
	\end{remark}
	
	
	\subsection{Residuals in the thin graph-like junction $\mathcal{R}_\varepsilon$}
	
	Here we compute residuals produced when $\mathfrak{A}_\varepsilon$ (see \eqref{first_app}) is substituted into the differential equations and 
	boundary conditions of the original problem $\mathbb{P}_\varepsilon$.
	
	Obviously,
	$\mathfrak{A}_\varepsilon\big|_{t=0} = 0$ and $\mathfrak{A}_\varepsilon\big|_{y^{(i)}_1=\ell_i} = q_i(t), \ i\in \{1,2,3\}.$
	In $\mathcal{R}^{(1)}_{\varepsilon,3\ell_0,\gamma}$,  the approximation
	$\mathfrak{A}_\varepsilon$ satisfies relation \eqref{eq-thin-rect+}. For $ i\in\{2, 3\} $ we should additionally calculate  residuals from the boundary-layer:
	\begin{multline}\label{Res_3}
		\chi_\delta^{(i)}(y^{(i)}_1) \, \partial_t\mathfrak{B}^{(i)}_\varepsilon(y^{(i)}, t) -  \varepsilon\, \Delta_x \big(\chi_\delta^{(i)}(y^{(i)}_1) \, \mathfrak{B}^{(i)}_\varepsilon\big) +
		v_1^{(i)}(\ell_i) \, \partial_{y^{(i)}_1} \big(\chi_\delta^{(i)}(y^{(i)}_1) \, \mathfrak{B}^{(i)}_\varepsilon\big)
		\\
		= \varepsilon^2\, \chi_\delta^{(i)} \, \partial_t \Pi_2^{(i)}
		- \varepsilon \,  \big(\chi_\delta^{(i)}\big)^{\prime\prime}  \sum\limits_{k=0}^{2} \varepsilon^{k} \Pi_k^{(i)}
		+ 2  \big(\chi_\delta^{(i)}\big)'  \sum\limits_{k=0}^{2} \varepsilon^{k} \partial_{\eta_1} \Pi_k^{(3)}
		+  v_1^{(i)}  \,  \big(\chi_\delta^{(i)}\big)'  \sum\limits_{k=0}^{2} \varepsilon^{k} \Pi_k^{(i)}.
	\end{multline}
	The supports of terms  in the second line of \eqref{Res_3} coincide with $\mathrm{supp}\big(\big(\chi_\delta^{(i)}\big)'\big),$ where
	the functions $\{\Pi_k^{(i)}\}_{k=0}^{2}$ exponentially decay as $\varepsilon$ tends to zero. Therefore, in
	$\mathcal{R}^{(i)}_{\varepsilon,3\ell_0,\gamma},$ $i\in\{2,3\},$ the approximation
	$\mathfrak{A}_\varepsilon$ leaves residual $\varepsilon^2 \, \digamma^{(i)}_\varepsilon,$ where $\digamma^{(i)}_\varepsilon$ is uniformly bounded similarly as $\digamma^{(1)}_\varepsilon$ (see \eqref{uniform-est-1}).
	
	Based on calculations in \S~\ref{Par-3-1}, we obtain 
	$$
	-\Big(-\partial_{y^{(1)}_2}\mathfrak{W}^{(1)}_\varepsilon
	+ \mathfrak{W}^{(1)}_\varepsilon \,v^{(1)}_2\Big) = 
	\varepsilon\,\Phi^{(1,1)}\!\left(\mathfrak{U}^{(1)}_\varepsilon,
	\mathfrak{W}^{(1)}_\varepsilon,y^{(1)}_1,t\right)
	+ \mathcal{O}(\varepsilon^2) \quad \text{on} \ \ \Gamma_\varepsilon^{(1,1)}, \ \ y^{(1)}_1\in[3\ell_0\varepsilon^\gamma,\ell_1].
	$$
	Similar relations are valid on $\Gamma_\varepsilon^{(i,j)}$, for $y^{(j)}_1\in[3\ell_0\varepsilon^\gamma,\ell_j],$ for the indices $(i,j)$ specified in \eqref{indexes}.		
	
	In the neighborhood $\mathcal{R}^{(0)}_{\varepsilon,\gamma}$ of the node 
	$\mathcal{R}^{(0)}_{\varepsilon}$ we compute the residual generated by the 
	approximation~$\mathfrak{A}_\varepsilon,$ which coincides with $\mathfrak{N}_\varepsilon$ in this region. Substituting $\mathfrak{N}_\varepsilon$ 
	into the differential equation gives
	\begin{equation}\label{Res_6}
		\partial_t\,\mathfrak{N}_\varepsilon -  \varepsilon\, \Delta_x  \mathfrak{N}_\varepsilon +
		\mathrm{div} \big( \overrightarrow{V_\varepsilon}(x) \, \mathfrak{N}_\varepsilon\big)
		= \varepsilon\, \digamma^{(0)}_\varepsilon, 
	\end{equation}
	where $\digamma^{(0)}_\varepsilon = \partial_t N_1.$
	Inserting $\mathfrak{N}_\varepsilon$ into the boundary condition on 
	$\Gamma^{(i,0)}_\varepsilon$ yields
	\begin{equation}\label{Res_7}
		-   \partial_{\boldsymbol{\nu}_\varepsilon} \mathfrak{N}_\varepsilon  = \Phi_0\big(\mathfrak{U}^{(i)}_{\varepsilon}, \mathfrak{N}_\varepsilon,\tfrac{x}{\varepsilon}, t \big) + \varepsilon \, \gimel^{(i)}_\varepsilon(x,t) 
		\ \   \text{on} \ \   \Gamma^{(i,0)}_\varepsilon\times(0,T), 
	\end{equation}
	where $\gimel^{(i)}_\varepsilon$ denotes the residual term. The Lipschitz continuity of $\Phi_0$ 
	(assumption~$\mathbf{A3}$), the Hölder regularity of $u_0^{(i)}$ 
	(Proposition~\ref{prop-2-1}), and the boundedness of $N_1$ and $\mathcal{K}^{(i)}_1$ on 
	$\Gamma^{(i,0)}_\varepsilon\times(0,T)$ imply that this residual remains uniformly bounded:
	\begin{equation}\label{Res_5}
		\sup_{\Gamma^{(i,0)}_\varepsilon\times(0,T)}
		|\gimel^{(i)}_\varepsilon(x,t)| \le C_1.
	\end{equation}
	
	On the remaining part of $\partial\mathcal{R}^{(0)}_{\varepsilon,\gamma}$ the residual 
	vanishes, because $\Phi_0=\Phi^{(i,j)}=0$ (assumptions~$\mathbf{A3}$–$\mathbf{A4}$) and 
	$\partial_{\boldsymbol{\nu}_\varepsilon}N_0
	=\partial_{\boldsymbol{\nu}_\varepsilon}N_1=0$ there 
	(see~\S\ref{subsec_Inner_part}).
	
	Now consider the region $\mathcal{R}^{(i)}_{\varepsilon,2\ell_0,3\ell_0,\gamma}$ (see~\eqref{small-rec}), where $\mathfrak{A}_\varepsilon = \chi_{\ell_0}\big(\frac{y^{(i)}_1}{\varepsilon^\gamma}\big)\, \mathfrak{W}^{(i)}_\varepsilon +
	\big(1- \chi_{\ell_0}\big(\frac{y^{(i)}_1}{\varepsilon^\gamma}\big)\big) \mathfrak{N}_\varepsilon.$

	Based on the same arguments as above and using the second statement of 
	Remark~\ref{rem-2-3} concerning the coefficients 
	$\{z^{(i)}_1\}_{i=1}^3$ and $\{z^{(i)}_2\}_{i=1}^3$, we obtain
	\begin{equation}\label{Res_8}
		\partial_{y^{(i)}_2}\mathfrak{A}_\varepsilon = 0
	\end{equation}
	on the lateral sides of $\mathcal{R}^{(i)}_{\varepsilon,2\ell_0,3\ell_0,\gamma}.$

	Using \eqref{eq-thin-rect+}, \eqref{uniform-est-1}, and \eqref{Res_6}, we compute
	\begin{multline}\label{Res_9}
		\partial_t \mathfrak{A}_\varepsilon
		- \varepsilon\, \Delta \mathfrak{A}_\varepsilon
		+ \mathrm{v}_i\, \partial_{y^{(i)}_1}\mathfrak{A}_\varepsilon
		=
		-\,\varepsilon^2\, 
		\chi_{\ell_0}\!\Big(\frac{y^{(i)}_1}{\varepsilon^\gamma}\Big)
		\big(w^{(i)}_{1}(y^{(i)}_1,t)\big)^{\prime\prime}
		+ \varepsilon\Big(1-\chi_{\ell_0}\!\Big(\frac{y^{(i)}_1}{\varepsilon^\gamma}\Big)\Big)\partial_t N_1
		\\
		-\,\chi_{\ell_0}'' 
		\sum_{k=0}^{1}\varepsilon^{k+1-2\gamma}\big(w_k^{(i)}(y^{(i)}_1,t)-N_k\big)
		- 2\chi_{\ell_0}' 
		\sum_{k=0}^{1}\varepsilon^{k+1-\gamma}
		\Big(\partial_{y^{(i)}_1}w_k^{(i)}(y^{(i)}_1,t)
		- \varepsilon^{-1}\partial_{\xi_1^{(i)}}N_k\Big)
		\\
		+\,\mathrm{v}_i\,\chi_{\ell_0}'
		\sum_{k=0}^{1}\varepsilon^{k-\gamma}\big(w_k^{(i)}(y^{(i)}_1,t)-N_k\big)
		\quad\text{in }\ \ \mathcal{R}^{(i)}_{\varepsilon,2\ell_0,3\ell_0,\gamma}.
	\end{multline}
	
	All terms in the second and third lines of \eqref{Res_9} are supported in 
	$\operatorname{supp}(\chi_{\ell_0}')$.  
	Using the Taylor expansion of $w_0^{(i)}$ and $w_1^{(i)}$ at $y^{(i)}_1=0$ and 
	formula~\eqref{new-solution_k} for $N_1,$  these terms can be rewritten as
	\begin{equation*}
		\chi_{\ell_0}''\, \varepsilon^{2-2\gamma}\widetilde{N}_1\big(\tfrac{x}{\varepsilon},t\big)
		+ 2\chi_{\ell_0}' \, \varepsilon^{1-\gamma}
		\partial_{\xi_1^{(i)}}\widetilde{N}_1(\xi,t)\big|_{\xi=x/\varepsilon}
		- \mathrm{v}_i\,\chi_{\ell_0}' \, 
		\varepsilon^{1-\gamma}\widetilde{N}_1\big(\tfrac{x}{\varepsilon},t\big)
		+ \mathcal{O}(\varepsilon^\gamma).
	\end{equation*}
	By \eqref{new-solution_k} and \eqref{exp-decrease+1}, the maxima of 
	$|\widetilde{N}_k|$ and $|\partial_{\xi_i}\widetilde{N}_k|$ over
	\[
	\mathcal{R}^{(i)}_\varepsilon \cap 
	\big\{y^{(i)}: y^{(i)}_1\in[2\ell_0\varepsilon^\gamma,\,3\ell_0\varepsilon^\gamma]\big\}
	\times[0,T]
	\]
	are of order $\exp(-\beta_0\,2\ell_0\,\varepsilon^{\gamma-1})$, i.e., they decay exponentially as 
	$\varepsilon\to 0$.  
	Consequently, the right-hand side of \eqref{Res_9} is of order $\varepsilon^\gamma$ in the uniform norm. We denote it by $\varepsilon^\gamma\, \digamma^{(i)}_\varepsilon.$ 
	No ambiguity arises from using the same notation as in the differential equation on $\mathcal{R}^{(i)}_{\varepsilon,3\ell_0,\gamma},$
	since they are given  on different domains.
	
	Summing-up  calculations in \S~\ref{Par-3-1}, \S~\ref{subsec_Bound_layer}  and in this one, we get the statement.

	\begin{proposition}\label{Prop-3-1}
		There exists a positive number $\varepsilon_0$ such that, for all 
		$\varepsilon\in(0,\varepsilon_0)$, the difference between the approximation 
		\eqref{first_app} and the junction component $w_\varepsilon$ of the solution to 
		problem $\mathbb{P}_\varepsilon$ satisfies, for all $t\in(0,T)$ and $i\in\{1,2,3\},$ the relations
		\begin{equation}\label{dif_3}
			\partial_t(\mathfrak{A}_\varepsilon - w_\varepsilon)
			- \varepsilon\,\Delta_{y^{(i)}}(\mathfrak{A}_\varepsilon - w_\varepsilon)
			+ \mathrm{div}_{y^{(i)}}\!\big(\overrightarrow{V_\varepsilon}^{(i)}\,
			(\mathfrak{A}_\varepsilon - w_\varepsilon)\big)
			= \varepsilon^2\,\digamma^{(i)}_\varepsilon
			\quad\text{in } \ \mathcal{R}^{(i)}_{\varepsilon,3\ell_0,\gamma},
		\end{equation}
		\begin{equation}\label{dif_2}
			\partial_t(\mathfrak{A}_\varepsilon - w_\varepsilon)
			- \varepsilon\,\Delta_{y^{(i)}}(\mathfrak{A}_\varepsilon - w_\varepsilon)
			+ \mathrm{v}_i\,\partial_{y^{(i)}_1}(\mathfrak{A}_\varepsilon - w_\varepsilon)
			= \varepsilon^{\gamma}\,\digamma^{(i)}_\varepsilon
			\quad\text{in } \ \mathcal{R}^{(i)}_{\varepsilon,2\ell_0,3\ell_0,\gamma},
		\end{equation}
		\begin{equation}\label{dif_1}
			\partial_t(\mathfrak{A}_\varepsilon - w_\varepsilon)
			- \varepsilon\,\Delta_x(\mathfrak{A}_\varepsilon - w_\varepsilon)
			+ \mathrm{div}\!\big(\overrightarrow{V_\varepsilon}(x)\,
			(\mathfrak{A}_\varepsilon - w_\varepsilon)\big)
			= \varepsilon\,\digamma^{(0)}_\varepsilon
			\quad\text{in } \ \mathcal{R}^{(0)}_{\varepsilon,\gamma}.
		\end{equation}
		
		On the node interface $\Gamma^{(i,0)}_\varepsilon$ we have
		\begin{equation}\label{bc_1+}
			-\,\partial_{\boldsymbol{\nu}_\varepsilon}(\mathfrak{A}_\varepsilon - w_\varepsilon)
			- \Phi_0\!\left(\mathfrak{U}^{(i)}_\varepsilon,\mathfrak{A}_\varepsilon,
			\tfrac{x}{\varepsilon},t\right)
			+ \Phi_0\!\left(u^{(i)}_\varepsilon,w_\varepsilon,
			\tfrac{x}{\varepsilon},t\right)
			= \varepsilon\,\gimel^{(i)}_\varepsilon.
		\end{equation}
		
		On the lateral sides of $\mathcal{R}^{(i)}_{\varepsilon,2\ell_0,3\ell_0,\gamma}$ we have
		\begin{equation}\label{bc_2}
			-\,\partial_{y^{(j)}_2}(\mathfrak{A}_\varepsilon - w_\varepsilon)\big|_{y^{(j)}_2=\pm\varepsilon h_j}
			= 0,
			\qquad y^{(j)}_1\in[\ell_0\varepsilon,\,3\ell_0\varepsilon^\gamma].
		\end{equation}
		
		On $\Gamma_\varepsilon^{(j,j)}$, for $y^{(j)}_1\in[3\ell_0\varepsilon^\gamma,\ell_j]$, we have
		\begin{equation}\label{bc-3}
			-\Big(-\partial_{y^{(j)}_2}(\mathfrak{A}_\varepsilon - w_\varepsilon)
			+ (\mathfrak{A}_\varepsilon - w_\varepsilon)\,v^{(j)}_2\Big)
			- \varepsilon\,\Phi^{(j,j)}\!\left(\mathfrak{U}^{(j)}_\varepsilon,
			\mathfrak{A}^{(j)}_\varepsilon,y^{(j)}_1,t\right)
			+ \varepsilon\,\Phi^{(j,j)}\!\left(u^{(j)}_\varepsilon,w_\varepsilon,
			y^{(j)}_1,t\right)
			= \varepsilon^2\,\gimel^{(j,j)}_\varepsilon.
		\end{equation}
		
		On $\Gamma_\varepsilon^{(i,j)}$, for $y^{(j)}_1\in[3\ell_0\varepsilon^\gamma,\ell_j]$, we have
		\begin{equation}\label{bc-4}
			\Big(-\partial_{y^{(j)}_2}(\mathfrak{A}_\varepsilon - w_\varepsilon)
			+ (\mathfrak{A}_\varepsilon - w_\varepsilon)\,v^{(j)}_2\Big)
			- \varepsilon\,\Phi^{(i,j)}\!\left(\mathfrak{U}^{(i)}_\varepsilon,
			\mathfrak{A}^{(j)}_\varepsilon,y^{(j)}_1,t\right)
			+ \varepsilon\,\Phi^{(i,j)}\!\left(u^{(i)}_\varepsilon,w_\varepsilon,
			y^{(j)}_1,t\right)
			= \varepsilon^2\,\gimel^{(i,j)}_\varepsilon.
		\end{equation}
		
		Finally,
		\[
		(\mathfrak{A}_\varepsilon - w_\varepsilon)\big|_{x_i=\ell_i}
		= 0 \quad\text{on }\Upsilon_\varepsilon^{(i)}(\ell_i),\ i\in\{1,2,3\},
		\qquad
		(\mathfrak{A}_\varepsilon - w_\varepsilon)\big|_{t=0}
		= 0 \quad\text{in } \ \mathcal{R}_\varepsilon.
		\]
		
		The functions 
		$\digamma^{(i)}_\varepsilon$, $\digamma^{(0)}_\varepsilon$, 
		$\gimel^{(i)}_\varepsilon$, and $\gimel^{(i,j)}_\varepsilon$ 
		appearing in the right-hand sides of 
		\eqref{dif_3}–\eqref{bc-4} are uniformly bounded on their respective domains of 
		definition by constants independent of $\varepsilon$. Here the indices $(i,j)$ range over the admissible pairs specified in \eqref{indexes}.	
	\end{proposition}

	\subsection{Justification} 
	Here we prove the asymptotic estimates for the difference between the solution 
	\begin{equation*}
		\begin{cases}
			u_\varepsilon^{(i)} & \text{in} \ \ \in\Omega_\varepsilon^{(i)}\times(0,T),\quad i\in\{1,2,3\},\\[3pt]
			w_\varepsilon & \text{in} \ \ \mathcal{R}_\varepsilon\times(0,T),
		\end{cases}
	\end{equation*}	
	to the original problem $\mathbb{P}_\varepsilon$ and the approximation function
	\begin{equation*}
		\begin{cases}
			\mathfrak{U}^{(i)}_\varepsilon & \text{in} \ \ \Omega_\varepsilon^{(i)}\times(0,T),\quad i\in\{1,2,3\},\\[3pt]
			\mathfrak{A}_\varepsilon  & \text{in} \ \ \mathcal{R}_\varepsilon\times(0,T),
		\end{cases}
	\end{equation*}	
	constructed by formulas \eqref{approxim-3-1} and \eqref{first_app}, respectively.

	\begin{theorem}\label{apriory_estimate} Let all the assumptions of Section~\ref{Sec-2} be satisfied.\
		Then for any $\mu_i \in (1, \tfrac{\pi}{\theta_i}),$  $i\in \{1,2,3\},$ there exist positive constants 
		$\tilde{C}_1$, $\tilde{C}_2$, $\tilde{C}_3$ and $\varepsilon_0$ such that, for all 
		$\varepsilon \in (0,\varepsilon_0)$, the following estimates hold:  
		\begin{equation}\label{main-1}
			\tfrac{1}{\sqrt{|\mathcal{R}_\varepsilon|}}\,\max_{t\in[0,T]} \big\|\mathfrak{A}_\varepsilon(\cdot,t) - w_\varepsilon(\cdot,t)\big\|_{L^2(\mathcal{R}_\varepsilon)} \le \tilde{C}_1 \Big(\varepsilon +  \sum_{i=1}^{3} \varepsilon^{\mu_i - \frac12} \sqrt{|\log\varepsilon|}\Big),
		\end{equation}
		\begin{equation}\label{main-2}
			\tfrac{1}{\sqrt{|\mathcal{R}_\varepsilon|}}\, \big\| \nabla\mathfrak{A}_\varepsilon - \nabla w_\varepsilon\big\|_{L^2(\mathcal{R}_\varepsilon\times (0, T))} \le \tilde{C}_2
			\Big(\varepsilon^{1/2} + \sum_{i=1}^{3} \varepsilon^{\mu_i - 1} \sqrt{|\log\varepsilon|} 
			\Big),
		\end{equation}	
		\begin{multline}\label{main-3}
			\max_{t\in[0,T]}\big\|\mathfrak{U}^{(i)}_\varepsilon(\cdot,t) - u^{(i)}_\varepsilon(\cdot,t)\big\|_{L^2(\Omega_\varepsilon^{(i)})} +	\big\|\nabla \mathfrak{U}^{(i)}_\varepsilon - \nabla u^{(i)}_\varepsilon\big\|_{L^2(\Omega_\varepsilon^{(i)}\times(0,T))}
			\\
			\le \tilde{C}_3 \Big(\varepsilon^{\mu_i -1} + \varepsilon |\log\varepsilon| + \sum_{\imath =1}^{3} \varepsilon^{\mu_\imath - \frac12} \sqrt{|\log\varepsilon|}
			\Big) \quad \text{for} \ \ i\in \{1,2,3\}.
		\end{multline}
		Here  $|\mathcal{R}_\varepsilon|$ is  the Lebesgue measure of $\mathcal{R}_\varepsilon.$
	\end{theorem}
	\begin{remark}
		Hereinafter, constants in all inequalities are positive and independent of the parameter $\varepsilon;$ constants with identical indices may differ between inequalities.
	\end{remark}
	\begin{proof} 	\textbf{1.} \emph{Estimates in $\Omega_\varepsilon^{(i)}\times (0,T)$.}
		We present the proof in $\Omega_\varepsilon^{(1)}\times (0,T)$; the arguments in 
		$\Omega_\varepsilon^{(2)}\times (0,T)$ and $\Omega_\varepsilon^{(3)}\times (0,T)$
		are identical.  Subtracting  the corresponding differential equation of problem $\mathbb{P}_\varepsilon$ from \eqref{res-reg-0}, \eqref{res-reg-1} and \eqref{as-est-4},  multiplying the resulting identities by $Z_\varepsilon := \mathfrak{U}^{(1)}_\varepsilon - u^{(1)}_\varepsilon,$ and integrating by parts using \eqref{as-est-bd-2} -- \eqref{as-est-bd-5} and applying  the mean value theorem to the nonlinear terms, we obtain 
		\begin{multline}\label{th-1}
			\frac{1}{2} \int_{\Omega_\varepsilon^{(1)}}Z_\varepsilon^2\big|_{t=\tau}\, dx + D_1 \|\nabla Z_\varepsilon\|^2_{L^2(\Omega_\varepsilon^{(1)}\times(0,\tau))} + \int\limits_{\Omega_\varepsilon^{(1)}\times(0,\tau)} \partial_s F_1 \, Z_\varepsilon^2 \, dxdt
			\\
			+ \sum_{j=1}^{2} \int\limits_{\Gamma_\varepsilon^{(1,j)}\times(0,\tau)}\Big(\partial_s \Psi^{(1,j)} \, Z_\varepsilon +  \partial_w \Psi^{(1,j)} \, Y_\varepsilon\Big) Z_\varepsilon \, dl_xdt +  \sum_{j=1}^{2} \int\limits_{(\Gamma_\varepsilon^{(1,j)} \cap \{|x| \le \delta\})\times(0,\tau)}\big(D_1\nabla_x u_0^{(1)}\cdot\boldsymbol{\nu}_\varepsilon \big) \, 
			Z_\varepsilon \, dl_xdt 
			\\
			+ \int\limits_{\Gamma_\varepsilon^{(1,0)}\times(0,\tau)}\Big(\partial_s \Psi_0 \, Z_\varepsilon +  \partial_w \Psi_0 \, Y_\varepsilon\Big) Z_\varepsilon \, dl_xdt +  \int\limits_{\Gamma_\varepsilon^{(1,0)} \times(0,\tau)}\Big(D_1\nabla_x u_0^{(1)}\cdot\boldsymbol{\nu}_\varepsilon \Big) Z_\varepsilon \, dl_xdt 
			\\
			= \varepsilon\,  D_1 \int\limits_{\mathcal{D}^{(1)}_\varepsilon \times (0,\tau)}\Big(\big(u_1^{(1)} - \mathcal{K}_1^{(1)}\big) \, \nabla_x\big(\chi_b(\tfrac{|x|}{\varepsilon^{1/2}})\cdot \nabla_x Z_\varepsilon - Z_\varepsilon \nabla_x\big(u_1^{(1)} - \mathcal{K}_1^{(1)}\big)\cdot\nabla_x\big(\chi_b(\tfrac{|x|}{\varepsilon^{1/2}})\Big)dxdt
			\\
			+ \varepsilon |\log\varepsilon| \int\limits_{\Omega_\varepsilon^{(1)}\times(0,\tau)}\mathfrak{O}_\varepsilon^{(1)} \, Z_\varepsilon \, dxdt + \varepsilon \sum_{j=0}^{2} \, \int\limits_{\Gamma_\varepsilon^{(1,j)}\times(0,\tau)}\mathfrak{O}_\varepsilon^{(1,j)} \, Z_\varepsilon \, dl_xdt.
		\end{multline} 
		Here, $Y_\varepsilon := \mathfrak{A}_\varepsilon - w_\varepsilon,$	$\delta$ is chosen before \eqref{as-est-bd-1}, the integrands in the last line are uniformly bounded residuals (see Remark~\ref{rem-3-1}, 
		\eqref{as-est-1}, \eqref{as-est-4}, \eqref{as-est-bd-2}, \eqref{as-est-bd-2+}, \eqref{as-est-bd-5}) and we take the smallest order of $\varepsilon.$

		On $\mathcal{I}_j \cap \{|x| \le \delta\}, \ j\in \{1,2\},$ the normal derivative of $u_0^{(1)}$ is equal to zero.  
		Therefore, taking into account the Hölder regularity of $u_0^{(1)}$ (see \eqref{hoelder-0}), we derive
		\begin{equation}\label{th-2}
			\bigg|\int_{(\Gamma_\varepsilon^{(1,j)} \cap \{|x| \le \delta\})\times(0,\tau)}\big(D_1\nabla_x u_0^{(1)}\cdot\boldsymbol{\nu}_\varepsilon \big)\,  Z_\varepsilon \, dl_xdt\bigg| \le C_1 \, \varepsilon^{\mu_1 -1}\, \|\nabla Z_\varepsilon\|_{L^2(\Omega_\varepsilon^{(1)}\times(0,\tau))}, 
		\end{equation}
		where $\mu_1 \in (1, \tfrac{\pi}{\theta_1}).$ 
		
		Based on the Hölder regularity of $u_0^{(1)}$ and \eqref{asym-in-coner},  
		the normal derivative $
		\nabla_x u_0^{(1)}\cdot\boldsymbol{\nu}_\varepsilon  
		$ on $\Gamma_\varepsilon^{(1,0)}$ is of order $\varepsilon^{\tfrac{\pi}{\theta_1} -1}.$
		Using the fact that $Z_\varepsilon|_{x\in \mathcal{B}_1} = 0,$ it is easy to verify that 
		\begin{equation}\label{trace-ineq}
			\|Z_\varepsilon\|_{L^2(\Gamma_\varepsilon^{(1,0)})} \le C_2 \, \sqrt{\varepsilon \, |\log\varepsilon|} \, 
			\|\nabla Z_\varepsilon\|_{L^2(\Omega_\varepsilon^{(1)})}.
		\end{equation}
		Therefore,  
		\begin{align}\label{th-3}
			\bigg|\int_{\Gamma_\varepsilon^{(1,0)} \times(0,\tau)}\Big(D_1\nabla_x u_0^{(1)}\cdot\boldsymbol{\nu}_\varepsilon \Big) Z_\varepsilon \, dl_xdt\bigg| &\le C_3 \, \varepsilon^{\tfrac{\pi}{\theta_1} -1} \, \varepsilon^{1/2} \, \big(\varepsilon \, |\log\varepsilon|\big)^{1/2}  \|\nabla Z_\varepsilon\|_{L^2(\Omega_\varepsilon^{(1)}\times(0,\tau))} \notag
			\\
			& = C_3 \, \varepsilon^{\tfrac{\pi}{\theta_1}} \big(|\log\varepsilon|\big)^{1/2} \, \|\nabla Z_\varepsilon\|_{L^2(\Omega_\varepsilon^{(1)}\times(0,\tau))}. 
		\end{align}
		
		Based on assumptions ${\bf A1},$ ${\bf A3}$ and ${\bf A4}$ for the functions $F_i,$ $\Psi^{(i,j)},$ $\Psi_0,$ and using \eqref{th-2}, \eqref{trace-ineq}, \eqref{th-3}, \eqref{as-est-2} and \eqref{as-est-3}, we first derive from \eqref{th-1} the  inequality 
		\begin{equation}\label{th-4}
			\|\nabla Z_\varepsilon\|_{L^2(\Omega_\varepsilon^{(1)}\times(0,\tau))} \le C_3 \Big(\varepsilon^{\mu_1 -1} + \varepsilon |\log\varepsilon| + \sum_{j=1}^{2} \|Y_\varepsilon\|_{L^2(\Gamma_\varepsilon^{(1,j)}\times (0,\tau))} +
			\sqrt{\varepsilon \, |\log\varepsilon|} \,\|Y_\varepsilon\|_{L^2(\Gamma_\varepsilon^{(1,0)}\times (0,\tau))}
			\Big).
		\end{equation}
		Substituting it back into \eqref{th-1} yields  
		\begin{align*}
			\int_{\Omega_\varepsilon^{(1)}}Z_\varepsilon^2\big|_{t=\tau}\, dx  & \le C_4 \Big(
			\varepsilon |\log\varepsilon| \int_{\Omega_\varepsilon^{(1)}\times(0,\tau)} Z^2_\varepsilon \, dxdt 
			+ 
			\varepsilon^{2(\mu_1 -1)} + \varepsilon^2 |\log\varepsilon|^2 \notag
			\\
			&  + \sum_{j=1}^{2} \|Y_\varepsilon\|^2_{L^2(\Gamma_\varepsilon^{(1,j)}\times (0,\tau))} + \varepsilon \, |\log\varepsilon| \,\|Y_\varepsilon\|^2_{L^2(\Gamma_\varepsilon^{(1,0)}\times (0,\tau))}\Big),
		\end{align*}
		and Gronwall's inequality gives 
		\begin{equation}\label{th-5+}
			\max_{t\in[0,T]}\|Z_\varepsilon(\cdot,t)\|_{L^2(\Omega_\varepsilon^{(1)})} \le C_5 \Big(
			\varepsilon^{\mu_1 -1} + \varepsilon |\log\varepsilon| 
			+ \sum_{j=1}^{2} \|Y_\varepsilon\|_{L^2(\Gamma_\varepsilon^{(1,j)}\times (0,T))} + \sqrt{\varepsilon \, |\log\varepsilon|} \,\|Y_\varepsilon\|_{L^2(\Gamma_\varepsilon^{(1,0)}\times (0,\tau))}
			\Big).
		\end{equation}

		
		\textbf{2.} \emph{Estimates in $\mathcal{R}_\varepsilon\times (0,T)$.} 
		Multiplying \eqref{dif_3}--\eqref{dif_1} by  $Y_\varepsilon = \mathfrak{A}_\varepsilon - w_\varepsilon,$  integrating over the corresponding domain and over $ (0, \tau),$ where $\tau$ is an arbitrary number from $(0, T),$ and integrating by parts and taking \eqref{lap_1}--\eqref{grad_1} and the boundary conditions and initial condition into account, we obtain
		\begin{align}\label{th-6}
			\frac{1}{2}\, \big\|Y_\varepsilon\big|_{t=\tau}\big\|^2_{L^2(\mathcal{R}_\varepsilon)}
			& + \varepsilon \, \big\| {\nabla Y_\varepsilon}\big\|^2_{L^2(\mathcal{R}_\varepsilon\times (0, \tau))} \notag
			\\
			&+ 
			\varepsilon^2 \sum_{i=1}^{3} \sum_j\int_{0}^{\tau}\bigg(\int_{\Gamma^{(i,j)}_\varepsilon}
			\Big(\partial_s\Phi^{(i,j)}\, \big(\mathfrak{U}^{(i)}_\varepsilon - u^{(i)}_\varepsilon\big) + \partial_w\Phi^{(i,j)}\, Y_\varepsilon  \Big) Y_\varepsilon \, d y_1^{(j)}\bigg)dt \notag
			\\
			&+ \varepsilon \sum_{i=1}^{3}\int_{0}^{\tau}\bigg(\int_{\Gamma^{(i,0)}_\varepsilon}\Big(
			\partial_s\Phi_0\, \big(\mathfrak{U}^{(i)}_\varepsilon - u^{(i)}_\varepsilon\big) + \partial_w\Phi_0\, Y_\varepsilon\Big) Y_\varepsilon \, dl_x\bigg)dt \notag
			\\
			& = 
			\int_{0}^{\tau} \bigg(\int_{\mathcal{R}^{(0)}_\varepsilon} Y_\varepsilon\,
			\overrightarrow{V_\varepsilon} \cdot
			{\nabla Y_\varepsilon} \, dx + \sum_{i=1}^{3}\int_{\mathcal{R}^{(i)}_\varepsilon} Y_\varepsilon\,
			\overrightarrow{V_\varepsilon} \cdot
			{\nabla Y_\varepsilon} \, dy^{(i)}\bigg) dt \notag
			\\
			& +
			\int_{0}^{\tau}\bigg(
			\varepsilon \int\limits_{\mathcal{R}_{\varepsilon, \gamma}^{(0)}}  \digamma^{(0)}_\varepsilon \, Y_\varepsilon \, dx
			+ \varepsilon^{{\gamma}} \sum_{i=1}^{3} \int\limits_{\mathcal{R}^{(i)}_{\varepsilon,2\ell_0,3\ell_0,\gamma}}
			\digamma^{(i)}_\varepsilon \, Y_\varepsilon \, dy^{(i)} + \varepsilon^2 \sum_{i=1}^{3} \int\limits_{\mathcal{R}^{(i)}_{\varepsilon,3\ell_0,\gamma}}  \digamma^{(i)}_\varepsilon \, Y_\varepsilon \, dy^{(i)}\bigg)dt \notag
			\\
			& - \varepsilon^2 \sum_{i=1}^{3}\int_{0}^{\tau} \int_{\Gamma^{(i,0)}_\varepsilon}\gimel^{(i)}_\varepsilon \,  Y_\varepsilon \, dl_x\, dt
			- \varepsilon^3 \sum_{i=1}^{3} \sum_j\int_{0}^{\tau} \int_{\Gamma^{(i,j)}_\varepsilon}
			\gimel^{(i,j)}_\varepsilon \,  Y_\varepsilon \, d y_1^{(j)} \, dt .
		\end{align}
		To estimate integrals in the second, third and the last lines of \eqref{th-6} we will use the following inequalities:
		\begin{gather}\label{m-1}
			\int_{\mathcal{R}_\varepsilon} u^2 \, dx \le C \, \int_{\mathcal{R}_\varepsilon} |\nabla_x u|^2 \, dx, \qquad
			\int_{\mathcal{R}^{(0)}_\varepsilon} u^2 \, dx \le C\, \varepsilon \,  \int_{\mathcal{R}_\varepsilon} |\nabla_x u|^2 \, dx,	
			\\\label{m-2}
			\varepsilon\, \int_{\Gamma^{(i,j)}_\varepsilon} u^2 \, dx \le C \bigg(\varepsilon^2 \int_{\mathcal{R}^{(i)}_\varepsilon} |\nabla_x u|^2 \, dx + \int_{\mathcal{R}^{(i)}_\varepsilon} u^2 \, dx\bigg),
			\\\label{m-3}
			\varepsilon\, \int_{\Gamma^{(i,0)}_\varepsilon} u^2 \, dx \le C \bigg(\varepsilon^2 \int_{\mathcal{R}^{(0)}_\varepsilon} |\nabla_x u|^2 \, dx + \int_{\mathcal{R}^{(0)}_\varepsilon} u^2 \, dx\bigg)
		\end{gather}   
		for any function $u \in H^1(\mathcal{R}_\varepsilon)$ whose traces on $\{\Upsilon_\varepsilon^{(i)}(\ell_i)\}_{i=1}^3$ vanish (for more detail see \cite[Sect.2]{Mel-AA-2021}). 
		
		The advection terms in the fourth line of \eqref{th-6} are estimated using the properties of the advection field $\overrightarrow{V_\varepsilon}$ (see \S~\ref{convextion-flux}) and Cauchy’s inequality with any $\lambda > 0$ $(ab \le \lambda a^2 + b^2/4\lambda),$ giving
		$$
			\bigg|\int_{0}^{\tau} \bigg(\int_{\mathcal{R}^{(0)}_\varepsilon} Y_\varepsilon\,
			\overrightarrow{V_\varepsilon} \cdot
			{\nabla Y_\varepsilon} \, dx + \sum_{i=1}^{3}\int_{\mathcal{R}^{(i)}_\varepsilon} Y_\varepsilon\,
			\overrightarrow{V_\varepsilon} \cdot
			{\nabla Y_\varepsilon} \, dy^{(i)}\bigg) dt\bigg|
		$$
		$$
			=	\bigg|\int_{0}^{\tau} \bigg(\frac12\int_{\mathcal{R}^{(0)}_\varepsilon}\mathrm{div}_x\big(Y^2_\varepsilon\, \overrightarrow{V_\varepsilon}\big)\, dx + \sum_{i=1}^{3}\int_{\mathcal{R}^{(i)}_\varepsilon}\Big(\frac12\, v^{(i)}_1 \, \partial_{y_1^{(i)}}Y_\varepsilon^2  + \varepsilon \, Y_\varepsilon \, v^{(i)}_2 \partial_{y_2^{(i)}}Y_\varepsilon\Big) dy^{(i)}\bigg) dt\bigg|
		$$
		$$
			\le \int_{0}^{\tau} \bigg(\varepsilon\, \lambda \sum_{i=1}^{3}\Big(\int_{\mathcal{R}^{(i)}_\varepsilon} |\partial_{y_2^{(i)}}Y_\varepsilon|^2  + 
			\frac{\varepsilon \, C_1}{4 \lambda} \, Y_\varepsilon^2 + C_2 Y_\varepsilon^2\Big) dy^{(i)}\bigg) dt.
		$$
		
		Collecting all contributions, including assumptions $\mathbf{A3}$ and $\mathbf{A4}$ for the functions $\Phi_0$ and $\Phi^{(i,j)},$ and choosing $\lambda$ appropriately, we deduce from \eqref{th-6} the inequality    
		\begin{multline}\label{th-8}
			\big\|Y_\varepsilon(\cdot,t)\big\|^2_{L^2(\mathcal{R}_\varepsilon)}\Big|_{t=\tau}
			+ \varepsilon \, \big\| {\nabla Y_\varepsilon}\big\|^2_{L^2(\mathcal{R}_\varepsilon\times (0, \tau))} \le C_1 \bigg(\int_{0}^{\tau} \big\|Y_\varepsilon(\cdot,t)\big\|^2_{L^2(\mathcal{R}_\varepsilon)} dt + \varepsilon^3
			\\
			+
			\varepsilon^3 \sum_{i=1}^{3} \sum_j\int_{0}^{\tau}\int_{\Gamma^{(i,j)}_\varepsilon}
			\big(\mathfrak{U}^{(i)}_\varepsilon - u^{(i)}_\varepsilon\big)^2\, d y_1^{(j)}dt
			+  \varepsilon \sum_{i=1}^{3}\int_{0}^{\tau}\int_{\Gamma^{(i,0)}_\varepsilon}\big(\mathfrak{U}^{(i)}_\varepsilon - u^{(i)}_\varepsilon\big)^2\, dl_xdt\bigg)
			\\
			\stackrel{\eqref{trace-ineq}}{\le}
			C_2 \bigg(\int_{0}^{\tau} \big\|Y_\varepsilon(\cdot,t)\big\|^2_{L^2(\mathcal{R}_\varepsilon)} dt + \varepsilon^3
			+
			(\varepsilon^3 + \varepsilon^2 |\log\varepsilon|) \sum_{i=1}^{3} \big\|\nabla_x(\mathfrak{U}^{(i)}_\varepsilon - u^{(i)}_\varepsilon)\big\|^2_{L^2(\Omega_\varepsilon^{(i)}\times(0,\tau))}\bigg)
			\\
			\stackrel{\eqref{th-4}}{\le}
			C_3 \bigg(\int_{0}^{\tau} \big\|Y_\varepsilon(\cdot,t)\big\|^2_{L^2(\mathcal{R}_\varepsilon)} dt + \varepsilon^3
			+
			\\
			\varepsilon^2 |\log\varepsilon| \Big(
			\sum_{i=1}^{3}\varepsilon^{2(\mu_i -1)} + 
			\varepsilon^2 |\log\varepsilon|^2 
			+ \sum_{i=1}^{3}\sum_{j} \|Y_\varepsilon\|^2_{L^2(\Gamma_\varepsilon^{(i,j)}\times (0,\tau))} + \varepsilon \, |\log\varepsilon| \sum_{i=1}^{3}\|Y_\varepsilon\|^2_{L^2(\Gamma_\varepsilon^{(i,0)}\times (0,\tau))}
			\Big)\bigg),
		\end{multline}
		where $\mu_i \in (1, \tfrac{\pi}{\theta_i}).$

		Using \eqref{m-2}--\eqref{m-3} and absorbing gradient terms for  sufficiently small $\varepsilon,$ we arrive  the inequality
		\begin{equation}\label{th-9}
			\big\|Y_\varepsilon(\cdot,t)\big\|^2_{L^2(\mathcal{R}_\varepsilon)}\Big|_{t=\tau}
			+ \varepsilon \, \big\| {\nabla Y_\varepsilon}\big\|^2_{L^2(\mathcal{R}_\varepsilon\times (0, \tau))} \le C_1 \bigg(\int_{0}^{\tau} \big\|Y_\varepsilon(\cdot,t)\big\|^2_{L^2(\mathcal{R}_\varepsilon)} dt + \varepsilon^3 + \sum_{i=1}^{3}\varepsilon^{2 \mu_i} |\log\varepsilon| \bigg). 	
		\end{equation}

		Applying  Gronwall's inequality in time to \eqref{th-9} leads first to the inequality
		\begin{equation}\label{th-10}
			\max_{t\in[0,T]} \|Y_\varepsilon(\cdot,t)\|^2_{L^2(\mathcal{R}_\varepsilon)} \le C_2 \Big(\varepsilon^3 +  \sum_{i=1}^{3} \varepsilon^{2 \mu_i} |\log\varepsilon|\Big)
		\end{equation}
		and then to 
		\begin{equation}\label{th-11}
			\big\| {\nabla Y_\varepsilon}\big\|^2_{L^2(\mathcal{R}_\varepsilon\times (0, T))} \le C_3
			\Big(\varepsilon^2 + \sum_{i=1}^{3} \varepsilon^{2 \mu_i - 1} |\log\varepsilon| 
			\Big),
		\end{equation}	
		which imply \eqref{main-1} and \eqref{main-2}.

		\textbf{3.} \emph{Final estimates in $\Omega_\varepsilon^{(i)}\times (0,T)$.} 
		From \eqref{th-10}--\eqref{th-11} and \eqref{m-2}--\eqref{m-3} we obtain the boundary estimates 
		\begin{equation*}
			\|Y_\varepsilon\|_{L^2(\Gamma_\varepsilon^{(1,j)}\times (0,T))} + \|Y_\varepsilon\|_{L^2(\Gamma_\varepsilon^{(1,0)}\times (0,T))} \le C_4
			\Big(\varepsilon + \sum_{i=1}^{3} \varepsilon^{\mu_i - \frac12} \sqrt{|\log\varepsilon|} 
			\Big)
		\end{equation*}
		for all admissible values of the index $j.$ 	
		Substituting them in \eqref{th-4} and \eqref{th-5+} gives 
		\begin{equation}\label{th-12}
			\max_{t\in[0,T]}\|Z_\varepsilon(\cdot,t)\|_{L^2(\Omega_\varepsilon^{(1)})} +	\|\nabla Z_\varepsilon\|_{L^2(\Omega_\varepsilon^{(1)}\times(0,T))} \le C_3 \Big(\varepsilon^{\mu_1 -1} + \varepsilon |\log\varepsilon| + \sum_{i=2}^{3} \varepsilon^{\mu_i - \frac12} \sqrt{|\log\varepsilon|}
			\Big).
		\end{equation}
		Similar considerations apply for $\mathfrak{U}^{(i)}_\varepsilon - u^{(i)}_\varepsilon,$ $i \in \{2, 3\},$  which completes the proof. 
	\end{proof}
	\begin{remark}
		The term $\varepsilon^{\mu_i - \frac12} \sqrt{|\log\varepsilon|}$ is of lower order than $\varepsilon^{\mu_i - 1}$  as $\varepsilon \to 0,$	
		so it may be omitted in \eqref{main-3} as in \eqref{th-12}. It is included only to unify the notation.	
	\end{remark}	
	
	\begin{corollary}\label{corollary-4-1}
		As a consequence of \eqref{main-1} and \eqref{main-3}, we obtain the following estimates for the differences between the solution of the original problem $\mathbb{P}_\varepsilon$ and the zeroth-order approximation:  
		\begin{equation}\label{main-1+0}
			\tfrac{1}{\sqrt{|\mathcal{R}_{\varepsilon}|}}\,\max_{t\in[0,T]} \big\|\mathfrak{A}_{\varepsilon,0}(\cdot,t) - w_\varepsilon(\cdot,t)\big\|_{L^2(\mathcal{R}_\varepsilon)} \le C_1 \Big(\varepsilon +  \sum_{i=1}^{3} \varepsilon^{\mu_i - \frac12} \sqrt{|\log\varepsilon|}\Big),
		\end{equation}
		\begin{equation}\label{main-3+0}
			\max_{t\in[0,T]}\big\|u^{(i)}_0(\cdot,t) - u^{(i)}_\varepsilon(\cdot,t)\big\|_{L^2(\Omega_\varepsilon^{(i)})} 
			\le C_2 \Big(\varepsilon^{\mu_i -1} + \varepsilon |\log\varepsilon| + \sum_{\imath =1}^{3} \varepsilon^{\mu_\imath - \frac12} \sqrt{|\log\varepsilon|}
			\Big) \quad \text{for} \ \ i\in \{1,2,3\},
		\end{equation}	
		where 
		\begin{equation}\label{first_app+0}
			\mathfrak{A}_{\varepsilon,0} =
			\left\{
			\begin{array}{ll}
				w^{(1)}_0(y^{(1)}, t) & \text{in} \  \  \mathcal{R}^{(1)}_{\varepsilon,3\ell_0,\gamma}\times[0,T], 
				\\[4pt]
				w^{(i)}_0(y^{(i)}, t) + \chi_\delta^{(i)}(y^{(i)}_1) \, \Pi_0^{(i)}\!\Big(
				\tfrac{\ell_i - y^{(i)}_1}{\varepsilon},
				\tfrac{y^{(i)}_2}{\varepsilon}, t\Big)& \text{in} \ \
				\mathcal{R}^{(i)}_{\varepsilon,3\ell_0,\gamma}\times[0,T], \ \ i\in\{2, 3\},
				\\[4pt]
				w^{(1)}_0(0, t) & \text{in} \ \  \mathcal{R}^{(0)}_{\varepsilon, \gamma}\times[0,T],
				\\[4pt]
				\chi_{\ell_0}\big(\frac{y^{(i)}_1}{\varepsilon^\gamma}\big)\, w^{(i)}_0(y^{(i)}, t) +
				\Big(1- \chi_{\ell_0}\big(\frac{y^{(i)}_1}{\varepsilon^\gamma}\big)\Big)\, w^{(1)}_0(0, t) & \text{in} \ \
				\mathcal{R}^{(i)}_{\varepsilon,2\ell_0,3\ell_0,\gamma }\times[0,T], \ \ i\in\{1,2,3\},
			\end{array}
			\right.
		\end{equation}
		$\{w^{(i)}_0\}_{i=1}^3$ solves the limit hyperbolic problem \eqref{limit_prob}, and $u^{(i)}_0$ solves the limit problem \eqref{relations+}. 
	\end{corollary}	
	
	
	\section{Conclusion and Perspectives}\label{Sect-5}
	
	\textbf{1.} 	
	In this work we developed a rigorous multiscale framework for the analysis of
	transport in planar domains with micro‑branched fractures. Beyond the simple
	geometry considered here, the method extends naturally to other configurations
	of micro‑defects. Our earlier works demonstrate that the same multiscale
	strategy applies to curvilinear thin branches \cite{Mel-Roh_AA-2024}, to
	networks of thin branches connected by small nodes of diameter
	$\mathcal{O}(\varepsilon)$ \cite{Mel-Roh_JMAA-2024}, to more general diffusion
	operators in the fracture and to time‑dependent convection fields
	\cite{Mel-Roh_JMAA-2024}, and even to diffusion–reaction–advection systems for
	multiple dissolved substances \cite{Mel-Pop-Rohde-preprint-2026}.
	
	The analysis shows that the geometry of a crack network induces singular behavior of solutions both near the junction points and within them. This requires a careful matching of the corner and outer bulk ansatzes (see \S~\ref{corner-2-4}), as well as the inner two-scaled ansatzes with the node-layer ansatz in the fracture (see \S~\ref{Par-3-1}). A key outcome  is the identification of the logarithmic
	singularity and the precise matching constant $\Lambda^{(i)}_1(t)$ linking the bulk and corner regimes.
	
	The approach also reveals the hyperbolic interface dynamics driven by strong
	advection in the fracture $\mathcal{R}_\varepsilon$. As a consequence of this
	strong advection, singular behaviour appears on the outflow branches near their
	bases, which is captured by the boundary–layer ansatz constructed in
	\S~\ref{subsec_Bound_layer}.
	
	The global approximation developed in Sect.~\ref{Sect-4}, together with the
	asymptotic estimates rigorously proven therein (Theorem~\ref{apriory_estimate}
	and Corollary~\ref{corollary-4-1}), reveals the structure of the solution and
	its gradient in all parts of the fractured medium and shows how the fracture
	geometry influences the asymptotic behavior of the model. A notable feature of
	these estimates is the explicit dependence of their orders on the corner angles
	$\theta_1,\theta_2,\theta_3$ between the edges of the graph, which originates
	from the singular behavior of the bulk solutions to problems~\eqref{relations+}
	near the corner points. For $T$-shaped fractures, however, these corner
	singularities disappear, the bulk solutions become regular, and an improvement
	in the order of the asymptotic estimates is therefore expected.

	\smallskip 
	
	\textbf{2.} The analysis can be adapted to different scaling regimes in the interface conditions \eqref{int-1} and \eqref{int-2}. However, when a problem involves additional parameters, it is necessary to modify the asymptotic scale by adjusting it precisely to account for these
	parameters (see, e.g., \cite{Mel-AA-2021,Mel-Roh_JMAA-2024},
	\cite[Sect.~6]{Mel-Roh_Non-Diff-2024}).
	If, for instance, $\alpha\in(\tfrac{3}{2},2)$ and $\beta =1,$ then the appropriate two-scaled ansatz in each thin branch $\mathcal{R}^{(i)}_\varepsilon$ takes the form
	\begin{multline}\label{reg-1}
		\mathfrak{W}^{(i)}_\varepsilon =
		w_0^{(i)}(y^{(i)}_1,t)
		+ \varepsilon^{\alpha-1} w_{\alpha-1}^{(i)}(y^{(i)}_1,t)
		+ \varepsilon\Big(
		w_1^{(i)}(y^{(i)}_1,t)
		+ z_1^{(i)}\big(y^{(i)}_1,\tfrac{y^{(i)}_2}{\varepsilon},t\big)
		\Big)
		\\
		+ \varepsilon^{2\alpha-2} w_{2\alpha-2}^{(i)}(y^{(i)}_1,t)
		+ \varepsilon^\alpha\Big(
		w_\alpha^{(i)}(y^{(i)}_1,t)
		+ z_\alpha^{(i)}\big(y^{(i)}_1,\tfrac{y^{(i)}_2}{\varepsilon},t\big)
		\Big)
		+ \varepsilon^2 z_2^{(i)}\big(y^{(i)}_1,\tfrac{y^{(i)}_2}{\varepsilon},t\big)
		\\
		+ \varepsilon^{2\alpha-1} z_{2\alpha-1}^{(i)}\big(y^{(i)}_1,\tfrac{y^{(i)}_2}{\varepsilon},t\big)
		+ \varepsilon^{\alpha+1} z_{\alpha+1}^{(i)}\big(y^{(i)}_1,\tfrac{y^{(i)}_2}{\varepsilon},t\big).
	\end{multline}
	This necessarily leads to modifications of both the node–layer and
	boundary–layer ansatzes. Interestingly, the closer $\alpha$ is to $1$, the more
	intermediate terms appear between integer powers of $\varepsilon$. For
	$\alpha\in(\tfrac{4}{3},\tfrac{3}{2})$, the asymptotic scale becomes
	$
	\varepsilon^0, \ \varepsilon^{\alpha -1}, \  \varepsilon^{2\alpha -2},  \  \varepsilon^1, \  \varepsilon^{3\alpha -3},\  \varepsilon^{\alpha}, \  \varepsilon^{2\alpha -1}, \  \varepsilon^2, \  \varepsilon^{3\alpha -2}, \ \varepsilon^{\alpha +1},\ldots,
	$
	while the order of the first two terms  remains unchanged.  The term
	$\varepsilon^{\alpha-1} w_{\alpha-1}^{(i)}$ is crucial, as it incorporates the
	effect of the boundary conditions on the lateral sides of the thin branches.
	For integer values of $\alpha$, several terms in \eqref{reg-1} become of the
	same order, and this case therefore requires a separate and more delicate
	analysis, as was carried out here for $\alpha=2$.
	
	\smallskip 
	
	\textbf{3.} 
	If $\beta\neq 1$, the asymptotic scale will depend simultaneously on
	$\beta$, and identifying this dependence is one of the main tasks in future
	asymptotic studies. A natural first step is the case $\alpha=2$ and $\beta=0$.
	Here the influence of the junction node appears already in the leading terms
	$\{w_0^{(i)}\}_{i=1}^3$ through a nonstandard Kirchhoff condition at the vertex
	(analogous to \eqref{Kirch-2}). Instead of the coupled recursive limit problems
	\eqref{relations+} and \eqref{limit_prob}, one obtains a new type of limit
	problem linking the bulk problems \eqref{relations+} with the corresponding
	hyperbolic problem on the graph via a nonstandard vertex condition. This may
	produce new qualitative effects. The main challenge is to prove well–posedness
	and regularity of the resulting limit problem.
	
	A further promising direction is the study of transport problems in
	three–dimensional geometries with thin fissures composed of a finite number of
	interconnected thin cylinders. In such settings, singular behaviour of bulk
	solutions will appear both near the edges of the corresponding graph and at its
	vertices. We hope that the methodology developed in this article will be
	helpful for these investigations.
	
	\section*{Acknowledgments}
	The first author gratefully acknowledges support from the MSCA4Ukraine grant, which made it possible to carry out this research at the University of Stuttgart. This project has received funding through the MSCA4Ukraine project (ID: 101101923), funded by the European Union.  Views and opinions expressed are however those of the author(s) only and do not 
	necessarily reflect those of the European Union, the European Research Executive Agency or the 
	MSCA4Ukraine Consortium. Neither the European Union nor the European Research Executive 
	Agency, nor the MSCA4Ukraine Consortium as a whole nor any individual member institutions of 
	the MSCA4Ukraine Consortium can be held responsible for them.
	
	All authors acknowledge  (partial) funding from the Deutsche Forschungsgemeinschaft (DFG, German Research Foundation) – Project Number 327154368 – SFB 1313.


\end{document}